\documentclass[11pt]{article}

\usepackage{graphicx}
\usepackage{amssymb}
\usepackage[english]{babel}
\usepackage[utf8]{inputenc}
\usepackage{amsmath}
\usepackage{amsthm}
\usepackage{latexsym}
\usepackage{stmaryrd}
\usepackage{color}
\usepackage{mathtools}
\usepackage{amsfonts}
\usepackage[margin=2cm]{geometry}
\usepackage[mathscr]{euscript}
\usepackage{tikz}
\usepackage[colorlinks=true,allcolors=black]{hyperref}
\usepackage{subcaption}
\usepackage{caption}
\usepackage{setspace}

\theoremstyle{plain}
\newtheorem{thm}{Theorem}
\newtheorem{proposition}{Proposition}
\newtheorem{lem}{Lemma}
\newtheorem{claim}{Claim}
\newtheorem{cor}{Corollary}

\theoremstyle{definition}
\newtheorem{definition}{Definition}

\theoremstyle{remark}
\newtheorem{rem}{Remark}
\makeatletter
\newcommand*{\rom}[1]{\expandafter\@slowromancap\romannumeral #1@}
\makeatother

\DeclareMathOperator*{\esssup}{ess\,sup}
\DeclareMathOperator*{\essinf}{ess\,inf} 

\title{Upper Semicontinuity of Index Plus Nullity for Minimal and CMC Hypersurfaces}
\author{Myles Workman \\
University College London}
\date{}

\begin{document}

\maketitle

\begin{abstract}
We consider a sequence, $\{M_{k}\}_{k \in \mathbb{N}}$, of bubble converging minimal hypersurfaces, or $H$-CMC hypersurfaces, in compact Riemannian manifolds without boundary, of dimension $4 \leq n + 1 \leq 7$. 
We prove the following upper semicontinuity of index plus nullity: 
\begin{eqnarray*}
    \limsup_{k \rightarrow \infty} ( \text{ind} \, (M_{k}) + \text{nul} \, (M_{k}) ) &\leq& \sum_{i = 1}^{l} co(m)_{i} (\text{anl-ind} \, (co (M^{i}_{\infty})) + \text{anl-nul} \, (co (M^{i}_{\infty}))) \\
    && + \sum_{j = 1}^{J} \text{ind} \, (\Sigma^{j}) + \text{nul}_{\omega_{\Sigma^{j}, R}} \, (\Sigma^{j}),
\end{eqnarray*}
for such a bubble converging sequence $M_{k} \rightarrow (\cup_{i = 1}^{l} M_{\infty}^{i}, \Sigma^{1}, \ldots, \Sigma^{J})$, where $co(m)_{i} \in \mathbb{Z}_{\geq 1}$ is a notion of multiplicity of the convergence to the connected component $M_{\infty}^{i}$, and $\Sigma^{1}, \ldots, \Sigma^{L}$ are the bubbles.
This complements the previously known lower semicontinuity of index obtained by Buzano--Sharp \cite{BS-QQ-estimates-minimal-hypersurfaces-bounded-index-and-area}, and Bourni--Sharp--Tinaglia \cite{BST-compactness-bubbling-CMC}.

The strategy of our proof is to analyse a weighted eigenvalue problem along our sequence of degenerating hypersurfaces, $\{M_{k}\}_{k \in \mathbb{N}}$. This strategy is inspired by the recent work of Da Lio--Gianocca--Rivi\`{e}re \cite{DGR-morse-index-stability}.
A key aspect of our proof is making use of a Lorentz--Sobolev inequality to study the behaviour of these weighted eigenfunctions on the neck regions along the sequence, as well as the index and nullity of our non-compact bubbles $\Sigma^{1}, \ldots, \Sigma^{J}$.
\end{abstract}

\bigskip 

Let $(N, g)$ be a compact Riemannian manifold, with no boundary, of dimension $n + 1$, and $H > 0$, be a fixed constant. 
In this paper we investigate two classes, $\mathfrak{M} (N, g)$, and $\mathfrak{C}_{H} (N, g)$.
Here, $\mathfrak{M} (N, g)$ is the class of smooth, closed, properly embedded, minimal hypersurfaces of $N$, with respect to the metric $g$, and $\mathfrak{C}_{H} (N, g)$ is the class of smooth, closed, properly embedded hypersurfaces in $N$, of constant mean curvature $H > 0$, with respect to the metric $g$. 

\bigskip 

As these hypersurfaces arise as critical points to appropriately chosen area-type functionals, a natural property to study is their Morse index (with respect to the associated functional).
For two fixed numbers $\Lambda > 0$, and $I \in \mathbb{Z}_{\geq 0}$, we define the subclasses, 
\begin{equation*}
    \begin{split}
        \mathfrak{M} (N, g, \Lambda, I) = \{ M \in \mathfrak{M} (N, g) \colon \mathcal{H}^{n}_{g} (M) \leq \Lambda, \, \text{ind} \, (M) \leq I \}, \\
        \mathfrak{C}_{H} (N, g, \Lambda, I) = \{ M \in \mathfrak{C}_{H} (N, g) \colon \mathcal{H}^{n}_{g} (M) \leq \Lambda, \, \text{ind} \, (M) \leq I \}. \\
    \end{split}
\end{equation*}
Making use of the curvature estimates for stable minimal hypersurfaces by Schoen--Simon--Yau \cite{SSY-CurvatureEstimatesMinimalHypersurfaces} and Schoen--Simon \cite{SS-RegularityStableMinimalHypersurfaces} (see also the recent proof by Bellettini \cite{B-SSYandSSviaDeGiorgi}), for $2 \leq n \leq 6$, compactness properties have been proven for $\mathfrak{M} (N, g, \Lambda, I)$ by Sharp \cite{Sharp-minimal-hypersurface-with-bounded-index}, and for $\mathfrak{C}_{H} (N, g, \Lambda, I)$ by Bourni--Sharp--Tinaglia \cite{BST-compactness-bubbling-CMC}.
In dimension $n + 1 = 3$, it is worth noting that various other compactness results have been shown for minimal surfaces by Choi--Schoen \cite{CS-MinimalEmbeddingsPositiveRicciCurvature}, Anderson \cite{A-CurvatureEstimatesMinimalSurfaces3Manifolds}, Ros \cite{R-CompactnessMinimalSurfacesFiniteTotalCurvature}, and White \cite{W-CompactnessMinimalSurfaces}, and for $H$-CMC surfaces by Sun \cite{S-CompactnessCMC3-Manifold}.
We note that $\mathfrak{M} (N, g, \Lambda, I)$ is sequentially compact (under the correct notion of convergence), whereas for $\mathfrak{C}_{H} (N, g, \Lambda, I)$, we must expand our class to quasi-embedded, $H$-CMC hypersurfaces (see Definition \ref{def: quasi-embedded hypersurfaces}).
We denote this enlarged class by $\overline{\mathfrak{C}_{H}} (N, g, \Lambda, I)$. 
We briefly describe this notion of convergence, with full details described in point 1 of Definition \ref{def: bubble convergence}. 
Consider a sequence $\{M_{k}\} \subset \mathfrak{M} (N, g, \Lambda, I)$ (resp. $\mathfrak{C}_{H} (N, g, \Lambda, I)$), then after potentially taking a subsequence and renumerating, there is a smooth, closed, embedded minimal hypersurface (resp. $H$-CMC quasi-embedded hypersurface) $M_{\infty}$, and a finite set of points $\mathcal{I} \subset M_{\infty}$, where $|\mathcal{I}| \leq I$, such that on compact subsets of $N \setminus \mathcal{I}$, $M_{k}$ will converge to $M_{\infty}$, smoothly and locally graphically, with integer multiplicity potentially greater than $1$. 
It is then shown that $M_{\infty} \in \mathfrak{M} (N, g, \Lambda, I)$ (resp. $\overline{\mathfrak{C}_{H}} (N, g, \Lambda, I)$).
The set of points $\mathcal{I} \subset M_{\infty}$, is defined by the condition that for each $y \in \mathcal{I}$, there exists a sequence points $\{x_{k}^{y} \in M_{k}\}_{k \in \mathbb{N}}$, such that $x_{k}^{y} \rightarrow y$, and the curvature $|A_{M_{k}} (x_{k}^{y})|$, blows up as $k \rightarrow \infty$.
Thus we call $\mathcal{I}$ the singular set of the convergence. 

\bigskip 

In a bid to understand the formation of such singularities, a bubble analysis was carried out by Chodosh--Ketover--Maximo \cite{CKM-MinimalHypersurfacesBoundedIndex}, Buzano--Sharp \cite{BS-QQ-estimates-minimal-hypersurfaces-bounded-index-and-area}, and Bourni--Sharp--Tinaglia \cite{BST-compactness-bubbling-CMC}. 
Zooming in at appropriate rates, along particular sequences of points converging onto $\mathcal{I}$, yields a complete, embedded, non-planar, minimal hypersurface in $\mathbb{R}^{n + 1}$, of finite index, with Euclidean volume growth at infinity. 
These minimal hypersurfaces in $\mathbb{R}^{n + 1}$ are referred to as the `bubbles', and they are the singularity models at the singular points of the convergence.
The hypersurface $M_{\infty}$ is referred to as the `base'.
This terminology is borrowed from other non-linear geometric problems.
See Figure \ref{fig: picture of bubble converging sequence} for a heuristic picture.
In the case of $n = 2$, a bubble analysis has been carried out by Ros \cite{R-CompactnessMinimalSurfacesFiniteTotalCurvature}, in $\mathbb{R}^{3}$ with the standard Euclidean metric, assuming uniform bounds on the total curvature instead of the Morse index, and by White \cite{W-CompactnessMinimalSurfaces}, in general $3$-manifolds, assuming uniform bounds on genus instead of Morse index.

\begin{center}

\begin{figure}

    \tikzset{every picture/.style={line width=0.75pt}} 

    \begin{tikzpicture}[x=0.75pt,y=0.75pt,yscale=-1,xscale=1]
    
    \draw   (52,109.67) .. controls (52,59.59) and (92.59,19) .. (142.67,19) .. controls (192.74,19) and (233.33,59.59) .. (233.33,109.67) .. controls (233.33,159.74) and (192.74,200.33) .. (142.67,200.33) .. controls (92.59,200.33) and (52,159.74) .. (52,109.67) -- cycle ;
    \draw   (261,108.67) .. controls (261,58.59) and (301.59,18) .. (351.67,18) .. controls (401.74,18) and (442.33,58.59) .. (442.33,108.67) .. controls (442.33,158.74) and (401.74,199.33) .. (351.67,199.33) .. controls (301.59,199.33) and (261,158.74) .. (261,108.67) -- cycle ;
    \draw   (470,113.67) .. controls (470,63.59) and (510.59,23) .. (560.67,23) .. controls (610.74,23) and (651.33,63.59) .. (651.33,113.67) .. controls (651.33,163.74) and (610.74,204.33) .. (560.67,204.33) .. controls (510.59,204.33) and (470,163.74) .. (470,113.67) -- cycle ;
    \draw  [dash pattern={on 0.84pt off 2.51pt}] (266,321.67) .. controls (266,271.59) and (306.59,231) .. (356.67,231) .. controls (406.74,231) and (447.33,271.59) .. (447.33,321.67) .. controls (447.33,371.74) and (406.74,412.33) .. (356.67,412.33) .. controls (306.59,412.33) and (266,371.74) .. (266,321.67) -- cycle ;
    \draw  [dash pattern={on 0.84pt off 2.51pt}] (52,320.67) .. controls (52,270.59) and (92.59,230) .. (142.67,230) .. controls (192.74,230) and (233.33,270.59) .. (233.33,320.67) .. controls (233.33,370.74) and (192.74,411.33) .. (142.67,411.33) .. controls (92.59,411.33) and (52,370.74) .. (52,320.67) -- cycle ;
    \draw  [dash pattern={on 0.84pt off 2.51pt}] (473,319.67) .. controls (473,269.59) and (513.59,229) .. (563.67,229) .. controls (613.74,229) and (654.33,269.59) .. (654.33,319.67) .. controls (654.33,369.74) and (613.74,410.33) .. (563.67,410.33) .. controls (513.59,410.33) and (473,369.74) .. (473,319.67) -- cycle ;
    \draw    (60.33,71.33) .. controls (145.33,94.33) and (216.33,152.33) .. (52,109.67) ;
    \draw    (203.33,174.33) .. controls (186.33,154.33) and (99.33,97.33) .. (231.33,130.33) ;
    \draw    (148.33,122.33) .. controls (148.33,122.33) and (153.33,131.33) .. (160.33,125.33) ;
    \draw    (330.33,100.33) .. controls (334.33,100.33) and (375.33,123.33) .. (327.33,109.33) ;
    \draw    (372.33,124.33) .. controls (347.33,110.33) and (361.33,111.33) .. (373.69,114.33) ;
    \draw    (350.33,114.33) .. controls (350.33,112.33) and (354.33,119.33) .. (359.33,114.33) ;
    \draw    (62.33,277.67) .. controls (200.33,324.67) and (142.33,332.67) .. (60.33,360.67) ;
    \draw    (221.33,365.67) .. controls (131.33,334.67) and (156.67,299) .. (227.33,287.67) ;
    \draw    (271.33,289.67) .. controls (384.33,289.67) and (375.33,354.67) .. (271.33,353.67) ;
    \draw    (443.33,350.67) .. controls (339.33,354.67) and (359.67,293) .. (443.33,292.67) ;
    \draw    (477.33,292.67) .. controls (578.33,295.67) and (574.33,349.67) .. (478.33,351.67) ;
    \draw    (648.33,352.33) .. controls (562.33,352.33) and (559.33,293.67) .. (650.33,290.67) ;
    \draw    (147.33,324.33) .. controls (151.33,329.33) and (164.33,327.33) .. (163.33,326.33) ;
    \draw    (351.33,327.33) .. controls (355.33,332.33) and (372.33,333.33) .. (373.33,327.33) ;
    \draw    (551.33,325.33) .. controls (561.33,334.33) and (579.33,332.33) .. (583.33,325.33) ;
    \draw  [dash pattern={on 0.84pt off 2.51pt}] (112.67,125.83) .. controls (112.67,103.65) and (130.65,85.67) .. (152.83,85.67) .. controls (175.02,85.67) and (193,103.65) .. (193,125.83) .. controls (193,148.02) and (175.02,166) .. (152.83,166) .. controls (130.65,166) and (112.67,148.02) .. (112.67,125.83) -- cycle ;
    \draw  [dash pattern={on 0.84pt off 2.51pt}] (326.98,114.33) .. controls (326.98,101.44) and (337.44,90.98) .. (350.33,90.98) .. controls (363.23,90.98) and (373.69,101.44) .. (373.69,114.33) .. controls (373.69,127.23) and (363.23,137.69) .. (350.33,137.69) .. controls (337.44,137.69) and (326.98,127.23) .. (326.98,114.33) -- cycle ;
    \draw  [line width=6] [line join = round][line cap = round] (560.33,112.67) .. controls (560.33,112.67) and (560.33,112.67) .. (560.33,112.67) ;
    \draw    (147.33,165.67) -- (147.33,225.67) ;
    \draw [shift={(147.33,227.67)}, rotate = 270] [color={rgb, 255:red, 0; green, 0; blue, 0 }  ][line width=0.75]    (10.93,-3.29) .. controls (6.95,-1.4) and (3.31,-0.3) .. (0,0) .. controls (3.31,0.3) and (6.95,1.4) .. (10.93,3.29)   ;
    \draw    (350.33,137.69) -- (351.31,222.67) ;
    \draw [shift={(351.33,224.67)}, rotate = 269.34] [color={rgb, 255:red, 0; green, 0; blue, 0 }  ][line width=0.75]    (10.93,-3.29) .. controls (6.95,-1.4) and (3.31,-0.3) .. (0,0) .. controls (3.31,0.3) and (6.95,1.4) .. (10.93,3.29)   ;
    \draw    (560.67,113.67) -- (562.3,219.67) ;
    \draw [shift={(562.33,221.67)}, rotate = 269.12] [color={rgb, 255:red, 0; green, 0; blue, 0 }  ][line width=0.75]    (10.93,-3.29) .. controls (6.95,-1.4) and (3.31,-0.3) .. (0,0) .. controls (3.31,0.3) and (6.95,1.4) .. (10.93,3.29)   ;
    \draw    (261,108.67) .. controls (316,93.67) and (294.33,99.33) .. (327.33,109.33) ;
    \draw    (261.33,93.33) .. controls (304.33,76.33) and (316.33,95.33) .. (330.33,100.33) ;
    \draw    (372.33,124.33) .. controls (401.33,136.33) and (412.33,119.33) .. (439.33,137.33) ;
    \draw    (373.69,114.33) .. controls (411.33,126.33) and (402.67,111.67) .. (439.33,125.33) ;
    \draw    (471.33,102.33) .. controls (512.33,84.33) and (520,90) .. (560.67,113.67) ;
    \draw    (560.67,113.67) .. controls (590.33,139.33) and (611.33,99.33) .. (648.33,135.33) ;
    
    \draw (607,129.4) node [anchor=north west][inner sep=0.75pt]    {$M_{\infty}$};
    \draw (571,97.4) node [anchor=north west][inner sep=0.75pt]    {$y$};
    \draw (609,358.73) node [anchor=north west][inner sep=0.75pt]    {$\Sigma$};
    
    \end{tikzpicture}

\caption{The top row depicts a local picture of a converging sequence about a point $y \in \mathcal{I}$, which converges with multiplicity 2 on the base $M_{\infty}$. 
The second row then depicts a dilation of the dotted circles in the top row, which when we take a limit, as seen in the last column, allows us to see a catenoid as the bubble $\Sigma$.}\label{fig: picture of bubble converging sequence}

\end{figure}
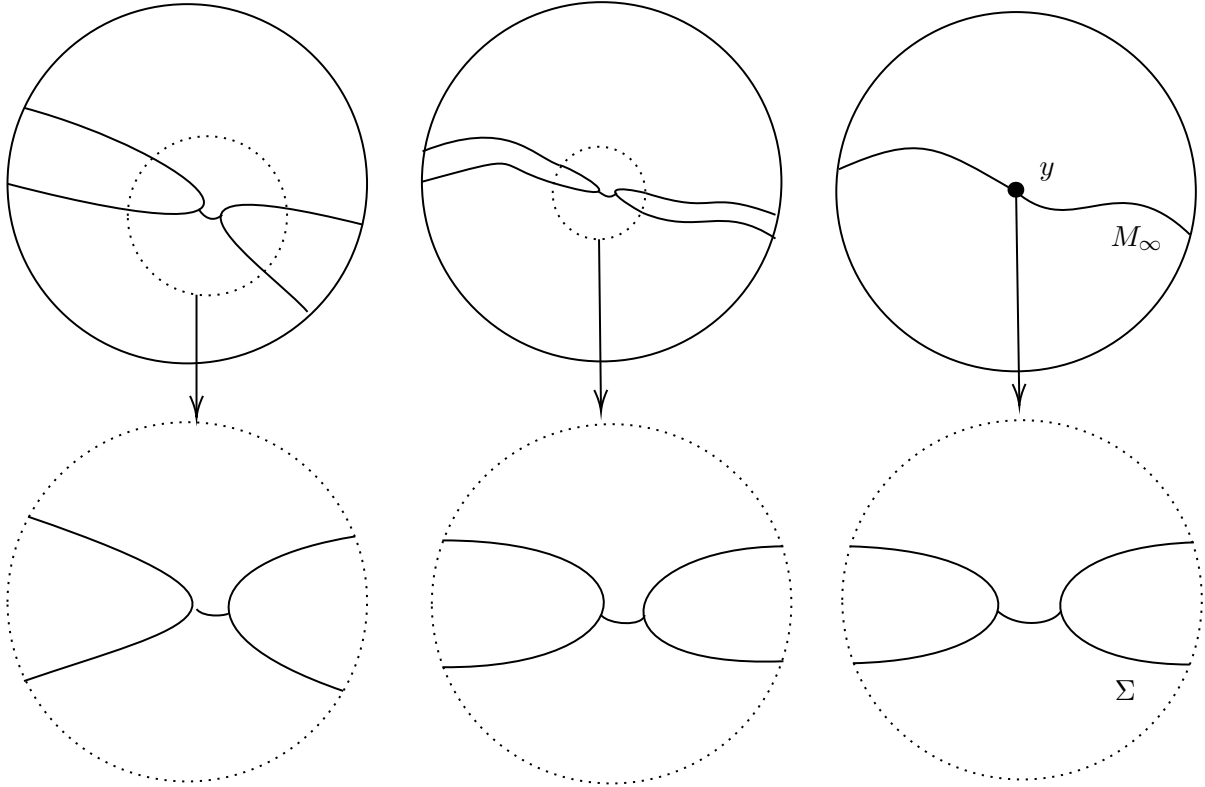

\end{center}

One may be interested in certain information about the hypersurfaces along these sequences, for example; genus (\cite{CKM-MinimalHypersurfacesBoundedIndex}), index and total curvature (\cite{BS-QQ-estimates-minimal-hypersurfaces-bounded-index-and-area, BST-compactness-bubbling-CMC}).
Understanding the formation of these singularities through this bubble analysis allows us to track this information along the sequence, and how it behaves when taking the limit $M_{k} \rightarrow M_{\infty}$. 
For example, if we have our sequence $\{M_{k}\}$ as above, and we know all our `bubbles' are given by $\Sigma^{1}, \ldots, \Sigma^{J} \subset \mathbb{R}^{n + 1}$, then we say that $M_{k} \rightarrow (M_{\infty}, \Sigma^{1}, \ldots, \Sigma^{J})$ `bubble converges' (see Definition \ref{def: bubble convergence} for a detailed definition), and \cite{BS-QQ-estimates-minimal-hypersurfaces-bounded-index-and-area,BST-compactness-bubbling-CMC},
\begin{equation}\label{eqn: lower semicontinuity of index}
    \text{ind} \, (M_{\infty}) + \sum_{j = 1}^{J} \text{ind} \, (\Sigma^{j}) \leq \liminf_{k \rightarrow \infty} \, \text{ind} \, (M_{k}).
\end{equation}
This inequality gives a quantitative way of accounting for some of the index lost when taking the limit $M_{k} \rightarrow M_{\infty}$. 
In this work we are interested in proving an opposite inequality, which will give a finer analysis and description of the index along such a converging sequence, and shows that in certain situations, this bubble analysis will account for all the index in the limit.

\bigskip

Our main result is Theorem \ref{thm: main theorem intro} below, however we first illustrate the conclusions in a simplified setting with the following special case. 
We delay the statement of Theorem \ref{thm: main theorem intro} until the end of this section.

\bigskip 

\begin{cor}\label{cor: theorem for two-sided minimal case}
    Consider a compact Riemannian manifold $(N, g)$, without boundary and of dimension $n + 1$, $3 \leq n \leq 6$. 
    Let $\{M_{k}\}$ be a sequence of smooth, closed, embedded minimal hypersurfaces of $(N, g)$, such that $M_{k} \rightarrow (M_{\infty}, \Sigma^{1}, \ldots, \Sigma^{J})$ bubble converges as in Definition \ref{def: bubble convergence}, with $M_{\infty} \subset N$ being a smooth, connected, two-sided, closed, embedded minimal hypersurface.
    Then:
    \begin{eqnarray*}
        \limsup_{k \rightarrow \infty} ( \text{ind} \, (M_{k}) + \text{nul} \, (M_{k}) ) &\leq& m (\text{ind} \, (M_{\infty}) + \text{nul} \, (M_{\infty})) \\
        && + \sum_{j = 1}^{J} \text{ind} \, (\Sigma^{j}) + \text{nul}_{\omega_{\Sigma^{j}, R}} \, (\Sigma^{j}),
    \end{eqnarray*}
    where $m \in \mathbb{Z}_{\geq 1}$ is the multiplicity of the convergence onto $M_{\infty}$. 
    Here, 
    \begin{equation*}
        \text{ind} \, (\Sigma^{j}) = \lim_{S \rightarrow \infty} \text{ind}_{B_{S}^{n + 1} (0)} \, (\Sigma^{j}),
    \end{equation*}
    and $R$ may be chosen to be any finite positive real number greater than some $R_{0} = R_{0} (\Sigma^{1}, \ldots, \Sigma^{J}) \in [1, \infty)$, and
    \begin{equation*}
        \omega_{\Sigma^{j}, R} (x) = \begin{cases}
            R^{-2}, \, x \in B_{R}^{n + 1} (0) \cap \Sigma^{j}, \\
            |x|^{-2}, \, x \in \Sigma^{j} \setminus B_{R}^{n + 1} (0).
        \end{cases}
    \end{equation*}
    and, 
    \begin{equation*}
        \text{nul}_{\omega_{\Sigma^{j}, R}} \, (\Sigma^{j}) = \text{dim} \, \{ f \in C^{\infty} (\Sigma^{j}) \colon \Delta f + |A_{\Sigma^{j}}|^{2} f = 0, \, f^{2} \, \omega_{\Sigma^{j}, R} \in L^{1} (\Sigma^{j}), \, |\nabla f|^{2} \in L^{1} (\Sigma^{j}) \}.
    \end{equation*}
\end{cor}

The inequality in Corollary \ref{cor: theorem for two-sided minimal case} can be further strengthened by noting that in this situation if $m \geq 2$, then $\text{ind} \, (M_{\infty}) = 0$ (\cite[Claim 6]{Sharp-minimal-hypersurface-with-bounded-index}), and $\text{nul} \, (M_{\infty}) = 1$ (as the first eigenvalue of the stability operator will be simple).
Moreover, Theorem \ref{thm: main theorem intro} (and hence Corollary \ref{cor: theorem for two-sided minimal case}) also hold for $n \geq 7$, under the additional assumption that the bubbles $\Sigma^{1}, \ldots, \Sigma^{J}$, have finite total curvature.

\bigskip 

Results on the lower semicontinuity of index along converging sequences (\ref{eqn: lower semicontinuity of index}), are common in the literature. 
For certain classes of minimal hypersurfaces see Sharp \cite{Sharp-minimal-hypersurface-with-bounded-index}, Buzano--Sharp \cite{BS-QQ-estimates-minimal-hypersurfaces-bounded-index-and-area}, Ambrozio--Carlotto--Sharp \cite{ACS-minimal-hypersurface-bounded-pth-eigenvalue} and Ambrozio--Buzano--Carlotto--Sharp \cite{ABCS-Free-Boundary-Bubble-Analysis}, and for certain classes of CMC hypersurfaces see Bourni--Sharp--Tinaglia \cite{BST-compactness-bubbling-CMC}.
In the setting of Allen--Cahn solutions see Le \cite{L-SecondInnerVariationOfAllenCahn}, Hiesmayr \cite{Fritz-index-paper}, Gaspar \cite{Gaspar-index} and Mantoulidis \cite{Mantoulidis_2022}.
For the setting of of bubble converging harmonic maps see Moore--Ream \cite[Theorem 6.1]{MR-MinimalTwoSpheresLowIndex}, and Hirsch--Lamm \cite[Theorem 1.1]{HL-morse-index-estimates}.

\bigskip

The opposite upper semicontinuity inequality (Theorem \ref{thm: main theorem intro}) is more intricate. 
We recall a few examples of such results from the literature. 
When convergence happens with multiplicity one for sequences of critical points of the Allen--Cahn functional, upper semicontinuity of the index plus nullity has been established by Chodosh--Mantoulidis \cite[Theorem 1.9]{CM-Allen-Cahn-in-3-manifolds} and Mantoulidis \cite[Theorem 1 (c)]{Mantoulidis_2022}.
In the case of bubble converging harmonic maps such an inequality was first established by Yin \cite{Y-GeneralisedNeckAnalysis,Y-HigherOrderNeckAnalysis}, and then by Da Lio--Gianocca--Rivi\`{e}re \cite{DGR-morse-index-stability} and Hirsch--Lamm \cite{HL-morse-index-estimates}. 
Note that in \cite{DGR-morse-index-stability} and \cite[Section 6]{HL-morse-index-estimates} the proofs are for a bubble converging sequence of critical points for a general class of conformally invariant lagrangians (fixed along the sequence).
The method of Da Lio--Gianocca--Rivi\`{e}re, has also been extended to prove upper semicontinuity of index plus nullity in the setting of Willmore immersions by Michelat--Rivi\`{e}re \cite{MR-IndexStabilityWillmoreOne}, and to biharmonic maps by Michelat \cite{M-IndexStabilityBiharmonicMaps}. 

\bigskip 

When combined with the lower semicontinuity of index, the inequality in Theorem \ref{thm: main theorem intro} shows that in the case of the limiting hypersurface being two-sided and minimal (as is the case of Corollary \ref{cor: theorem for two-sided minimal case}), the index along the sequence can be \emph{fully} accounted for in the limit.
Thus we should view Theorem \ref{thm: main theorem intro} as saying that we cannot lose index to the neck, or index cannot merely just disappear in the bubble convergence of Chodosh--Ketover--Maximo \cite{CKM-MinimalHypersurfacesBoundedIndex} and Buzano--Sharp \cite{BS-QQ-estimates-minimal-hypersurfaces-bounded-index-and-area} for minimal hypersurfaces (in dimensions $3 \leq n \leq 6)$. 

\bigskip 

In order to conclude that the inequality in Theorem \ref{thm: main theorem intro} is non-trivial, we must show that for each bubble, $\Sigma^{j}$, of finite index, $\text{nul}_{\omega_{\Sigma^{j}, R}} \, (\Sigma^{j}) < + \infty$. 
This is shown in Proposition \ref{prop: finiteness of nullity of Sigma}. 
Proposition \ref{prop: finiteness of nullity of Sigma} also has the following Corollary which may be of interest to some readers (and may be compared with \cite[Corollary 2]{Li-Yau-SchrodingerEquationandEigenvalueProblem}). 

\begin{cor}\label{cor: finiteness of index of minimal immersions}
    Let $\Sigma$ be a complete, connected, $n$-dimensional manifold, $n \geq 3$, and $\iota \colon \Sigma \rightarrow \mathbb{R}^{n + 1}$ be a two-sided, proper, minimal immersion, with finite total curvature, 
    \begin{equation*}
        \int_{\Sigma} |A_{\Sigma}|^{n} < + \infty,
    \end{equation*}
    and Euclidean volume growth at infinity,
    \begin{equation*}
        \limsup_{r \rightarrow \infty} \frac{\mathcal{H}^{n} (\iota (\Sigma) \cap B_{r}^{n + 1} (0))}{r^{n}} < + \infty.
    \end{equation*} 
    Then: 
    \begin{equation*}
        \text{anl-nul} \, (\Sigma) \coloneqq \text{dim} \, \{ f \in W^{1,2} (\Sigma) \colon \Delta f + |A_{\Sigma}|^{2} f = 0\} < + \infty.
    \end{equation*}
\end{cor}

We note that as $\Sigma$ is not compact, Corollary \ref{cor: finiteness of index of minimal immersions} does not follow from analysing the spectrum of a compact operator.

\bigskip

We briefly remark on the strategy of the proof for Theorem \ref{thm: main theorem intro}, which is close to the strategy of Da Lio--Gianocca--Rivi\`{e}re \cite{DGR-morse-index-stability}.
We prove Theorem \ref{thm: main theorem intro} by reframing the problem in terms of a weighted eigenvalue problem. 
The weight is specifically chosen so that sequences of normalised weighted eigenfunctions $\{f_{k}\}$, along the sequence $\{M_{k}\}$, with non-positive weighted eigenvalues, exhibit good convergence on the base $M_{\infty}$, and the bubbles, $\Sigma^{1}, \ldots, \Sigma^{J}$.
The key steps for the proof are showing the equivalence of the weighted and unweighted eigenvalue problems (Section \ref{app: equivalence of weighted and unweighted Espaces}), the convergence on the base $M_{\infty}$ (Section \ref{sec: convergence on the base}), and the convergence on the bubbles $\Sigma^{1}, \ldots, \Sigma^{J}$ (Section \ref{sec: convergence on the bubble}), along with a Lorentz--Sobolev inequality on the neck, which shows that the normalised weighted eigenfunctions cannot concentrate on the neck (Section \ref{sec: Strict Stability of the neck}). 

\bigskip 

Due to the different settings, there are several key differences between our work and that of \cite{DGR-morse-index-stability}. 
One such difference is our use of a Lorentz--Sobolev inequality to deduce strict stability on the neck.
Another major difference is that in our setting the bubbles are non-compact. 
This poses complications in the theory of the elliptic operator on the bubble. 
In particular its spectrum may not be discrete, and thus effectively analysing the index and nullity of these bubbles is subtle.

\bigskip 

It is worth pointing out that the method used in \cite{DGR-morse-index-stability}, and in Theorem \ref{thm: main theorem intro}, is rather general. 
In the proof of Theorem \ref{thm: main theorem intro}, only a few aspects rely specifically on the mean curvature assumptions of the submanifolds. 
Thus it is plausible that the ideas and techniques could be applied to a large range of problems in which one wishes to study how an elliptic PDE behaves along a sequence of (sub)manifolds which `bubble converge' in an appropriate sense.

\bigskip 

\begin{thm}\label{thm: main theorem intro}
    For a compact Riemannian manifold $(N, g)$ without boundary, of dimension $n + 1$, $3 \leq n \leq 6$, if we have a sequence $\{ M_{k} \} \subset \mathfrak{M} (N, g)$ ($\{ M_{k} \} \subset \mathfrak{C}_{H} (N, g)$), such that $M_{k} \rightarrow (M_{\infty}, \Sigma^{1}, \ldots, \Sigma^{J})$ bubble converges as in Definition \ref{def: bubble convergence}, with $M_{\infty} = \cup_{i = 1}^{l} M^{i}$, where each $M^{i}$ is a closed minimal hypersurface (resp. closed, quasi-embedded $H$-CMC hypersurface, with $co(M^{i})$ connected) and $\theta_{|M^{i}} = m_{i} \in \mathbb{Z}_{\geq 1}$ ($\theta^{i} = m_{i} \in \mathbb{Z}_{\geq 1}$), then
    \begin{eqnarray*}
        \limsup_{k \rightarrow \infty} ( \text{ind} \, (M_{k}) + \text{nul} \, (M_{k}) ) &\leq& \sum_{i = 1}^{l} co(m)_{i} (\text{anl-ind} \, (co (M^{i}_{\infty})) + \text{anl-nul} \, (co (M^{i}_{\infty}))) \\
        && + \sum_{j = 1}^{J} \text{ind} \, (\Sigma^{j}) + \text{nul}_{\omega_{\Sigma^{j}, R}} \, (\Sigma^{j}),
    \end{eqnarray*}
    where for each $i = 1, \ldots, l$, $co (m)_{i} \in \mathbb{Z}_{\geq 1}$, is such that $co (m)_{i} \leq m_{i}$ if $M^{i}$ is one-sided, and $co (m)_{i} = m_{i}$ if $M^{i}$ is two-sided.
    Here, 
    \begin{equation*}
        \text{ind} \, (\Sigma^{j}) = \lim_{S \rightarrow \infty} \text{ind} \, (\Sigma^{j} \cap B_{S}^{n + 1} (0)),
    \end{equation*}
    and $R$ may be chosen to be any finite positive real number greater than some $R_{0} = R_{0} (\Sigma^{1}, \ldots, \Sigma^{J}) \in [1, \infty)$, and
    \begin{equation*}
        \omega_{\Sigma^{j}, R} (x) = \begin{cases}
            R^{-2}, \, x \in B_{R}^{n + 1} (0) \cap \Sigma^{j}, \\
            |x|^{-2}, \, x \in \Sigma^{j} \setminus B_{R}^{n + 1} (0).
        \end{cases}
    \end{equation*}
    and, 
    \begin{equation*}
        \text{nul}_{\omega_{\Sigma^{j}, R}} \, (\Sigma^{j}) = \text{dim} \, \{ f \in C^{\infty} (\Sigma^{j}) \colon \Delta f + |A_{\Sigma^{j}}|^{2} f = 0, \, f^{2} \, \omega_{\Sigma^{j}, R} \in L^{1} (\Sigma^{j}), \, |\nabla f|^{2} \in L^{1} (\Sigma^{j}) \}.
    \end{equation*}
\end{thm}

\bigskip 

The exact method of proof we employ does not extend to the case of $n = 2$ ($2$ dimensional surfaces in $3$-manifolds). 
Two key reasons are the choice of weight (Remark \ref{rem: choice of weight incorrect for n equals 2}), and the criticality of the Lorentz--Sobolev inequality (Proposition \ref{prop: Lorentz Sobolev inequality on R n}) for $n = p = 2$.

\bigskip

We take a moment to comment on the terms $\text{anl-ind} \, (M_{\infty}^{i})$ and $\text{anl-nul} \, (M_{\infty}^{i})$, that appear in the statement of Theorem \ref{thm: main theorem intro}.
These terms respectively stand for the \emph{analytic index} and \emph{analytic nullity} of $M_{\infty}^{i}$. 
This refers to the index and nullity of the stability operator acting on the function space $C^{\infty} (co (M_{\infty}^{i}))$, where $co (M_{\infty}^{i})$ is a connected component of the oriented double cover of $M_{\infty}^{i} \subset N$. 
We explain the reasoning behind this with the following example, which is also demonstrated in Figure \ref{fig: touching spheres}. 
Consider a unit hypersphere in $\mathbb{R}^{n + 1}$ (this is a CMC hypersurface), then the function $f = 1$ on the hypersphere, is an eigenfunction of the stability operator, with negative eigenvalue, and corresponds to shrinking the hypersphere. 
Now consider a sequence of two disjoint unit hyperspheres in $\mathbb{R}^{n + 1}$, such that in the limit they touch at a point. 
In order to account for these eigenfunctions in the limit, we must allow for variations that act on the hyperspheres independently, even at the touching point. 
Thus we view the hyperspheres as immersions, and allow variations which `shrink' the hyperspheres separately. 
This type of variation cannot arise through an ambient vector field due to the behaviour at the touching point.
Thus in general, the analytic index and analytical nullity of $M_{\infty}$, will not be equivalent to the Morse index and nullity of $M_{\infty}$, which is customarily defined through ambient vector fields.
See Section \ref{subsec: stability operator, index and nullity} for further details. 

\begin{center}
    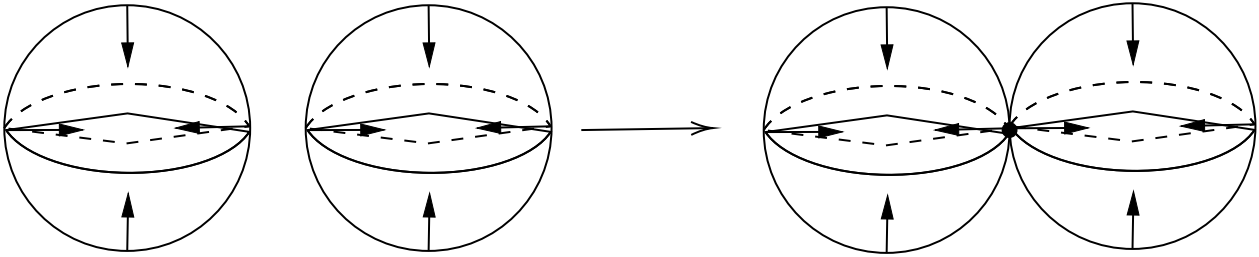
\begin{figure}[hbt!]
\tikzset{every picture/.style={line width=0.75pt}} 

\begin{tikzpicture}[x=0.75pt,y=0.75pt,yscale=-1,xscale=1]

\draw   (14,119) .. controls (14,84.76) and (41.76,57) .. (76,57) .. controls (110.24,57) and (138,84.76) .. (138,119) .. controls (138,153.24) and (110.24,181) .. (76,181) .. controls (41.76,181) and (14,153.24) .. (14,119) -- cycle ;
\draw  [draw opacity=0] (137.56,120.91) .. controls (129.52,133.2) and (104.86,141.94) .. (75.8,141.59) .. controls (47.19,141.24) and (23.07,132.2) .. (14.93,120.02) -- (76.16,111.59) -- cycle ; \draw   (137.56,120.91) .. controls (129.52,133.2) and (104.86,141.94) .. (75.8,141.59) .. controls (47.19,141.24) and (23.07,132.2) .. (14.93,120.02) ;  
\draw  [draw opacity=0][dash pattern={on 4.5pt off 4.5pt}] (14,119) .. controls (21.73,106.51) and (46.16,97.15) .. (75.22,96.75) .. controls (104.88,96.36) and (129.97,105.43) .. (137.5,118.11) -- (75.62,126.75) -- cycle ; \draw  [dash pattern={on 4.5pt off 4.5pt}] (14,119) .. controls (21.73,106.51) and (46.16,97.15) .. (75.22,96.75) .. controls (104.88,96.36) and (129.97,105.43) .. (137.5,118.11) ;  
\draw    (137.5,118.11) -- (102.33,118.95) ;
\draw [shift={(100.33,119)}, rotate = 358.63] [fill={rgb, 255:red, 0; green, 0; blue, 0 }  ][line width=0.08]  [draw opacity=0] (12,-3) -- (0,0) -- (12,3) -- cycle    ;
\draw    (14.93,120.02) -- (51.76,119.92) ;
\draw [shift={(53.76,119.91)}, rotate = 179.84] [fill={rgb, 255:red, 0; green, 0; blue, 0 }  ][line width=0.08]  [draw opacity=0] (12,-3) -- (0,0) -- (12,3) -- cycle    ;
\draw    (76,57) -- (76.31,86) ;
\draw [shift={(76.33,88)}, rotate = 269.38] [fill={rgb, 255:red, 0; green, 0; blue, 0 }  ][line width=0.08]  [draw opacity=0] (12,-3) -- (0,0) -- (12,3) -- cycle    ;
\draw    (76,181) -- (76.47,153.89) ;
\draw [shift={(76.5,151.89)}, rotate = 90.99] [fill={rgb, 255:red, 0; green, 0; blue, 0 }  ][line width=0.08]  [draw opacity=0] (12,-3) -- (0,0) -- (12,3) -- cycle    ;
\draw   (166,119) .. controls (166,84.76) and (193.76,57) .. (228,57) .. controls (262.24,57) and (290,84.76) .. (290,119) .. controls (290,153.24) and (262.24,181) .. (228,181) .. controls (193.76,181) and (166,153.24) .. (166,119) -- cycle ;
\draw  [draw opacity=0] (289.56,120.91) .. controls (281.52,133.2) and (256.86,141.94) .. (227.8,141.59) .. controls (199.19,141.24) and (175.07,132.2) .. (166.93,120.02) -- (228.16,111.59) -- cycle ; \draw   (289.56,120.91) .. controls (281.52,133.2) and (256.86,141.94) .. (227.8,141.59) .. controls (199.19,141.24) and (175.07,132.2) .. (166.93,120.02) ;  
\draw  [draw opacity=0][dash pattern={on 4.5pt off 4.5pt}] (166,119) .. controls (173.73,106.51) and (198.16,97.15) .. (227.22,96.75) .. controls (256.88,96.36) and (281.97,105.43) .. (289.5,118.11) -- (227.62,126.75) -- cycle ; \draw  [dash pattern={on 4.5pt off 4.5pt}] (166,119) .. controls (173.73,106.51) and (198.16,97.15) .. (227.22,96.75) .. controls (256.88,96.36) and (281.97,105.43) .. (289.5,118.11) ;  
\draw    (289.5,118.11) -- (254.33,118.95) ;
\draw [shift={(252.33,119)}, rotate = 358.63] [fill={rgb, 255:red, 0; green, 0; blue, 0 }  ][line width=0.08]  [draw opacity=0] (12,-3) -- (0,0) -- (12,3) -- cycle    ;
\draw    (166.93,120.02) -- (203.76,119.92) ;
\draw [shift={(205.76,119.91)}, rotate = 179.84] [fill={rgb, 255:red, 0; green, 0; blue, 0 }  ][line width=0.08]  [draw opacity=0] (12,-3) -- (0,0) -- (12,3) -- cycle    ;
\draw    (228,57) -- (228.31,86) ;
\draw [shift={(228.33,88)}, rotate = 269.38] [fill={rgb, 255:red, 0; green, 0; blue, 0 }  ][line width=0.08]  [draw opacity=0] (12,-3) -- (0,0) -- (12,3) -- cycle    ;
\draw    (228,181) -- (228.47,153.89) ;
\draw [shift={(228.5,151.89)}, rotate = 90.99] [fill={rgb, 255:red, 0; green, 0; blue, 0 }  ][line width=0.08]  [draw opacity=0] (12,-3) -- (0,0) -- (12,3) -- cycle    ;
\draw    (305,120) -- (369.33,119.03) ;
\draw [shift={(371.33,119)}, rotate = 179.14] [color={rgb, 255:red, 0; green, 0; blue, 0 }  ][line width=0.75]    (10.93,-3.29) .. controls (6.95,-1.4) and (3.31,-0.3) .. (0,0) .. controls (3.31,0.3) and (6.95,1.4) .. (10.93,3.29)   ;
\draw   (397,120) .. controls (397,85.76) and (424.76,58) .. (459,58) .. controls (493.24,58) and (521,85.76) .. (521,120) .. controls (521,154.24) and (493.24,182) .. (459,182) .. controls (424.76,182) and (397,154.24) .. (397,120) -- cycle ;
\draw  [draw opacity=0] (520.56,121.91) .. controls (512.52,134.2) and (487.86,142.94) .. (458.8,142.59) .. controls (430.19,142.24) and (406.07,133.2) .. (397.93,121.02) -- (459.16,112.59) -- cycle ; \draw   (520.56,121.91) .. controls (512.52,134.2) and (487.86,142.94) .. (458.8,142.59) .. controls (430.19,142.24) and (406.07,133.2) .. (397.93,121.02) ;  
\draw  [draw opacity=0][dash pattern={on 4.5pt off 4.5pt}] (397,120) .. controls (404.73,107.51) and (429.16,98.15) .. (458.22,97.75) .. controls (487.88,97.36) and (512.97,106.43) .. (520.5,119.11) -- (458.62,127.75) -- cycle ; \draw  [dash pattern={on 4.5pt off 4.5pt}] (397,120) .. controls (404.73,107.51) and (429.16,98.15) .. (458.22,97.75) .. controls (487.88,97.36) and (512.97,106.43) .. (520.5,119.11) ;  
\draw    (520.5,119.11) -- (485.33,119.95) ;
\draw [shift={(483.33,120)}, rotate = 358.63] [fill={rgb, 255:red, 0; green, 0; blue, 0 }  ][line width=0.08]  [draw opacity=0] (12,-3) -- (0,0) -- (12,3) -- cycle    ;
\draw    (397.93,121.02) -- (434.76,120.92) ;
\draw [shift={(436.76,120.91)}, rotate = 179.84] [fill={rgb, 255:red, 0; green, 0; blue, 0 }  ][line width=0.08]  [draw opacity=0] (12,-3) -- (0,0) -- (12,3) -- cycle    ;
\draw    (459,58) -- (459.31,87) ;
\draw [shift={(459.33,89)}, rotate = 269.38] [fill={rgb, 255:red, 0; green, 0; blue, 0 }  ][line width=0.08]  [draw opacity=0] (12,-3) -- (0,0) -- (12,3) -- cycle    ;
\draw    (459,182) -- (459.47,154.89) ;
\draw [shift={(459.5,152.89)}, rotate = 90.99] [fill={rgb, 255:red, 0; green, 0; blue, 0 }  ][line width=0.08]  [draw opacity=0] (12,-3) -- (0,0) -- (12,3) -- cycle    ;
\draw   (521,118) .. controls (521,83.76) and (548.76,56) .. (583,56) .. controls (617.24,56) and (645,83.76) .. (645,118) .. controls (645,152.24) and (617.24,180) .. (583,180) .. controls (548.76,180) and (521,152.24) .. (521,118) -- cycle ;
\draw  [draw opacity=0] (644.56,119.91) .. controls (636.52,132.2) and (611.86,140.94) .. (582.8,140.59) .. controls (554.19,140.24) and (530.07,131.2) .. (521.93,119.02) -- (583.16,110.59) -- cycle ; \draw   (644.56,119.91) .. controls (636.52,132.2) and (611.86,140.94) .. (582.8,140.59) .. controls (554.19,140.24) and (530.07,131.2) .. (521.93,119.02) ;  
\draw  [draw opacity=0][dash pattern={on 4.5pt off 4.5pt}] (521,118) .. controls (528.73,105.51) and (553.16,96.15) .. (582.22,95.75) .. controls (611.88,95.36) and (636.97,104.43) .. (644.5,117.11) -- (582.62,125.75) -- cycle ; \draw  [dash pattern={on 4.5pt off 4.5pt}] (521,118) .. controls (528.73,105.51) and (553.16,96.15) .. (582.22,95.75) .. controls (611.88,95.36) and (636.97,104.43) .. (644.5,117.11) ;  
\draw    (644.5,117.11) -- (609.33,117.95) ;
\draw [shift={(607.33,118)}, rotate = 358.63] [fill={rgb, 255:red, 0; green, 0; blue, 0 }  ][line width=0.08]  [draw opacity=0] (12,-3) -- (0,0) -- (12,3) -- cycle    ;
\draw    (519.5,119.11) -- (558.76,118.92) ;
\draw [shift={(560.76,118.91)}, rotate = 179.72] [fill={rgb, 255:red, 0; green, 0; blue, 0 }  ][line width=0.08]  [draw opacity=0] (12,-3) -- (0,0) -- (12,3) -- cycle    ;
\draw    (583,56) -- (583.31,85) ;
\draw [shift={(583.33,87)}, rotate = 269.38] [fill={rgb, 255:red, 0; green, 0; blue, 0 }  ][line width=0.08]  [draw opacity=0] (12,-3) -- (0,0) -- (12,3) -- cycle    ;
\draw    (583,180) -- (583.47,152.89) ;
\draw [shift={(583.5,150.89)}, rotate = 90.99] [fill={rgb, 255:red, 0; green, 0; blue, 0 }  ][line width=0.08]  [draw opacity=0] (12,-3) -- (0,0) -- (12,3) -- cycle    ;
\draw  [fill={rgb, 255:red, 0; green, 0; blue, 0 }  ,fill opacity=1 ] (517.5,120) .. controls (517.5,118.07) and (519.07,116.5) .. (521,116.5) .. controls (522.93,116.5) and (524.5,118.07) .. (524.5,120) .. controls (524.5,121.93) and (522.93,123.5) .. (521,123.5) .. controls (519.07,123.5) and (517.5,121.93) .. (517.5,120) -- cycle ;

\end{tikzpicture}
\caption{Sequence of two spheres coming together to touch at a point. 
The arrows attached to the spheres demonstrate the ambient vector field needed to give rise to the eigenfunction which corresponds to `shrinking' of these spheres.
Notice that at the non-embedded point this vector field is not well defined as an ambient one.}\label{fig: touching spheres}
    \end{figure}
\end{center}

Finally we make note of another special case, by considering a sequence of $H$-CMC hypersurfaces $\{M_{k}\}$, which bubble converge $M_{k} \rightarrow (\cup_{i = 1}^{l} M^{i}, \Sigma^{1}, \ldots, \Sigma^{J})$, such that each $M_{k}$ arises as the boundary of some open set $E_{k} \subset N$. 
This particular setting has been analysed by Bourni--Sharp--Tinaglia \cite{BST-compactness-bubbling-CMC}, were they showed that for each $i$, $co (m)_{i} = 1$, and, by applying a uniqueness result of Schoen \cite[Theorem 3]{S-UniquenessSymmetryMinimalSurfaces}, that each bubble $\Sigma^{j}$, will be given by a catenoid $\mathcal{C}$.
Thus, as catenoids have index $1$, we have that
\begin{eqnarray*}
        \limsup_{k \rightarrow \infty} ( \text{ind} \, (M_{k}) + \text{nul} \, (M_{k}) ) &\leq& \sum_{i = 1}^{l} (\text{anl-ind} \, (co (M^{i}_{\infty})) + \text{anl-nul} \, (co (M^{i}_{\infty}))) \\
        && \hspace{2cm} + J (1 + \text{nul}_{\omega_{\mathcal{C}, R}} \, (\mathcal{C})).
\end{eqnarray*}
In Section \ref{sec: Jacobi fields on the catenoid} we investigate $\text{nul}_{\omega_{\mathcal{C}, R}} (\mathcal{C})$, and show that it has a lower bound of $n$. 
In particular, in Section \ref{sec: Jacobi fields on the catenoid} we analyse Jacobi fields on the $n$-dimensional catenoid $\mathcal{C} \subset \mathbb{R}^{n + 1}$ (for $n \geq 3$), which arise from rigid motions of $\mathbb{R}^{n + 1}$ (translations, rotations and scalings).
We show that the only non-trivial such Jacobi fields which lie in $W^{1,2} (\mathcal{C})$, or the weighted space $W^{1,2}_{\omega_{\mathcal{C}, R}} (\mathcal{C})$, are those generated by translations which are parallel to the ends of $\mathcal{C}$.

\section*{Acknowledgements}

First and foremost I am indebted to my supervisor Costante Bellettini for time and support throughout this project. 
I also wish to thank Ben Sharp and Giuseppe Tinaglia for informative discussions at the beginning of this project. 
Moreover, I am grateful to Matilde Gianocca for taking the time to answer questions related to \cite{DGR-morse-index-stability}, as well as Jonas Hirsch for an interesting and insightful conversation on aspects related to this work, and \cite{HL-morse-index-estimates}. 
I would also like to thank Royu P.T. Wang for an interesting discussion on spectral theory, and Louis Yudowitz for his interest in this work and valuable comments on preliminary versions of the manuscript. 

\bigskip 

The author was supported by the Engineering and Physical Sciences Research Council [EP/S021590/1].
The EPSRC Centre for Doctoral Training in Geometry and Number Theory (The London School
of Geometry and Number Theory), University College London.

\tableofcontents

\section{Preliminaries}

We briefly define some notation used throughout this paper. 
Consider a Riemannian manifold $(N, g)$.
Throughout the paper, for $p \in N$, and $r \in (0, \text{inj} \, (N))$, we denote $B_{r}^{N} (p)$ to be the open geodesic ball of radius $r$ centred at point $p$.
For $r > 0$, and $p \in \mathbb{R}^{m}$, we denote, $B_{r}^{m} (p)$ as the open ball in $\mathbb{R}^{m}$ of radius $r$ centred at point $p$. 
We denote the Ricci curvature of $(N, g)$ by $\text{Ric}_{N}$.
For a proper immersion $\iota \colon S \rightarrow N$, we denote the second fundamental form of this immersion by $|A_{S}|$.

\bigskip 

We denote an $n$-rectifiable varifold $V$ in $N$, by a pair $V = (\Sigma, \theta)$, where, $\Sigma \subset N$ is an $n$-rectifiable set, and $\theta \colon \Sigma \rightarrow \mathbb{R}_{\geq 0}$.
We refer the reader to \cite[Chapter 4]{LSimonGMTNotes} for a more detailed discussion on $n$-rectifiable varifolds.

\bigskip 

\begin{definition}\label{def: quasi-embedded hypersurfaces}
    For $H > 0$, we say that a closed set $M \subset N^{n + 1}$ is a $H$-CMC quasi-embedded hypersurface of $N$, if there exists a smooth manifold $S$ of dimension $n$, and a smooth, proper, two-sided immersion $\iota \colon S \rightarrow N$, such that $M = \iota (S)$, and for every $y \in M$, there exists a $\rho > 0$, such that either:
    \begin{enumerate} 
        \item\label{point: embedded point} $B_{\rho}^{N} (y) \cap M$ is a smooth embedded, $n$-dimensional disk, of constant mean curvature $H$,
        \item\label{point: non-embedded point} $B_{\rho}^{N} (y) \cap M$ is the union of two smooth embedded, $n$-dimensional disks, both of constant mean curvature $H$, and only intersecting tangentially along a set containing $y$, with mean curvature vectors pointing in opposite directions.
    \end{enumerate}
    Points $y \in M$ which satisfy \ref{point: embedded point} will be refered to as embedded points of $M$, and we denote the set of such points by $e(M)$. 
    Points in the set $t (M) = M \setminus e (M)$, which satisfy \ref{point: non-embedded point}, will be refered to as non-embedded points of $M$.
\end{definition}

\bigskip 

It is worth noting that in the literature these hypersurfaces are defined under different names. 
In \cite{ZZ-min-max-CMC} they are refered to as almost embedded and in \cite{BST-compactness-bubbling-CMC} they are refered to as effectively embedded.
We opt for the name quasi-embedded as in \cite{BW-existence-pmc}.
We define the set $\overline{\mathfrak{C}_{H}} (N, g)$ to be the set of quasi-embedded, $H$-CMC hypersurfaces in $N$, with respect to the metric $g$.

\subsection{Bubble Convergence Preliminaries}

In this section we give a precise definition of bubble convergence (Definition \ref{def: bubble convergence}), and describe the structure of the neck regions in the bubble convergence (Remark \ref{rem: neck region can be written as a graph}), as well as the ends of the bubbles (Remark \ref{rem: structure of ends of minimal immersions of finite total curvature}).

\begin{definition}\label{def: bubble convergence}
    Consider a Riemannian manifold $(N,g)$, of dimension $n + 1$, $n \geq 2$, (and $H > 0$) along with a sequence $\{ M_{k} \}_{k \in \mathbb{N}} \subset \mathfrak{M} (N, g)$ ($\{ M_{k} \}_{k \in \mathbb{N}} \subset \mathfrak{C}_{H} (N, g)$), an $M_{\infty} \in \mathfrak{M} (N,g)$ ($M_{\infty} \in \overline{\mathfrak{C}_{H}} (N, g)$), and a collection of non-planar, complete, properly embedded, minimal hypersurfaces $\{ \Sigma_{j} \}_{j = 1}^{J}$ in $\mathbb{R}^{n + 1}$, with $J \in \mathbb{Z}_{\geq 1}$.
    Then, we say that 
    \begin{equation*}
        M_{k} \rightarrow (M_{\infty}, \Sigma_{1}, \ldots, \Sigma_{J}),
    \end{equation*}
    bubble converges, if:
    \begin{enumerate}
        \item\label{point: 1 of bubble convergence def} For the case of minimal hypersurfaces; $(M_{k}, 1) \rightarrow (M_{\infty}, \theta)$ as varifolds, where $\theta: M_{\infty} \rightarrow \mathbb{Z}_{\geq 1}$, and is constant on connected components of $M_{\infty}$.
        Moreover, there exists an at most finite collection of points $\mathcal{I} \subset M_{\infty}$, such that locally on $M_{\infty} \setminus \mathcal{I}$, $M_{k}$ converges smoothly and graphically, with multiplicity locally given by $\theta$.

        \bigskip 

        For the case of $H$-CMC hypersurfaces; $(M_{k}, 1) \rightarrow (M_{\infty}, \theta)$ as varifolds, where $M_{\infty} = \cup_{i = 1}^{a} M_{\infty}^{i}$, and each $M_{\infty}^{i}$ is a distinct, closed, quasi-embedded $H$-CMC hypersurface, such that for its respective immersion, $\iota^{i} \colon S^{i} \rightarrow M_{\infty}^{i}$, $S^{i}$ is connected, and there exists a $\theta^{i} \in \mathbb{Z}_{\geq 1}$, such that $\theta (y) = \sum_{i = 1}^{a} (|(\iota^{i})^{-1} (y)| \theta^{i})$.
        Moreover, there exists an at most finite collection of points $\mathcal{I} \subset M_{\infty}$, such that locally on $M_{\infty} \setminus \mathcal{I}$, $M_{k}$ converges smoothly and graphically, with multiplicity locally given by $\theta$.
        \item\label{point: 2 of bubble convergence def} For each $i \in \{1, \ldots, J\}$, there exist point-scale sequences $\{ (p_{k}^{i}, r_{k}^{i}) \}_{k \in \mathbb{N}}$, such that for each $k \in \mathbb{Z}_{\geq 1}$, $p_{k}^{i} \in M_{k}$, and there exists a $y^{i} \in \mathcal{I}$, such that $p_{k}^{i} \rightarrow y^{i}$, $r_{k}^{i} \rightarrow 0$.
        Moreover, for each $R \in (0, \infty)$, and large enough $k$, the connected component of $M_{k} \cap B_{R r_{k}^{i}}^{N} (p_{k}^{i})$, through $p_{k}^{i}$, denoted $\Sigma_{k}^{i, R}$, is such that, if we rescale the geodesic ball $B^{N}_{R r_{k}^{i}} (p_{k}^{i})$ by $r_{k}^{i}$, and denote
        \begin{equation*}
            \tilde{\Sigma}_{k}^{i, R} \cap B_{R}^{n + 1} (0) \coloneqq \frac{1}{r_{k}^{i}} \exp_{p_{k}^{i}}^{-1} (\Sigma_{k}^{i, R} \cap B_{R r_{k}^{i}}^{N} (p_{k}^{i})) \subset B_{R}^{n + 1} (0) \subset \mathbb{R}^{n + 1},
        \end{equation*}
        then $\tilde{\Sigma}_{k}^{i, R} \rightarrow \Sigma^{i} \cap B_{R} (0)$ smoothly and graphically, and hence with multiplicity one. 
        Furthermore, for $i \not= j$, either
        \begin{equation*}
            \lim_{k \rightarrow \infty} \frac{r_{k}^{i}}{r_{k}^{j}} + \frac{r_{k}^{j}}{r_{k}^{i}} + \frac{\text{dist}_{g}^{N} (p_{k}^{i}, p_{k}^{j})}{r_{k}^{i} + r_{k}^{j}} = \infty,
        \end{equation*} 
        or for each $R \in (0, \infty)$, and then large enough $k \in \mathbb{N}$, $p_{k}^{j} \not\in \Sigma_{k}^{i, R}$.
        \item\label{pt: neck condition in bubble convergence} Defining, 
        \begin{equation*}
            d_{k} (x) \coloneqq \min_{i = 1, \ldots, J} \text{dist}_{g}^{N} (x, p_{k}^{i}), 
        \end{equation*}
        then, 
        \begin{equation*}
            \lim_{\delta \rightarrow 0} \lim_{R \rightarrow \infty} \lim_{k \rightarrow \infty} \sup_{x \in M_{k} \cap (\cup_{y \in \mathcal{I}} B^{N}_{\delta} (y) \setminus \cup_{i = 1}^{J} \Sigma_{k}^{i, R})} \int_{M^{x}_{k} \cap B^{N}_{d_{k} (x) / 2} (x)} |A_{k}|^{n} = 0,
        \end{equation*}
        where $M_{k}^{x} \cap B^{N}_{d_{k} (x) / 2} (x)$ is the connected component of $M_{k} \cap B^{N}_{d_{k} (x) / 2} (x)$, that contains $x$. 
\end{enumerate}
\end{definition}

\bigskip 

Consider such a convergence  $M_{k} \rightarrow (M_{\infty}, \Sigma^{1}, \ldots, \Sigma^{J})$, as in Definition \ref{def: bubble convergence}, then we may remark: 
\begin{itemize}
    \item The convergence considered in the bubble analysis of Chodosh--Ketover--Maximo \cite{CKM-MinimalHypersurfacesBoundedIndex} and Buzano--Sharp \cite{BS-QQ-estimates-minimal-hypersurfaces-bounded-index-and-area} for minimal hypersurfaces, and Bourni--Sharp--Tinaglia \cite{BST-compactness-bubbling-CMC} for CMC boundaries, satisfies Definition \ref{def: bubble convergence}.
    \item As the multiplicity function $\theta \colon M_{\infty} \rightarrow \mathbb{Z}_{\geq 1}$ is uniformly bounded, and $\mathcal{H}^{n} (M_{\infty} \cap U) < + \infty$, for $U \subset N$ compact, then we may deduce by the varifold convergence that $\sup_{k} \mathcal{H}^{n} (M_{k} \cap U) < + \infty$. 
    Applying this fact, and using the monotonicity formula for varifolds with bounded mean curvature (\cite[Theorem 17.7]{LSimonGMTNotes}), we may deduce that each $\Sigma^{j}$ must have Euclidean volume growth at infinity (see \cite[Corollary 2.6]{BS-QQ-estimates-minimal-hypersurfaces-bounded-index-and-area}).
    \item The final sentence of point \ref{point: 2 of bubble convergence def} in Definition \ref{def: bubble convergence} guarantees that all of these bubbles are distinct.
    \item For each $j = 1, \ldots, J$, as $\Sigma^{j}$ is complete and properly embedded, by \cite{S-OrientabilityOfHypersurfaces} $\Sigma^{j}$ is two-sided in $\mathbb{R}^{n + 1}$.
    \item Without loss of generality we may assume that $|A_{\Sigma^{j}}| (0) > 0$, for each $j = 1, \ldots, J$.
\end{itemize}
We also now assume that each $\Sigma^{j}$ has finite index. 
This is a reasonable assumption for us to make, as if it does not hold then the result that we are interested in (Theorem \ref{thm: main theorem intro}) would hold trivially.
Thus, for $3 \leq n \leq 6$, by a result of Tysk \cite{Tysk-finite-curvature-index-minimal-surfaces}, the finite index and Euclidean volume growth at infinity, imply that each bubble $\Sigma^{j}$ will have finite total curvature; 
\begin{equation*}
    \int_{\Sigma^{j}} |A_{\Sigma^{j}}|^{n} < + \infty.
\end{equation*}

\bigskip 

The following curvature estimate (Proposition \ref{prop: curvature estimate}) is important in our analysis, and the proof follows from a standard point-picking argument, which may be found in \cite[Lecture 3]{White-IntroMinimalSurfaces}.
We briefly list notation used in the statement; If $\iota \colon M \rightarrow N$ is a proper immersion, and $S \subset N$, $x \in M$, and $\iota (x) \in S$, then we denote $\iota^{-1} (S)_{x}$ as the connected component of $\iota^{-1} (S)$ which contains $x$.

\begin{proposition}\label{prop: curvature estimate}
    Consider $(B_{1}^{n + 1} (0), g)$, where $g$ is a Riemannian metric (a constant $H > 0$), and a proper, $g$-minimal ($g$-$H$-CMC) immersion $\iota \colon M \rightarrow B_{1}^{n + 1} (0)$, such that $\iota (\partial M) \subset \partial B_{1}^{n + 1} (0)$. 
    There exists an $\varepsilon_{0} = \varepsilon_{0} (g) (= \varepsilon_{0} (g, H)) > 0$, such that for $x \in \iota^{-1} (B_{1/2}^{n + 1} (0))$, $r \in (0, 1/4)$, and $\varepsilon \in (0, \varepsilon_{0})$, if 
    \begin{equation*}
        \int_{\iota^{-1} (B_{r}^{n + 1} (\iota (x)))_{x}} |A|^{n} \leq \varepsilon, 
    \end{equation*}
    then,
    \begin{equation*}
        \sup_{y \in B_{r}^{n + 1} (\iota (x)))_{x}} \text{dist}_{g}^{B_{1}^{n + 1} (0)} (\iota (y), \partial B_{r}^{n + 1} (\iota (x))) \, |A| (y) \leq C_{\varepsilon},
    \end{equation*}
    with $C_{\varepsilon} = C (g, \varepsilon) (= C (g, H, \varepsilon)) < + \infty$.
    Moreover, $C_{\varepsilon} \rightarrow 0$ as $\varepsilon \rightarrow 0$.
\end{proposition}

\bigskip 

Proposition \ref{prop: curvature estimate} implies that if we take a sequence $\delta_{k} \rightarrow 0$ and $R_{k} \rightarrow \infty$, such that $R_{k} r_{k}^{j} \rightarrow 0$ for all $j = 1, \ldots, J$, and $\delta_{k} \geq 4 R_{k} r_{k}^{j}$ for all $k \in \mathbb{Z}_{\geq 1}$ and $j = 1, \ldots, J$, and pick a sequence of points, 
\begin{equation*}
    x_{k} \in M_{k} \cap (\cup_{y \in \mathcal{I}} B^{N}_{\delta_{k}} (y) \setminus \cup_{i = 1}^{J} \Sigma_{k}^{i, R_{k}}),
\end{equation*}
and denote, $s_{k} = d_{k} (x_{k}) / 2$, and 
\begin{equation*}
    \tilde{M_{k}} \cap B_{1}^{n + 1} (0) \coloneqq \frac{1}{s_{k}} \exp_{x_{k}}^{-1} ( M_{k} \cap B_{s_{k}}^{N} (x_{k})),
\end{equation*}
then, after potentially taking a subsequence and renumerating, the component of $\tilde{M}_{k}$ through the origin must smoothly converge to a plane through the origin, on compact sets of $B_{1}^{n + 1} (0)$.

\begin{rem}\label{rem: neck region can be written as a graph}
    (Graphicality of the necks)
    Following arguments in \cite[Claim 1 of Lemma 4.1]{BS-QQ-estimates-minimal-hypersurfaces-bounded-index-and-area} for minimal hypersurfaces, and \cite[Lemma 5.6]{BST-compactness-bubbling-CMC} for CMC hypersurfaces, and replacing the use of stability with the combination of Point 3 of Definition \ref{def: bubble convergence} and the curvature estimate in Proposition \ref{prop: curvature estimate}, we have that there exists positive constants $S_{0}$ and $s_{0}$, such that for each $y \in \mathcal{I}$, $s \in (0, s_{0})$, and $S \in [S_{0}, + \infty)$, and taking large enough $k$, if we let $C_{k}$ denote a connected component of $M_{k} \cap (B_{s}^{N} (y) \setminus \cup_{i = 1}^{J} \overline{\Sigma_{k}^{i, S}})$, then there exists a non-empty open set $A_{k} = A_{k} (C_{k}, s, S) \subset T_{y} M_{\infty}$, and a smooth function,
    \begin{equation*}
        u_{k} \colon A_{k} \rightarrow \mathbb{R},
    \end{equation*}
    such that, 
    \begin{equation*}
        C_{k} = \{ \exp_{y} (x + u_{k} (x) \nu (y)) \colon x \in A_{k} \},
    \end{equation*}
    where $\nu (y)$ is a choice of unit normal to $M_{\infty}$ at $y$.
    Moreover, we have that
    \begin{equation*}
        \lim_{s \rightarrow 0} \lim_{S \rightarrow \infty} \lim_{k \rightarrow \infty} \| u_{k} \|_{C^{1}} = 0.
    \end{equation*}
\end{rem}

\bigskip

\begin{rem}\label{rem: structure of ends of minimal immersions of finite total curvature}
(Graphicality of the ends of bubbles)
Consider $\iota \colon \Sigma \rightarrow \mathbb{R}^{n + 1}$, a $n$-dimensional ($n \geq 3$), complete, connected, two-sided, proper minimal immersion, of finite total curvature, 
\begin{equation*}
    \int_{\Sigma} |A_{\Sigma}|^{n} < + \infty,
\end{equation*}
and Euclidean volume growth at infinity, 
\begin{equation*}
    \lim_{R \rightarrow \infty} \frac{\mathcal{H}^{n} (\iota (\Sigma) \cap B_{R}^{n + 1} (0))}{R^{n}} < + \infty.
\end{equation*}
Following arguments in \cite[Lemma 3 and 4]{Tysk-finite-curvature-index-minimal-surfaces} and \cite[Proposition 3]{BeChWi-CMC}, replacing any use of stability with the finite total curvature assumption and Proposition \ref{prop: curvature estimate} (see also \cite[Appendix A]{L-IndexTopologyMinimalHypersurfacesInRn}), we may conclude that $\Sigma$, will have finitely many ends, $E^{1}, \ldots, E^{m}$, such that for for each $i = 1, \ldots, m$, there exists a rotation $r_{i}$, centred at the origin, a compact set $B_{i} \subset \mathbb{R}^{n}$, and a smooth function, 
\begin{equation*}
    u_{i} \colon \mathbb{R}^{n} \setminus B_{i} \rightarrow \mathbb{R}, 
\end{equation*}
such that, 
\begin{equation*}
    r_{i} (\iota (E^{i})) = \text{graph} \, (u_{i}) \coloneqq \{ (x, u_{i} (x)) \colon x \in \mathbb{R}^{n} \setminus B_{i} \}.
\end{equation*}
Moreover, 
\begin{equation*}
    \lim_{R \rightarrow \infty} \| \nabla u_{i} \|_{L^{\infty} (\mathbb{R}^{n} \setminus B_{R}^{n} (0))} = 0.
\end{equation*}
\end{rem}

\subsection{Stability Operator, Index and Nullity}\label{subsec: stability operator, index and nullity}

Consider a smooth, closed, connected, and properly embedded hypersurface $M \subset N$. 
We denote its normal bundle by $TM^{\perp}$, and its oriented double cover by $o (M)$, which may be given by, 
\begin{equation*}
    o (M) = \{ (x, \nu) \colon x \in M, \, \nu \in T_{x}^{\perp} M, \, \text{and} \, |\nu| = 1 \}.
\end{equation*}
We then define the obvious projections $\iota \colon o (M) \rightarrow M$, and $\nu \colon o (M) \rightarrow TM^{\perp}$.
Taking a connected component of $o (M)$, which we denote by $co (M)$, we have that either, $M$ is two-sided (and therefore $TM^{\perp}$ is trivial), and thus $co (M) = M$, or $M$ is one-sided, and thus $co (M) = o (M)$.

\bigskip 

Let $M$ be a minimal hypersurface, then we consider the following bilinear form (which arises as the second variation of the area functional with respect to ambient variations \cite[Chapter 1 Section 8]{Colding-Minicozzi-Course-Minimal-Surfaces}) on sections of $TM^{\perp}$, 

\begin{equation*}
    B_{L} [v, v] \coloneqq \int_{M} |\nabla^{\perp} v|^{2} - |A_{M}|^{2} |v|^{2} - \text{Ric}_{N} (v , v) \, d \mathcal{H}^{n}. 
\end{equation*}

In the case when $M$ is two-sided, we see that this is equivalent to considering the following bilinear form, which can be extended to the function space $W^{1,2} (M)$, 
\begin{equation}\label{eqn: bilinear from}
    B_{L} [f, h] \coloneqq \int_{M} \nabla f \cdot \nabla h - |A_{M}|^{2} \, f \, h - \text{Ric}_{N} (\nu , \nu) \, f \, h \, d \mathcal{H}^{n},
\end{equation}
where $\nu$ is a choice of unit normal to $M$.

\bigskip 

In the case when $M$ is one-sided, this is equivalent to considering the bilinear form
\begin{equation}
    B_{L} [f, h] \coloneqq \int_{o (M)} \nabla f \cdot \nabla h - |A_{M} \circ \iota|^{2} \, f \, h - \text{Ric}_{N} (\nu ( \, \cdot \,) , \nu ( \, \cdot \,)) \, f \, h \, d \mathcal{H}^{n},
\end{equation}
but on the function space
\begin{equation*}
    W^{1,2} (o (M))^{-} \coloneqq \{ f \in W^{1,2} (o (M)) \colon \text{for a.e.} \, (x, \nu) \in o (M), \, f((x, \nu)) = - f ((x, - \nu)) \}.
\end{equation*}
In both scenarios, integrating by parts, we see that $B_{L}$, is the weak formulation of the well known stability (or Jacobi) operator, 
\begin{equation}
    L = \Delta + |A_{M}|^{2} + \text{Ric} (\nu, \nu).
\end{equation}
In order to stop having to differentiate between the cases when $M$ is two-sided and one-sided we introduce the notation $W^{1,2} (co (M))^{-}$, which simply equals $W^{1,2} (M)$, when $M$ is two-sided, and $W^{1,2} (o (M))^{-}$, when $M$ is one-sided.

\bigskip 

The Morse index of $M$ (denoted $\text{ind} \, (M)$) may then be defined to be the number of negative eigenvalues (counted with multiplicity) of the stability operator $L$, acting on the function space $W^{1,2} (co (M))^{-}$. 
The nullity of $M$ (denoted $\text{nul} \, (M)$), may then be defined as the number of zero eigenvalues of $L$ (counted with multiplicity) acting on the function space $W^{1,2} (co (M))^{-}$.
In this paper it is also necessary for us to also define the analytic index (denoted $\text{anl-ind} \, (M)$) and analytic nullity (denoted $\text{anl-nul} \, (M)$) of $M$, which is defined as the number of negative and zero eigenvalues of $L$, acting on the function space $W^{1,2} (co (M))$. 
We see that if $M$ is two-sided, then the Morse index and nullity and analytic index and nullity are equal.
One should think of $W^{1,2} (co (M))^{-}$, as `ambient variations', and $W^{1,2} (co (M))$ as `intrinsic variations'.

\bigskip 

If $M$ is an embedded $H$-CMC hypersurface, such that $M = \partial E$, for some non-empty, open set $E$, then one sees that the bilinear form $B_{L}$, defined on sections of $TM^{\perp}$, arises as the second variation of the functional, 
\begin{equation*}
    \mathcal{F}_{H} (E) = \mathcal{H}^{n} (\partial E) - H \mathcal{H}^{n + 1} (E),
\end{equation*}
with respect to ambient variations (\cite[Proposition 2.5]{Barbosa-doCarmo-Eschenburg-CMC}).
However, this bilinear form $B_{L}$ (and corresponding operator $L$), can be defined on $M$, even if it does not arise as the boundary, and thus (as embedded $H$-CMC hypersurfaces are two-sided for $H \not= 0$), we may similarly define the Morse index and nullity and $M$ as the number of negative and zero eigenvalues of the operator $L$, acting on the space $W^{1,2} (M)$.
In general, for a quasi-embedded $H$-CMC hypersurface $M$, given by the proper two-sided immersion $\iota \colon S \rightarrow N$, we denote $co (M) = S$, and define the analytic index and analytic nullity of $M$, as the number of negative and zero eigenvalues of the operator 
\begin{equation*}
    L = \Delta + |A_{S}|^{2} + \text{Ric}_{N} (\nu, \nu),
\end{equation*}
acting on the function space $W^{1,2} (co (M))$. 
Here, $A_{S}$ is the second fundamental form of the immersion $\iota$, and $\nu$ is a choice of unit normal for this two-sided immersion.
Again we see that if $M$ is embedded then the Morse index and nullity match up with this analytic index and analytic nullity. 

\subsection{Lorentz Spaces}

Let $(M, g)$ be a Riemannian manifold and $\mu$ be the volume measure associated to $g$. 
For a $\mu$-measureable function $f \colon M \rightarrow \mathbb{R}$, we define the function 
\begin{equation*}
    \alpha_{f} (s) \coloneqq \mu (\{ x \in M \colon |f(x)| > s\}).
\end{equation*}
We may then define the decreasing rearrangement $f^{*}$, of $f$ by, 
\begin{equation*}
    f^{*} (t) \coloneqq \begin{cases}
        \inf \{ s > 0 \colon \alpha_{f} (s) \leq t \}, & t > 0, \\
        \esssup |f|, & t = 0. 
    \end{cases}
\end{equation*}
For $p \in [1, \infty)$, and $q \in [1, \infty]$, and a $\mu$-measureable function $f$ on $M$, we define,
\begin{equation*}
    \| f \|_{(p,q)} \coloneqq \begin{cases}
        \left( \int_{0}^{\infty} t^{q / p} f^{*} (t)^{q} \frac{dt}{t} \right)^{1 / q}, & 1 \leq q < \infty, \\
        \sup_{t > 0} t^{1 / p} f^{*} (t), & q = \infty.
    \end{cases}
\end{equation*}
The Lorentz space $L (p, q) (M, g)$ is then defined to be the space of $\mu$-measureable functions $f$ such that $\|f\|_{(p,q)} < + \infty$.
It is worth noting that $\| \, \cdot \, \|_{(p, q)}$ is not a norm on $L(p,q) (M, g)$ as it does not generally satisfy the triangle inequality.
However it is possible to define an appropriate norm, $\| \cdot \|_{p,q}$ (see \cite[Definition 2.10]{CC-LorentzSpaces}), on the space $L(p, q) (M, g)$, so that the normed space $(L(p, q) (M, g), \| \cdot \|_{p,q})$ is a Banach Space \cite[Theorem 2.19]{CC-LorentzSpaces}. 
Moreover, for $1 < p < \infty$, and $1 \leq q \leq \infty$, we have the following equivalence (\cite[Proposition 2.14]{CC-LorentzSpaces}), 
\begin{equation}\label{eqn: equivalence of Lorentz norm}
    \| \, \cdot \, \|_{(p, q)} \leq \| \, \cdot \, \|_{p, q} \leq \frac{p}{p - 1} \| \, \cdot \, \|_{(p, q)}.
\end{equation}

\begin{proposition}\label{prop: Holder Lorentz inequality}
    (H\"{o}lder--Lorentz inequality, \cite[Theorem 2.9]{CC-LorentzSpaces})
    Take $p_{1}, \, p_{2} \in (1, \infty)$ and $q_{1}, \, q_{2} \in [1, \infty]$ such that $1 / p_{1} + 1 / p_{2} = 1 / q_{1} + 1 / q_{2} = 1$. 
    Then for $f \in L (p_{1}, q_{1}) (M, g)$ and $h \in L (p_{2}, q_{2}) (M, g)$, we have, 
    \begin{equation*}
        \int_{M} |f \, h | \, d \mu \leq \| f \|_{(p_{1}, q_{1})} \, \| h \|_{(p_{2}, q_{2})}.
    \end{equation*}
\end{proposition}

The following fact can be easily derived from the definition of $\| \, \cdot \, \|_{(p, q)}$.

\begin{proposition}\label{prop: Lorentz norm of function raised to a power}
    Take $1 < p < + \infty$, $1 \leq q \leq \infty$, and $\gamma > 0$, then for $f \in L (p, q) (M, g)$, we have, 
    \begin{equation*}
        \||f|^{\gamma}\|_{(\frac{p}{\gamma}, \frac{q}{\gamma})} = \| f \|^{\gamma}_{(p,q)}.
    \end{equation*}
\end{proposition}

The following Lorentz--Sobolev inequality on $\mathbb{R}^{n}$ is crucial in Section \ref{sec: Strict Stability of the neck}. 
For a proof see \cite[Appendix]{ALT-Prescribed-Reaarangement}

\begin{proposition}\label{prop: Lorentz Sobolev inequality on R n}
    (Lorentz--Sobolev inequality on $\mathbb{R}^{n}$)
    Take $1 < p < n$, and $p^{*} = n p / (n - p)$, then there exists a constant $C = C (n, p)$ such that for all $u \in C_{c}^{\infty} (\mathbb{R}^{n})$, 
    \begin{equation*}
        \| u \|_{(p^{*}, p)} \leq C \|\nabla u\|_{L^{p}}.
    \end{equation*}
\end{proposition}

By standard covering and partitions of unity arguments, from Proposition \ref{prop: Lorentz Sobolev inequality on R n} we may also obtain a Lorentz--Sobolev inequality on a bounded subset of a Riemannian manifold.

\begin{proposition}\label{prop: Lorentz Sobolev inequality on manifolds}
    (Lorentz--Sobolev inequality on manifolds)
    Let $(M, g)$ be a complete Riemannian manifold of dimension $n$. 
    Take $1 < p < n$, $p^{*} = n p / (n - p)$, and an open, bounded set $\Omega \subset M$. 
    Then there exists a constant $C = C (\Omega, g, n, p) < + \infty$, such that for all $u \in C_{c}^{\infty} (\Omega)$, 
    \begin{equation*}
        \| u \|_{(p^{*}, p)} \leq C \| u \|_{W^{1, p} (M)}.
    \end{equation*}
\end{proposition}

\begin{rem}
    We may extend the inequalties in Propositions \ref{prop: Lorentz Sobolev inequality on R n} and \ref{prop: Lorentz Sobolev inequality on manifolds} to $u \in W^{1,p} (\mathbb{R}^{n})$ and $W^{1,p}_{0} (\Omega)$ respectively, by using a standard density argument, the fact $(L(p,q) (M, g), \| \, \cdot \, \|_{p, q})$ is a Banach space, and the equivalence in (\ref{eqn: equivalence of Lorentz norm}).
\end{rem}

\begin{proposition}\label{prop: W 1 2 bounds on weighted L2 inner product}
    Let $(M, g)$ be a compact Riemannian manifold of dimension $n \geq 3$, with $\omega \in L(n / 2, \infty) (M, g)$. 
    Then for any $f_{1}, \, f_{2} \in W^{1,2} (M)$, we have that there exists a $C = C (M, g) < \infty$, such that,
    \begin{equation*}
        \left| \int_{M} f_{1} f_{2} \omega \right| \leq C \|\omega\|_{(n / 2, \infty)} \|f_{1}\|_{W^{1,2} (M)} \|f_{2}\|_{W^{1,2} (M)}.
    \end{equation*}
\end{proposition}

\begin{proof}Using Propositions \ref{prop: Holder Lorentz inequality}, \ref{prop: Lorentz norm of function raised to a power} and \ref{prop: Lorentz Sobolev inequality on manifolds},
    \begin{eqnarray*}
        \left| \int_{M} f_{1} f_{2} \omega \right| &\leq& \left( \int_{M} |\omega| \, f_{1}^{2} \right)^{1/2} \, \left( \int_{M} |\omega|  \, f_{2}^{2} \right)^{1/2}, \\
        &\leq& \| \omega \|_{(n / 2, \infty)} \, \| f_{1}^{2} \|_{(2^{*} / 2, 1)}^{1/2} \, \| f_{2}^{2} \|_{(2^{*} / 2, 1)}^{1/2}, \\
        &\leq& \| \omega \|_{(n / 2, \infty)} \, \| f_{1} \|_{(2^{*}, 2)} \, \| f_{2} \|_{(2^{*}, 2)}, \\
        &\leq& C \| \omega \|_{(n / 2, \infty)} \, \| f_{1} \|_{W^{1,2} (M)} \, \| f_{2} \|_{W^{1,2} (M)}.
    \end{eqnarray*}
\end{proof}

\section{Weighted Eigenfunctions}

We proceed with the proof of Theorem \ref{thm: main theorem intro}, taking a sequence $M_{k} \rightarrow (M_{\infty}, \Sigma^{1}, \ldots, \Sigma^{J})$ as in Definition \ref{def: bubble convergence}.
We assume that for each $k \in \mathbb{Z}_{\geq 1}$, $M_{k}$ is a single connected component. 
Thus for the case of minimal hypersurfaces, $l = 1$.
The general statement is proven by applying the argument to each individual connected component of $M_{k}$.

\bigskip 

We also note that the reader may find it easier to follow the rest of this paper by only considering the case where each $M_{k}$, and $M_{\infty}$, are two-sided minimal hypersurfaces.
By doing so much of the notation introduced in Section \ref{subsec: stability operator, index and nullity} can be ignored, and we can just consider $M_{k}$ instead of $co (M_{k})$, and $W^{1,2} (M_{k})$ instead of $W^{1,2} (co (M_{k}))^{-}$. 

\subsection{The Weight}

Take, $R \geq 4$, and $\delta > 0$, such that for large enough $k$, and all $j = 1, \ldots, J$, 
\begin{equation*}
    4 R r_{k}^{j} < \delta < \min \left\{ \frac{\text{inj} (N)}{8}, \, \min_{y_{1}, \, y_{2} \in \mathcal{I}, \, y_{1} \not= y_{2}} \frac{\text{dist}_{g}^{N} (y_{1}, y_{2})}{8} \right\}.
\end{equation*} 
We first define our weight, on $co (M_{k})$, about the point scale sequence $\{ (p_{k}^{j}, r_{k}^{j}) \}_{k \in \mathbb{N}}$, 
\begin{equation*}
    \omega_{k, \delta, R}^{j} (x) \coloneqq \begin{cases}
        \max \{ \delta^{-2}, \text{dist}_{g}^{N} (\iota (x), p_{k}^{j})^{-2} \}, & \iota_{k} (x) \in M_{k} \setminus B_{R r_{k}^{j}}^{N} (p_{k}^{j}), \\
        (R r_{k}^{j})^{-2}, & \iota (x) \in B_{R r_{k}^{j}} (p_{k}^{j}) \cap M_{k},
    \end{cases}
\end{equation*}
We consider the weight,
\begin{equation*}
    \omega_{k, \delta, R} (x) \coloneqq \max_{j = 1, \ldots J} \, \omega_{k, \delta, R}^{j} (x).
\end{equation*}
We also define $\omega_{\delta} \in W^{1, \infty}_{\text{loc}} (co (M_{\infty}) \setminus \iota_{\infty}^{-1} (\mathcal{I})) \cap L (n/2, \infty) (co (M_{\infty}))$, by
\begin{equation*}
    \omega_{\delta} (x) \coloneqq \max \{ \delta^{-2}, \text{dist}_{g}^{N} (\iota_{\infty} (x), \mathcal{I})^{-2} \}.
\end{equation*}
The fact that $\omega_{\delta} \in L (n / 2, \infty) (co (M_{\infty}))$ will follow from a similar calculation to that in Claim \ref{claim: uniform L n over 2 infinity bounf on weight}.

\bigskip 

Recall the stability operator on $co (M_{k})$, 
\begin{equation*}
    L_{k} \coloneqq \Delta + |A_{k}|^{2} + R_{k},
\end{equation*}
where $A_{k}$ denotes the second fundamental form of the immersion $\iota_{k} \colon M_{k} \rightarrow N$, and $R_{k} (x) = \text{Ric}_{N} (\nu_{k} (x), \nu_{k} (x)) (\iota_{k} (x))$, and the associated bilinear form, $B_{k}$, acting on $W^{1,2} (co (M_{k}))$, 
\begin{equation*}
    B_{k} [\varphi, \psi] \coloneqq \int_{M_{k}} \nabla^{k} \varphi \cdot \nabla^{k} \psi - (|A_{k}|^{2} + R_{k}) \, \varphi \, \psi.
\end{equation*}
We define the unweighted eigenspace for an eigenvalue $\lambda$ of $L_{k}$ by
\begin{equation}\label{eqn: def of unweighted eigenspace}
    \begin{split}
    \mathcal{E} (\lambda; L_{k}, W^{1,2} (co (M_{k}))^{-}) \coloneqq \{ f \in W^{1,2} (co (M_{k}))^{-} \colon& B_{k} [f, \psi] = \lambda \int_{M_{k}} f \, \psi, \\
    & for \, all \, \psi \in W^{1,2} (co (M_{k}))^{-} \},
    \end{split} 
\end{equation}\label{eqn: def of weighted eigenspace}
and the weighted eigenspace for a weighted eignevalue $\lambda$ of $L_{k}$ by,
\begin{equation}
    \begin{split}
    \mathcal{E}_{\omega_{k, \delta, R}} (\lambda; L_{k}, W^{1,2} (co (M_{k}))^{-}) \coloneqq \{ f \in W^{1,2} (co (M_{k}))^{-} \colon& B_{k} [f, \psi] = \lambda \int_{M_{k}} f \, \psi \, \omega_{k, \delta, R}, \\
    & for \, all \, \psi \in W^{1,2} (co (M_{k}))^{-} \}.
    \end{split}
\end{equation}
Identical definitions hold for $\mathcal{E} (\lambda; L_{\infty}, W^{1,2} (co (M_{\infty}))$ and $\mathcal{E}_{\omega_{\delta}} (\lambda; L_{\infty}, W^{1,2} (co (M_{\infty})) )$.
Recall that when we refer to function space $W^{1,2} (co (M))^{-}$ we are considering \textit{`ambient variations'}, and when we consider the function space $W^{1,2} (co (M))$ we are considering \textit{`analytic variations'}.
In general, for $\lambda \in \mathbb{R}$, a Riemannian manifold $(M, g)$, with a second order, linear, elliptic operator $L$, with associated bilinear form $B_{L}$, and $\omega \colon M \rightarrow \mathbb{R}$, we define the function space, 
\begin{equation*}
    \mathcal{E}_{\omega} (\lambda; L, A, B) \coloneqq \{ f \in A \colon B_{L} [f, \varphi] = \int_{M} f \, \varphi \, \omega, \, for \, all \, \varphi \in B\},
\end{equation*} 
where $A$, and $B$ are some function spaces on $M$.
If $A = B$, we write $\mathcal{E}_{\omega} (\lambda; L, A)$, and if $\omega \equiv c$, for some constant $c$, we write, $\mathcal{E}_{c} (\lambda; L, A, B) = \mathcal{E} (c \lambda; L, A, B)$.

\begin{lem}\label{claim: weight bounding second fun form}
    There exists a $C = C(N, g, M_{\infty}, \Sigma^{1}, \ldots, \Sigma^{J}, \delta, R) < + \infty$ such that for all $k$, 
    \begin{equation}\label{eqn: bound on A k over omega}
        \frac{|A_{k}|^{2} + |R_{k}|}{\omega_{k, \delta, R}} \leq C.
    \end{equation} 
\end{lem}

\begin{proof}
    The proof follows from a contradiction argument. 
    We outline the basic idea, and leave details to the reader. 
    If we had a sequence of points $\{x_{k} \in M_{k}\}$, such that
    \begin{equation*}
        \frac{|A_{k}|^{2} (x_{k}) + |R_{k}| (x_{k})}{\omega_{k, \delta, R} (x_{k})} \rightarrow \infty,
    \end{equation*}
    Then by smooth convergence on $M_{\infty} \setminus \mathcal{I}$, and the bubbles $\Sigma^{1}, \ldots, \Sigma^{J}$, and the scale invariance of the quantity, we see that no subsequences can concentrate on the base $M_{\infty} \setminus \mathcal{I}$, or on the bubbles $\Sigma^{1}, \ldots, \Sigma^{J}$. 
    Thus this sequence must concentrate on the intermediate neck regions of the bubble convergence, however this cannot happen by point \ref{pt: neck condition in bubble convergence} in Definition \ref{def: bubble convergence} and Proposition \ref{prop: curvature estimate}.
\end{proof}

\begin{claim}\label{claim: uniform L n over 2 infinity bounf on weight}
    We have that there exists a $C = C (N, g, M_{\infty}, m, \delta, J) (= C (N, g, M_{\infty}, m_{1}, \ldots, m_{a}, H, \delta, J)) < + \infty$ such that for large enough $k \in \mathbb{Z}_{\geq 1}$, and $R \geq 1$, 
    \begin{equation*}
        \|\omega_{k, \delta, R} \|_{(n / 2, \infty)} \leq C.
    \end{equation*}
\end{claim}

\begin{proof}
    For a $j = 1, \ldots, J$, consider $f = \omega_{k, \delta, R}^{j}$. 
    We have, 
    \begin{equation*}
        \alpha_{f} (s) = \begin{cases}
            \mathcal{H}^{n} (co (M_{k})), & s \in [0, \delta^{-2}), \\
            \mathcal{H}^{n} (\iota_{k}^{-1} (B_{1 / \sqrt{s}}^{N} (p_{k}^{j})) \cap co (M_{k})), & \delta^{-2} \leq s < (R r_{k}^{j})^{-2}, \\
            0, & s \geq (R r_{k}^{j})^{-2}.
        \end{cases}
    \end{equation*}
    We may choose $k$ large enough such that $\sup_{k} \mathcal{H}^{n} (M_{k}) \leq m \mathcal{H}^{n} (M_{\infty}) + 1$.
    Then, by the monotonicity formula (\cite[Theorem 17.7]{LSimonGMTNotes}), we have that there exists a uniform $C = C (N, g, m, M_{\infty}) (= C (N, g, M_{\infty}, m_{1}, \ldots, m_{a}, H) ) < + \infty$, such that for $r \in (0, \text{inj} (N) / 2)$, 
    \begin{equation*}
        \mathcal{H}^{n} (co (M_{k}) \cap \iota_{k}^{-1} (B_{r}^{N} (p_{k}^{j}))) \leq C r^{n}
    \end{equation*}
    Then, 
    \begin{equation*}
        \begin{cases}
        f^{*} (t) = 0, & t \geq \mathcal{H}^{n} (co (M_{k})), \\
        f^{*} (t) = \delta^{-2}, & \mathcal{H}^{n} (co (M_{k}) \cap \iota_{k}^{-1} (B_{\delta}^{N} (p_{k}^{j}))) \leq t < \mathcal{H}^{n} (co (M_{k})), \\
        f^{*} (t) \leq (C / t)^{2 /n}, & \mathcal{H}^{n} (co (M_{k}) \cap \iota_{k}^{-1} (B_{R r_{k}^{j}}^{N} (p_{k}^{j}))) \leq t \leq \mathcal{H}^{n} (co (M_{k}) \cap \iota_{k}^{-1} (B_{\delta}^{N} (p_{k}^{j}))), \\
        f^{*} (t) = (R r_{k}^{j})^{-2}, & 0 \leq t \leq \mathcal{H}^{n} (co (M_{k}) \cap \iota_{k}^{-1} (B_{R r_{k}^{j}}^{N} (p_{k}^{j}))).
        \end{cases}
    \end{equation*}
    Thus, 
    \begin{equation*}
        \| f \|_{(n / 2, \infty)} = \sup_{t > 0} t^{2 / n} f^{*} (t) \leq C, 
    \end{equation*}
    for $C = C (N, g, m, M_{\infty}, \delta) (= C (N, g, M_{\infty}, m_{1}, \ldots, m_{a}, H, \delta)) < + \infty$.
    This in turn implies,
    \begin{equation*}
        \| \omega_{k, \delta, R} \|_{(n / 2, \infty)} \leq \left(\frac{n}{n - 2} \right) \sum_{j = 1}^{J} \| \omega_{k, \delta, R}^{j} \|_{(n / 2, \infty)} \leq C,
    \end{equation*}
    for $C = C (N, g, m, M_{\infty}, \delta, J) (= C (N, g, M_{\infty}, m_{1}, \ldots, m_{a}, H, \delta, J)) < + \infty$.
    The $n / (n-2)$ factor is coming from (\ref{eqn: equivalence of Lorentz norm}).
\end{proof}

\bigskip 

\begin{rem}\label{rem: choice of weight incorrect for n equals 2}
    This choice of weight $\omega_{k, \delta, R}$ fails to work for the case of $n = 2$. 
    In \cite{DGR-morse-index-stability} (in which $n = 2$)  the choice of weight is subtle, and relies on improved estimates on the neck region of the bubbles. 
    We were unable to derive appropriate corresponding estimates on the neck regions in the setting of this paper. 
\end{rem}

\subsection{Convergence on the Base}\label{sec: convergence on the base}

Consider a sequence of functions $\{f_{k} \in W^{1,2} (co (M_{k}))^{-}\}_{k \in \mathbb{Z}_{\geq 1}}$, which satisfy the following weighted eigenvalue problem, 
\begin{equation}\label{eqn: weak weighted eigenvalue equation}
    \int_{co (M_{k})} \nabla f_{k} \cdot \nabla \varphi - (|A_{k}|^{2} + R_{k}) \, f_{k} \, \varphi = \lambda_{k} \int_{co (M_{k})} f_{k} \, \varphi \, \omega_{k, \delta, R},
\end{equation}
for all $\varphi \in W^{1,2} (co (M_{k}))^{-}$, with $\lambda_{k} \leq 0$, for all $k$. 
We take 
\begin{equation*}
    \int_{co (M_{k})} f_{k}^{2} \, \omega_{k, \delta, R} = 1,
\end{equation*}
and by Lemma \ref{claim: weight bounding second fun form}, 
\begin{eqnarray*}
    \int_{co (M_{k})} |\nabla f_{k}|^{2} \leq \int_{co (M_{k})} (|A_{k}|^{2} + |R_{k}|) f_{k}^{2} \leq C \int_{co (M_{k})} f_{k}^{2} \, \omega_{k, \delta, R} = C.
\end{eqnarray*}
Furthermore, 
\begin{equation*}
    \delta^{-2} \|f_{k}\|^{2}_{L^{2} (co (M_{k}))} \leq \int_{co (M_{k})} f_{k}^{2} \, \omega_{k, \delta, R} = 1.
\end{equation*}
Thus, for all $k \in \mathbb{N}$.,
\begin{equation*}
    \|f_{k}\|_{W^{1,2} (co (M_{k}))} \leq C = C (N, g, \delta, R, m, M_{\infty}, \Sigma^{1}, \ldots, \Sigma^{J}) < + \infty.
\end{equation*}
Using Lemma \ref{claim: weight bounding second fun form} we may also obtain a lower bound on the negative eigenvalues, 
\begin{equation}\label{eqn: lower bound on eigenvalue}
    \lambda_{k} = \lambda_{k} \int_{co (M_{k})} f_{k}^{2} \, \omega_{k, \delta, R} = \int_{co (M_{k})} |\nabla f_{k}|^{2} - (|A_{k}|^{2} + R_{k}) f_{k}^{2} \geq - C \int_{co (M_{k})} f_{k}^{2} \, \omega_{k, \delta, R} = - C. 
\end{equation}
Thus, after potentially taking a subsequence and renumerating, we may assume that $\lambda_{k} \rightarrow \lambda_{\infty} \leq 0$.

\bigskip 

We define the map, 
\begin{eqnarray*}
    F \colon co (M_{\infty}) \times \mathbb{R} &\rightarrow& N, \\
    (x, t) &\mapsto& \exp_{\iota_{\infty} (x)} (t \nu_{\infty} (x)).
\end{eqnarray*}
Note, as $M_{\infty}$ is smooth, and properly embedded, there exists a $\tau = \tau (N, M_{\infty}, g) > 0$, such that, 
\begin{equation*}
    F \colon co (M_{\infty}) \times (- \tau, \tau) \rightarrow F(co (M_{\infty}) \times (- \tau, \tau)) \subset N,
\end{equation*}
is a smooth, local diffeomorphism. 
We define the metric $\tilde{g} = F^{*} g$, on $co (M_{\infty}) \times (- \tau, \tau)$, and assume that for all $k \in \mathbb{Z}_{\geq 1}$, $M_{k} \subset F (co (M_{\infty}) \times (-\tau, \tau))$.

\bigskip 

First we consider the case of minimal hypersurfaces.
As $M_{\infty}$ is properly embedded, and $N$ is compact, we may take $\tau > 0$, such that, $F^{-1} (M_{\infty}) = co (M_{\infty}) \times \{0\}$.
For $r > 0$, define the open set $\Omega_{r} \subset co (M_{\infty})$, by, 
\begin{equation*}
    \Omega_{r} \coloneqq \iota_{\infty}^{-1} \left( M_{\infty} \setminus \bigcup_{y \in \mathcal{I}} \overline{B_{r}^{N} (y)} \right) \subset co (M_{\infty}).
\end{equation*}
We define, $M_{k}^{r} \coloneqq M_{k} \cap F (\Omega_{r} \times (- \tau, \tau))$, and $\tilde{M}_{k}^{r} = F^{-1} (M_{k}^{r}) \cap ( \Omega_{r} \times (- \tau, \tau) )$.
By the convergence described in Definition \ref{def: bubble convergence}, along with the the fact that $co(M_{\infty})$ is two-sided in $co (M_{\infty}) \times (- \tau, \tau)$, and $\theta_{|M_{\infty}} \equiv m$, for large enough $k$, there exists $m$ smooth functions,
\begin{equation*}
    u_{k}^{i, r} \colon \Omega_{r} \rightarrow (- \tau, \tau), 
\end{equation*}
such that, $u_{k}^{1, r} < u_{k}^{2, r} < \cdots < u_{k}^{m, r}$, and
\begin{equation*}
    \tilde{M}_{k}^{r} = \bigcup_{i = 1}^{m} \{ (x, u_{k}^{i, r} (x)) \colon x \in \Omega_{r} \} \subset \Omega_{r} \times (- \tau, \tau).
\end{equation*}
We also note that $u_{k}^{i, r} \rightarrow 0$ in $C^{l} (\Omega_{r})$ for all $l \in \mathbb{Z}_{\geq 1}$, and for $0 < r < s$, $u_{k}^{i, r} = u_{k}^{i, s}$ on $\Omega_{s} \subset \Omega_{r}$.
Moreover, we define the metric $g_{k} = \iota^{*}_{k} (g_{|M_{k}})$ on $co (M_{k})$, and the metric $g_{\infty} = \iota^{*}_{\infty} (g_{|M_{\infty}})$ on $co (M_{\infty})$.

\bigskip 

First we consider the case in which $M_{k}$ is one-sided. 
For each connected component of $\tilde{M}_{k}^{r}$, 
\begin{equation*}
    \tilde{M}_{k}^{i, r} \coloneqq \{ (x, u_{k}^{i, r} (x)) \colon x \in \Omega_{r} \},
\end{equation*}  
we denote the $\nu_{k}^{i,r}$ to be the choice of unit normal to $\tilde{M}_{k}^{i, r}$ (with respect to $\tilde{g}$) which points in the positive $\tau$ direction.
Through this choice of unit normal we identify $\Omega_{r}$ as a subset of $co (M_{k})$, by the map (which is a diffeomorphism onto its image)
\begin{equation}\label{eqn: definition of the map F k r}
    \begin{split}
        F_{k}^{i,r} &\colon \Omega_{r} \rightarrow co (M_{k}), \\
        x &\mapsto (F(x, u_{k}^{i,r} (x)), d F (\nu_{k}^{i,r})),
    \end{split}
\end{equation}
and define, 
\begin{equation*}
    \tilde{f}_{k}^{i,r} (x) = (f_{k} \circ F_{k}^{i,r}) (x).
\end{equation*}
We do note that $\tilde{f}_{k}^{i,r}$ depends on the choice of unit normal to $\tilde{M}_{k}^{i, r}$ that we pick, however as $f_{k} \in W^{1,2} (co (M_{k}))^{-}$, this choice is only up to a sign.

\bigskip 

For the case of $M_{k}$ being two-sided, we simply define, 
\begin{equation*}
    \begin{split}
        F_{k}^{i,r} \colon \Omega_{r} \rightarrow co (M_{k}) = M_{k}, \\
        x \mapsto F (x, u_{k}^{i,r} (x)),
    \end{split}
\end{equation*}
and define, 
\begin{equation*}
    \tilde{f}_{k}^{i,r} (x) = (f_{k} \circ F_{k}^{i,r}) (x).
\end{equation*}

\bigskip 

For the case of $H$-CMC hypersurfaces, we have that
\begin{equation*}
    co (M_{\infty}) = \sqcup_{i = 1}^{a} co (M_{\infty}^{i}), 
\end{equation*}
where each $M_{\infty}^{i}$ is a distinct, closed, quasi-embedded $H$-CMC hypersurface such that $co (M_{\infty}^{i})$ is connected.
We have that $\theta^{i} = m_{i} \in \mathbb{Z}$, and we denote, $\Omega_{r / 2}^{i} \subset co (M_{\infty}^{i})$ as before.
Then similarly to before, for each $i = 1, \ldots, a$, and large enough $k$, there exists $m_{i}$ smooth graphs ($j = 1, \ldots, m_{i}$), 
\begin{equation*}
    u_{k}^{j, i, r} \colon \Omega_{r / 2}^{i} \rightarrow (- \tau, \tau), 
\end{equation*}
such that, $u_{k}^{1, i, r} < u_{k}^{2, i, r} < \cdots < u_{k}^{m_{i}, i, r}$, $u_{k}^{j, i, r} \rightarrow 0$ in $C^{l} (\Omega_{r}^{i})$ for all $l \in \mathbb{Z}_{\geq 1}$, and, for large enough $k$,
\begin{equation*}
    M_{k} \setminus \bigcup_{y \in \mathcal{I}} B_{r}^{N} (y) \subset \bigcup_{i = 1}^{a} \bigcup_{j = 1}^{m_{i}} \{ F (x, u_{k}^{j, i, r / 2} (x)) \colon x \in \Omega_{r / 2}^{i} \}.
\end{equation*}
Define $\tilde{M}_{k}^{j, i, r} = \{ (x, u_{k}^{j, i, r} (x)) \colon x \in \Omega_{r}^{i}\}$, and as before we identify this as a subset of $co (M_{k})$, and similarly define the map $F_{k}^{j,i,r}$, and the function $\tilde{f}_{k}^{j,i,r} \in W^{1,2} (\Omega_{r}^{i})$.

\bigskip 

For ease of notation we just consider the case of minimal hypersurfaces.
For an open set $\Omega \subset \subset co (M_{\infty}) \setminus \iota_{\infty}^{-1} (\mathcal{I})$, we may take $r > 0$, such that $\Omega \subset \subset \Omega_{r}$, and then define, for large enough $k$, 
\begin{equation*}
    \tilde{f}_{k}^{i} (x) = \tilde{f}_{k}^{i, r} (x), \, x \in \Omega.
\end{equation*}
Note that this definition is independent of the choice of $0 < r < r_{0}$, for $\Omega \subset \Omega_{r_{0}}$. 
When dealing with a fixed open set $\Omega \subset \subset co (M_{\infty}) \setminus \iota_{\infty}^{-1} (\mathcal{I})$, for appropriate choices of $r$, we drop the superscript $r$ in the notation of the maps $F_{k}^{i,r}$, and functions $u_{k}^{i, r}$.
Then, choosing $k$ large enough (so that $\| u_{k}^{i} \|_{C^{1} (\Omega_{r})}$ is small enough), we have that, 
\begin{equation*}
    \| \tilde{f}_{k}^{i} \|_{W^{1,2} (\Omega)} \leq 2 \| f_{k} \|_{W^{1,2} (co(M_{k}))} \leq C.
\end{equation*}
and thus for each $i = 1, \ldots, m$, there exists an $\tilde{f}^{i}_{\infty} \in W^{1,2}_{\text{loc}} (co (M_{\infty}) \setminus \iota_{\infty}^{-1} (\mathcal{I}))$, such that, after potentially picking a subsequence and renumerating,
\begin{equation}\label{eqn: weak W 1 2 convergence of tilde f k}
    \begin{cases}
        \tilde{f}_{k}^{i} \rightharpoonup \tilde{f}_{\infty}^{i}, & W^{1,2}_{\text{loc}} (co (M_{\infty}) \setminus \iota_{\infty}^{-1} (\mathcal{I})), \\
        \tilde{f}_{k}^{i} \rightarrow \tilde{f}_{\infty}^{i}, & L^{2}_{\text{loc}} (co (M_{\infty}) \setminus \iota_{\infty}^{-1} (\mathcal{I})), \\
    \end{cases}
\end{equation}
Note that by lower semicontinuity of the $W^{1,2}$ norm for (\ref{eqn: weak W 1 2 convergence of tilde f k}), for all open $\Omega \subset co (M_{\infty}) \setminus \iota_{\infty}^{-1} (\mathcal{I})$, we have a uniform bound $\|\tilde{f}_{\infty}^{i}\|_{W^{1,2} (\Omega)} \leq C$.
Therefore, we may deduce that in fact $\tilde{f}_{\infty}^{i} \in W^{1,2} (co (M_{\infty}))$.
Moreover, we have that 
\begin{equation*}
    \int_{co (M_{\infty})} (\tilde{f}_{\infty}^{i})^{2} \, (\omega_{\delta} \circ \iota_{\infty}) \leq 1.
\end{equation*}

\bigskip 

For $i = 1, \ldots, m$, and large enough $k$, we define the metric, $\tilde{g}^{i}_{k} \coloneqq (F^{i}_{k})^{*} g_{k}$, and its associated gradient $\tilde{\nabla}^{i}_{k}$, on $\Omega$. 
Let $J_{k}^{i}$ denote the Jacobian of the map $F_{k}^{i}$ with respect to the metric $g_{\infty}$ on $\Omega$. 
For a point $x_{0} \in \Omega$, we may choose $s > 0$, small enough so that $B_{s}^{co (M_{\infty})} (x_{0}) \subset \subset \Omega$,
\begin{equation*}
    F (B_{s}^{co (M_{\infty})} (x_{0}) \times (-\tau, \tau)) \subset B_{\text{inj} \, (N) / 2}^{N} (\iota (x_{0})).
\end{equation*}
Consider $\varphi \in C_{c}^{\infty} (B_{s}^{co (M_{\infty})} (x_{0}))$, and for each $i = 1, \ldots, m$ and $x \in B_{s}^{co (M_{\infty})} (x_{0})$, we define the function 
\begin{equation*}
    \varphi_{k}^{i} (F_{k}^{i} (x)) = \varphi (x),
\end{equation*}
on $C_{c}^{\infty} (F_{k}^{i} (B_{s}^{co (M_{\infty})} (x_{0}))) \subset C^{\infty} (co (M_{k}))$.
As each $M_{k}$ is properly embedded, by \cite[Lemma C.1]{CKM-MinimalHypersurfacesBoundedIndex} (cf. \cite{S-OrientabilityOfHypersurfaces}), 
\begin{equation*}
    \{ \iota_{k} (F_{k}^{i} (x)) \colon x \in B_{s}^{co (M_{\infty})} (x_{0}) \} \subset M_{k} \cap B_{\text{inj} \, (N) / 2}^{N} (\iota (x_{0}))
\end{equation*} 
is two-sided, and thus we can extend $\varphi_{k}^{i}$ to a vector field on $N$, and thus to a function in $C^{\infty} (co(M_{k}))^{-}$.
Thus we may plug $\varphi_{k}^{i}$ into (\ref{eqn: weak weighted eigenvalue equation}) and obtain, 
\begin{equation*}
    \int_{\Omega} g_{k}^{i} (\tilde{\nabla}^{i}_{k} \tilde{f}^{i}_{k}, \tilde{\nabla}^{i}_{k} \varphi ) J_{k}^{i} = \lambda_{k} \int_{\Omega} \tilde{f}_{k}^{i} \, \varphi \, (\omega_{k, \delta, R} \circ F_{k}^{i}) \, J_{k}^{i} + \int_{\Omega} ((|A_{k}|^{2} + R_{k}) \circ F_{k}^{i}) \, \tilde{f}_{k}^{i} \, \varphi \, J_{k}^{i}.
\end{equation*}
Hence, by (\ref{eqn: weak W 1 2 convergence of tilde f k}), convergence of $\omega_{k, \delta, R} \circ F_{k}^{i} \rightarrow \omega_{\delta} \circ \iota_{\infty}$, in $L^{\infty} (\Omega)$, and smooth convergence of, $F_{k}^{i} \rightarrow id$, on $\Omega$, we have that, 
\begin{equation}\label{eqn: tilde f infty weakly solves pde away from I}
    \int_{co (M_{\infty})} \nabla \tilde{f}^{i}_{\infty} \cdot \nabla \varphi - ((|A_{\infty}|^{2} + R_{\infty}) \circ \iota_{\infty}) \, \tilde{f}^{i}_{\infty} \, \varphi = \lambda_{\infty} \int_{co (M_{\infty})} \tilde{f}^{i}_{\infty} \, \varphi \, (\omega_{\delta} \circ \iota_{\infty}),
\end{equation}
holds for all $\varphi \in C_{c}^{\infty} (B_{s}^{co (M_{\infty})} (x_{0}))$. 
Thus by standard regularity theory for linear elliptic PDEs we have that $\tilde{f}^{i}_{\infty} \in W^{2,2} (\Omega)$ and $\Delta \tilde{f}^{i}_{\infty} + ((|A_{\infty}|^{2} + R_{\infty}) \circ \iota_{\infty}) \, \tilde{f}_{\infty}^{i} + \lambda_{\infty} \tilde{f}_{\infty}^{i} \, (\omega_{\delta} \circ \iota_{\infty}) = 0$ a.e. on $\Omega$. 
This then implies that (\ref{eqn: tilde f infty weakly solves pde away from I}) holds for all $\varphi \in C^{\infty}_{c} (co (M_{\infty}) \setminus \iota_{\infty}^{-1} (\mathcal{I}))$. 

\begin{proposition}\label{prop: u solves PDE away from I, then it solves it over I}
    Let $(M, g)$ be a compact, $n$-dimensional, Riemannian manifold, with $n \geq 3$. 
    Consider $V \in L^{\infty} (M)$, and $\omega \in L(n/2, \infty) (M)$. 
    Suppose that we have $u \in W^{1,2} (M)$, and a finite set of points $\mathcal{J} \subset M$, such that, for all $\varphi \in C^{\infty}_{c} (M \setminus \mathcal{J})$,
    \begin{equation}\label{eqn: u weakly solves away from I}
        \int_{M} \nabla u \cdot \nabla \varphi - V u \varphi - \omega u \varphi = 0,
    \end{equation}
    then in fact (\ref{eqn: u weakly solves away from I}) holds for all $\varphi \in W^{1,2} (M)$.
\end{proposition} 

\begin{proof}
    Consider a smooth function on $\mathbb{R}$ with the following properties, 
    \begin{equation*}
        \begin{cases}
            \chi (t) = 1, & t \leq 1, \\
            \chi (t) = 0, & t \geq 2, \\
            -3 \leq \chi' (t) \leq 0.
        \end{cases}
    \end{equation*}
    Then for $\varepsilon > 0$, chosen small enough we define the following smooth function on $M$, 
    \begin{equation*}
        \chi_{\varepsilon} (x) = \chi \left( \frac{d_{\mathcal{J}} (x)}{\varepsilon} \right),
    \end{equation*}
    where $d_{\mathcal{J}} (x) = \text{dist}_{g} (x, \mathcal{J})$.
    Then, for $\varphi = \chi_{\varepsilon} \varphi + (1 - \chi_{\varepsilon}) \varphi \in C^{\infty} (M)$, by H\"{o}lder inequality, Proposition \ref{prop: W 1 2 bounds on weighted L2 inner product}, and small enough $\varepsilon > 0$,
    \begin{equation*}
        \left| \int_{M} \nabla u \cdot \nabla \varphi - V u \varphi - \omega u \varphi \right| \leq C \|\varphi\|_{C^{1} (M)} \|u\|_{W^{1,2} (M)} (1 + \|V\|_{L^{\infty} (M)} + \|\omega\|_{(n/2, \infty) (M)}) \varepsilon^{(n - 2) / 2}
    \end{equation*} 
    Thus we see that (\ref{eqn: u weakly solves away from I}) holds for all $\varphi \in C^{\infty} (M)$, and the Proposition may then be completed by a standard density argument.
\end{proof}

By Proposition \ref{prop: u solves PDE away from I, then it solves it over I} we have that (\ref{eqn: tilde f infty weakly solves pde away from I}) holds for all $\varphi \in W^{1,2} (co (M_{\infty}))$. 
Thus, 
\begin{equation*}
    \tilde{f}^{i}_{\infty} \in \mathcal{E}_{\omega_{\delta}} (\lambda_{\infty}; L_{\infty}, W^{1,2} (co (M_{\infty}))).
\end{equation*}
It is worth noting that we could have $\tilde{f}^{i}_{\infty} = 0$.

\subsection{Convergence on the Bubble}\label{sec: convergence on the bubble}

For $S > 0$ fixed, and $i \in \{ 1, \ldots, J \}$, consider the bubble $\Sigma_{k}^{i, S}$, and its associated point-scale sequence, $\{ (p_{k}^{i}, r_{k}^{i}) \}_{k \in \mathbb{N}}$.
Let $\{\partial_{1}, \ldots, \partial_{n + 1}\}$ be an orthonormal basis for $T_{p_{k}^{i}} N$, with respect to the metric $g$, and define the map, 
\begin{eqnarray*}
    G_{k}^{i} \colon \mathbb{R}^{n + 1} = \text{span} \{ \partial_{1}, \ldots, \partial_{n + 1} \} &\rightarrow& N, \\
    x &\mapsto& \exp_{p_{k}^{i}} (r_{k}^{i} x),
\end{eqnarray*}
then on $B_{2S}^{n + 1} (0)$, for large enough $k$ we define the metric, $\tilde{g}_{k} = (r_{k}^{i})^{-2} (G_{k}^{i})^{*} g$, and we have that, 
\begin{equation*}
    (\tilde{g}_{k})_{\alpha, \beta} (x) \coloneqq g_{\alpha, \beta} (r_{k}^{i} x) \rightarrow \delta_{\alpha, \beta},
\end{equation*}
and for our bubble,
\begin{equation*}
    \tilde{\Sigma}^{i, 2 S}_{k} \coloneqq \frac{1}{r_{k}^{i}} \exp_{p_{k}^{i}}^{-1}( \Sigma^{i, S}_{k} \cap B_{2 S r_{k}^{i}}^{N} (p_{k}^{i}) ) \rightarrow \Sigma^{i} \cap B_{2 S}^{n + 1} (0) \eqqcolon \Sigma^{i, 2 S},
\end{equation*}
smoothly. 
As $\Sigma^{i}$ is two-sided there is a choice of unit normal $\nu$, and a $\tau > 0$, such that the map, 
\begin{eqnarray*}
    F \colon \Sigma^{i, S} \times (-\tau, \tau) &\rightarrow& \mathbb{R}^{n + 1}, \\
    (x, t) &\mapsto& x + t \nu (x),
\end{eqnarray*}
is a diffeomorphism onto its image. 
Then for large enough $k$, there exists a smooth function, 
\begin{equation*}
    v_{k}^{i, S} \colon \Sigma^{i, S} \rightarrow (-\tau, \tau),
\end{equation*}
such that, 
\begin{equation*}
    \tilde{\Sigma}^{i, 2 S}_{k} \cap F(\Sigma^{i, S} \times (- \tau, \tau)) = \{ F(x, v_{k}^{i, S} (x)) \colon x \in \Sigma^{i, S} \}.
\end{equation*}
As before, for large enough $k$, define the smooth map, 
\begin{eqnarray*}
    F_{k}^{i, S} \colon \Sigma^{i, S} &\rightarrow& \mathbb{R}^{n + 1}, \\
    x &\mapsto& F(x, v_{k}^{i, S} (x)).
\end{eqnarray*}
From $\nu$, we get a choice of unit normal $\nu_{k}^{i}$ (which points in the $d F (\partial_{t})$ direction), to $\tilde{\Sigma}_{k}^{i, 2 S} \cap F (\Sigma^{i, S} \times (- \tau, \tau))$, with respect to $\tilde{g}_{k}$.

\bigskip 

If $M_{k}$ is one-sided, then we define the following functions on $\Sigma^{i, S}$ (recalling our functions $f_{k}$ from Section \ref{sec: convergence on the base}),
\begin{equation*}
    \tilde{f}_{k}^{i, S} (x) \coloneqq (r_{k}^{i})^{n / 2 - 1} f_{k} ((G_{k}^{i} \circ F_{k}^{i}) (x), (r_{k}^{i})^{-1} d G_{k}^{i} ( \nu_{k}^{i} (F_{k}^{i} (x)) ))
\end{equation*}
and, 
\begin{equation*}
    \omega_{k}^{\Sigma^{i}, R, S} (x) \coloneqq (r_{k}^{i})^{2} \omega_{k, \delta, R} ((G_{k}^{i} \circ F_{k}^{i}) (x), (r_{k}^{i})^{-1} d G_{k}^{i} (\nu_{k}^{i} (F_{k}^{i} (x)) )).
\end{equation*}
We note that while $\tilde{f}_{k}^{i, S}$ depends on our choice of unit normal $\nu$ to $\Sigma^{i}$, this dependence is only up to a choice in sign.

\bigskip 

If $M_{k}$ is two sided, the we define the following functions on $\Sigma^{i, S}$,
\begin{equation*}
    \tilde{f}_{k}^{i, S} (x) \coloneqq (r_{k}^{i})^{n / 2 - 1} f_{k} ((G_{k}^{i} \circ F_{k}^{i}) (x))
\end{equation*}
and, 
\begin{equation*}
    \omega_{k}^{\Sigma^{i}, R, S} (x) \coloneqq (r_{k}^{i})^{2} \omega_{k, \delta, R} ((G_{k}^{i} \circ F_{k}^{i}) (x)).
\end{equation*}
We then have that there exists a function $\omega^{\Sigma^{i}, R} \in W^{1, \infty} (\Sigma^{i})$, such that $\omega_{k}^{\Sigma^{i}, R, S} \rightarrow \omega^{\Sigma^{i}, R}$ in $L^{\infty} (\Sigma^{i, S})$, and moreover satisfies the following: there exists a $T < + \infty$, such that, $\essinf \omega^{\Sigma^{i}, R} > 0$ in $B_{T}^{n + 1} (0) \cap \Sigma^{i}$, and there exists a $\Lambda \geq 1$, such that, 
\begin{equation*}
    \frac{1}{\Lambda |x|^{2}} \leq \omega^{\Sigma^{i}, R} (x) \leq \frac{\Lambda}{|x|^{2}}
\end{equation*}
on $\Sigma^{i} \setminus B_{T}^{n + 1} (0)$.

\bigskip 

We have, for large $k$,
\begin{equation*}
    \int_{\Sigma^{i, S}} (\tilde{f}_{k}^{i, S})^{2} \omega_{k}^{\Sigma^{i}, R, S} \leq 2 \int_{co (M_{k})} f_{k}^{2} \omega_{k, \delta, R} = 2, 
\end{equation*}
implying that, for large enough $k$, 
\begin{equation*}
    \int_{\Sigma^{i, S}} (\tilde{f}_{k}^{i, S})^{2} \leq \frac{1}{\min_{\Sigma^{i, S}} \omega^{\Sigma^{i}, R}} \int_{\Sigma^{i, S}} (\tilde{f}_{k}^{i, S})^{2} \omega_{k, R}^{\Sigma^{i}, S} \leq C (S, R, \Sigma^{1}, \ldots, \Sigma^{J}) < + \infty.
\end{equation*}
Moreover, 
\begin{equation*}
    \int_{\Sigma^{i, S}} |\nabla \tilde{f}_{k}^{i, S}|^{2} \leq 2 \int_{co(M_{k})} |\nabla f_{k}|^{2} \leq C (N, g, \delta, R, M_{\infty}, \Sigma^{1}, \ldots, \Sigma^{J}) < + \infty.
\end{equation*}
Similar to previous, for $0 < S_{1} < S_{2} < + \infty$, and large enough $k$, we have that $\tilde{f}_{k}^{i, S_{1}} = \tilde{f}_{k}^{i, S_{2}}$ on $\Sigma^{i, S_{1}}$. 
Thus, for any open, bounded set $\Omega \subset \Sigma^{i}$, we may take any $S > 0$ such that $\Omega \subset \Sigma^{i, S}$, then for large enough $k$, the function $\tilde{f}_{k}^{i} = \tilde{f}_{k}^{i, S}$ is well defined on $\Omega$, with, 
\begin{equation*}
    \int_{\Omega} (\tilde{f}_{k}^{i})^{2} + \int_{\Omega} |\nabla \tilde{f}_{k}^{i}|^{2} \leq C (\Omega, N, g, \delta, R, M_{\infty}, \Sigma^{1}, \ldots, \Sigma^{J}).
\end{equation*}
We may conclude that there exists an $\tilde{f}_{\infty}^{i} \in W^{1,2}_{\text{loc}} (\Sigma^{i})$,  
\begin{equation*}
    \begin{cases}
        \tilde{f}_{k}^{i} \rightarrow \tilde{f}_{\infty}^{i}, & L^{2}_{\text{loc}} (\Sigma^{i}), \\
        \tilde{f}_{k}^{i} \rightharpoonup \tilde{f}_{\infty}^{i}, & W^{1, 2}_{\text{loc}} (\Sigma^{i}),
    \end{cases}
\end{equation*}
and, we have
\begin{equation*}
    \tilde{f}_{\infty}^{i} \in W^{1,2}_{\omega^{\Sigma^{i}, R}} (\Sigma^{i}) \coloneqq \left\{ f \in W^{1,2}_{\text{loc}} (\Sigma^{i}) \colon \int_{\Sigma^{i}} |\nabla f|^{2} < + \infty, \, \int_{\Sigma^{i}} \omega^{\Sigma^{i}, R} f^{2} < + \infty \right\}
\end{equation*}
Similar to before, and the fact that our metric on $B_{S}^{n + 1} (0)$ converges to the standard Euclidean one, we deduce that for all $\varphi \in C^{\infty}_{c} (\Sigma^{i})$, 
\begin{equation}\label{eqn: tilde f weakly solves pde on sigma}
    \int_{\Sigma^{i}} \nabla \tilde{f}_{\infty}^{i} \cdot \nabla \varphi - |A_{\Sigma^{i}}|^{2} \, \tilde{f}_{\infty}^{i} \, \varphi - \lambda_{\infty} \, \omega^{\Sigma^{i}, R} \, \tilde{f}_{\infty}^{i} \, \varphi = 0.
\end{equation}
As $\Sigma^{i}$ has finite total curvature, we may deduce, by Proposition \ref{prop: curvature estimate}, that there exists a $C = C (\Sigma^{i}) < + \infty$, such that, 
\begin{equation*}\label{eqn: omega curvature estimate on ends}
    |A_{\Sigma^{i}}|^{2} \leq C \omega^{\Sigma^{i}, R}.
\end{equation*}
Thus by H\"{o}lder's inequality we have, for $\varphi \in L^{2} (\Sigma^{i})$,
\begin{equation*}
    \left| \int_{\Sigma^{i}} |A_{\Sigma^{i}}|^{2} \, \tilde{f}_{\infty}^{i} \, \varphi + \omega^{\Sigma^{i}, R} \, \tilde{f}_{\infty}^{i} \, \varphi \right| \leq C \| \omega^{\Sigma^{i}, R} \|_{L^{\infty} (\Sigma^{i})}^{1/2} \, \| \varphi \|_{L^{2} (\Sigma^{i})}.
\end{equation*}
This allows us to apply a standard density argument to deduce that (\ref{eqn: tilde f weakly solves pde on sigma}) holds for all $\varphi \in W^{1,2} (\Sigma^{i})$.
Thus we have that 
\begin{equation*}
    \tilde{f}_{\infty}^{i} \in \mathcal{E}_{\omega^{\Sigma^{i}, R}} (\lambda_{\infty}; L_{\Sigma^{i}}, W^{1,2}_{\omega^{\Sigma^{i}, R}} (\Sigma^{i}), W^{1,2} (\Sigma^{i})).
\end{equation*}  
Again it is worth noting that we could have $\tilde{f}_{\infty}^{i} = 0$.

\begin{rem}\label{rem: L 2 star is contained in W dot 1 2}
    By the Michael--Simon--Sobolev inequality (\cite[Theorem 2.1]{MS-Michael-Simon-Sobolev-inequality}), one may deduce that $\tilde{f}_{\infty} \in L^{2^{*}} (\tilde{\Sigma})$, for $2^{*} = 2n / (n - 2)$.
\end{rem}

\subsection{Strict Stability of the Neck}\label{sec: Strict Stability of the neck}

For a Riemannian manifold $(M, g)$, we define $W_{0}^{1,2} (M, g)$ to be the closure of $C^{1}_{c} (M)$, with respect to the standard norm on $W^{1,2} (M, g)$.

\begin{lem}\label{lem: Strict stability of the neck}
    For $n \geq 3$, consider the cylinder $(A \times \mathbb{R}, g)$, where $A \subset \mathbb{R}^{n}$ is a non-empty, open set, and $g$ is a smooth Riemannian metric, such that there exists a constant $K \in [1, \infty)$ such that for $x \in A \times \mathbb{R}$, and $X \in \mathbb{R}^{n + 1}$,
    \begin{equation}\label{eqn: upper and lower bounds on metric on cylinder}
            \frac{1}{K} \langle \, X , \, X \, \rangle \leq g_{x} (X, X) \leq K \langle \, X \, , \, X \, \rangle,
    \end{equation} 
    where $\langle \, \cdot \, , \, \cdot \, \rangle$ is the standard metric on $\mathbb{R}^{n}$.
    Now suppose we have a smooth function
    \begin{equation*}
        u \colon A \rightarrow (- T, T),
    \end{equation*}
    such that $\| \nabla^{\mathbb{R}^{n}} u\|_{L^{\infty} (A)} \leq 1 / 2$ ($\nabla^{\mathbb{R}^{n}}$ denotes the gradient on $\mathbb{R}^{n}$ with respect to the standard Euclidean metric), and denote $M \coloneqq \text{graph}  \, (u) \subset A \times (- 2 \, T, 2 \, T)$.
    For fixed $W \in (0, \infty)$ suppose we have functions, $\omega \in L(n / 2, \infty) (M, g)$, and $V \in L^{\infty} (M)$, such that $\| \omega \|_{L (n / 2, \infty) (M, g)} \leq W$, and $\essinf \omega > 0$.
    Then, there exists an $\varepsilon = \varepsilon (n, K, W) > 0$, and $C = C (n, K) \in (0, + \infty)$ such that if $|V| \leq \varepsilon \omega$ on $M$, then
    \begin{equation*}
        0 < \frac{1}{C \, W}  \leq \inf \Biggl\{ \int_{M} |\nabla f|^{2} - V f^{2}  \colon f \in W_{0}^{1,2} (M), \, \int_{M} f^{2} \omega = 1 \Biggr\}.
    \end{equation*}
\end{lem}

\begin{proof} 
    We define the following map, $F (x) \coloneqq (x, u (x))$, and the metric $\tilde{g} = F^{*} g$ on $A$.
    For $x \in A$, and $X \in \mathbb{R}^{n}$, we have, 
    \begin{eqnarray*}
        \tilde{g}_{x} (X, X) \leq K \langle X, X \rangle + 2 K \|\nabla^{\mathbb{R}^{n}} u\|_{L^{\infty} (A)} \langle X, X \rangle + K \|\nabla^{\mathbb{R}^{n}} u\|_{L^{\infty} (A)}^{2} \langle X, X \rangle = C \langle X, X \rangle,
    \end{eqnarray*}
    with $C = C (n, K)$, which from here on may be rechosen at each step.
    Thus, for $x \in A$, and $f \in C^{1} (M)$,
    \begin{eqnarray*}
        |\nabla^{\mathbb{R}^{n}} (f \circ F) (x)| = \sup_{\langle X, X \rangle \leq 1} |d (f \circ F) (x) (X)| \leq \sup_{\tilde{g}_{x} (X, X) \leq C} |d (f \circ F) (x) (X)| \leq C |\nabla^{\tilde{g}} (f \circ F)| (x).
    \end{eqnarray*}   
    Moreover, if we consider the metric on $A$, $g_{1} = F^{*} \langle \cdot, \cdot \rangle$, then $g_{1} (X, X) \leq K \tilde{g} (X, X) \leq C \langle X, X \rangle$, and thus
    \begin{eqnarray*}
        1 \leq 1 + |\nabla^{\mathbb{R}^{n}} u (x)|^{2} = |g_{1}| \leq K^{n} |\tilde{g}| \leq C.
    \end{eqnarray*}
    Now, take an $f \in C^{1}_{c} (M)$, that satisfies, 
    \begin{equation*}
        \int_{M} f^{2} \, \omega \, d \mu = 1,    
    \end{equation*}
    where $\mu$ denotes the volume measure of $(M, g)$.
    We have, by the above, and applying Propositions \ref{prop: Holder Lorentz inequality}, \ref{prop: Lorentz norm of function raised to a power} and \ref{prop: Lorentz Sobolev inequality on R n}, for large enough $k$ (see \cite[Theorem 1.1]{V-NoteOnGeneralizedHardySobolevInequality} for a similar computation)
    \begin{eqnarray*}
        1 &=& \int_{A} (f \circ F)^{2} \, (\omega \circ F) \sqrt{|\tilde{g}|} \, dx, \\
        &\leq& C \int_{A} (f \circ F)^{2} \, (\omega \circ F) \, dx, \\
        &\leq& C \|\omega \circ F\|_{(n/2, \infty) (A, \langle \, \cdot \, , \, \cdot \, \rangle)} \|f \circ F\|_{(2^{*}, 2) (A, \langle \, \cdot \, , \, \cdot \, \rangle)}^{2}, \\
        &\leq& C \|\omega \circ F\|_{(n/2, \infty) (A, \langle \, \cdot \, , \, \cdot \, \rangle)} \int_{A} |\nabla^{\mathbb{R}^{n}} (f \circ F)|^{2} \, dx, \\
        &\leq& C \|\omega \circ F\|_{(n/2, \infty) (A, \langle \, \cdot \, , \, \cdot \, \rangle)} \int_{M} |\nabla f|^{2} - V f^{2} \, d \mu + C \|\omega \circ F\|_{(n/2, \infty) (A, \langle \, \cdot \, , \, \cdot \, \rangle)} \int_{M} V f^{2} \, d \mu, \\
        &\leq& C \|\omega \circ F\|_{(n/2, \infty) (A, \langle \, \cdot \, , \, \cdot \, \rangle)} \left( \varepsilon + \int_{M} |\nabla f|^{2} - V f^{2} \, d \mu \right),
    \end{eqnarray*}
    were again we are potentially rechoosing $C = C (n, K)$ at each line.
    Moreover, 
    \begin{equation*}
        \| \omega \circ F \|_{(n / 2, \infty) (A, \langle \, \cdot \, , \, \cdot \, \rangle)} \leq K \| \omega \|_{(n / 2, \infty) (M, g)} \leq K \, W,
    \end{equation*}
    and thus, again rechoosing $C = C (n , K)$, and choosing $\varepsilon = (2 C W)^{-1}$,
    \begin{equation*}
        \int_{M} |\nabla f|^{2} - V f^{2} \geq \frac{1}{2 \, C \, W} > 0,
    \end{equation*}
    for all $f \in C_{c}^{1} (M)$.
    Now the Lemma may be concluded by a standard density argument.
\end{proof}

As previously, for ease of notation we only consider the case of minimal hypersurfaces, however the argument for $H$-CMC hypersurfaces is identical.

\bigskip 

In Sections \ref{sec: convergence on the base} and \ref{sec: convergence on the bubble} we showed that if we have a sequence, 
\begin{equation*}
    f_{k} \in \mathcal{E}_{\omega_{k, \delta, R}} (\lambda_{k}; L_{k}, W^{1,2} (co (M_{k}))^{-}),
\end{equation*}
with $\lambda_{k} \leq 0$, and
\begin{equation*}
    \int_{co (M_{k})} f_{k}^{2} \, \omega_{k, \delta, R} = 1,
\end{equation*}
for all $k$, then after potentially taking a subsequence and renumerating we have that $\lambda_{k} \rightarrow \lambda_{\infty} \leq 0$, and, 
\begin{equation*}
    f_{k} \rightarrow ((f_{\infty}^{1}, \ldots, f_{\infty}^{m}), f_{\infty}^{\Sigma^{1}}, \ldots, f_{\infty}^{\Sigma^{J}}), 
\end{equation*}
where, for $i = 1, \ldots, m$, $f_{\infty}^{i} \in W^{1,2} (co (M_{\infty}) )$, and for $j = 1, \ldots, J$, $f_{\infty}^{\Sigma^{j}} \in W_{\omega^{\Sigma^{j}, R}}^{1,2} (\Sigma^{j}) \subset L^{2^{*}} (\Sigma^{j})$. 
By this convergence we mean, for $i = 1, \ldots, m$ (recalling notation from Section \ref{sec: convergence on the base}), $\tilde{f}_{k}^{i} = (f_{k} \circ F_{k}^{i})$,
\begin{equation}
    \begin{cases}
        \tilde{f}_{k}^{i} \rightharpoonup f_{\infty}^{i}, & W^{1,2}_{\text{loc}} (co (M_{\infty}) \setminus \iota^{-1}_{\infty} (\mathcal{I}) )), \\
        \tilde{f}_{k}^{i} \rightarrow f_{\infty}^{i}, & L^{2}_{\text{loc}} (co (M_{\infty}) \setminus \iota^{-1}_{\infty} (\mathcal{I}))),
    \end{cases}
\end{equation}
and for $j = 1, \ldots, J$, setting $\tilde{f}_{k}^{\Sigma^{j}} = \tilde{f}_{k}^{j}$ (this time $\tilde{f}_{k}^{j}$ as defined in Section \ref{sec: convergence on the bubble}),
\begin{equation}
    \begin{cases}
        \tilde{f}_{k}^{\Sigma^{j}} \rightharpoonup f_{\infty}^{\Sigma^{j}}, & W^{1,2}_{\text{loc}} (\Sigma^{j}), \\
        \tilde{f}_{k}^{\Sigma^{j}} \rightarrow f_{\infty}^{\Sigma^{j}}, & L^{2}_{\text{loc}} (\Sigma^{j}).
    \end{cases}
\end{equation}
Furthermore, we are able to deduce that for $i = 1, \ldots, m$, 
\begin{equation*}
    f_{\infty}^{i} \in \mathcal{E}_{\omega_{\delta}} (\lambda_{\infty}; L_{\infty}, W^{1,2} (co (M_{\infty}) )),
\end{equation*}
and $j = 1, \ldots, J$,
\begin{equation*}
    f_{\infty}^{\Sigma^{j}} \in \mathcal{E}_{\omega^{\Sigma^{j}, R}} (\lambda_{\infty}; L_{\Sigma^{j}}, W^{1,2}_{\omega^{\Sigma^{j}, R}} (\Sigma^{j}), W^{1,2} (\Sigma^{j})).
\end{equation*}

\begin{claim}\label{claim: we cannot converge to zero}
    If $\{ f_{1,k}, \ldots, f_{b,k}\}$, is an orthonormal collection, with respect to the $\omega_{k, \delta, R}$-weighted $L^{2}$ norm, of $\omega_{k, \delta, R}$-weighted eigenfunctions, with non-positive eigenvalues, and $\sum_{i = 1} a_{i}^{2} = 1$, then, 
    \begin{equation*}
        h_{k} \coloneqq \sum_{i = 1}^{b} a_{i} f_{i,k} \rightarrow \sum_{i = 1}^{b} a_{i} ( (f_{i,\infty}^{1}, \ldots, f_{i, \infty}^{m}), f_{i, \infty}^{\Sigma^{1}}, \ldots, f_{i, \infty}^{\Sigma^{J}}) \not= ( (0, \ldots, 0), 0, \ldots, 0).
    \end{equation*}
\end{claim}

\begin{proof}
This claim is proven by a contradiction argument similar to that in \cite[Claim 1 of Lemma \rom{4}.6]{DGR-morse-index-stability}.

\bigskip 

By Remark \ref{rem: neck region can be written as a graph}, for all $\eta \in (0,1/2]$, $\tau \in (0, 1)$, and $y \in \mathcal{I}$, there exists an $r_{0} \in (0, \delta / 4)$, and $R_{0} \in (4 R, \infty)$, such that, taking geodesic normal coordinates about $y \in N$, such that $T_{y} M_{\infty} = \{ x_{n + 1} = 0\}$, for large enough $k$, if we denote $C_{k}$ to be a connected component of
\begin{equation*}
    \exp_{y}^{-1} (M_{k} \cap B_{r_{0}}^{N} (y) \setminus \cup_{j = 1}^{J} \overline{\Sigma_{k}^{j, R_{0}}}) \subset \mathbb{R}^{n + 1},
\end{equation*}
then there exists a non-empty open set $A (C_{k}) \subset \{ x_{n + 1} = 0\}$, and a smooth function,
\begin{equation*}
    u_{k} \colon A(C_{k}) \rightarrow (- \tau, \tau),
\end{equation*}
such that, $C_{k} = \text{graph} \, (u_{k})$, and $\| \nabla^{\mathbb{R}^{n}} u_{k} \|_{L^{\infty}} \leq \eta$. 
Moreover, for each $\varepsilon > 0$, we may make further choices of $r_{0}$ and $R_{0}$, such that for large enough $k$, on each $C_{k}$, if we denote $V_{k} = (|A_{k}|^{2} + R_{k})_{|C_{k}} \in L^{\infty} (C_{k})$, and $\omega_{k} = (\omega_{k, \delta, R})_{|C_{k}} \in L (n/ 2, \infty) (C_{k}, g) \cap L^{\infty} (C_{k})$, then
\begin{equation*}
    \|\omega_{k}\|_{L (n/ 2, \infty) (C_{k}, g)} \leq W, \hspace{1cm} \text{and} \hspace{1cm} |V_{k}| \leq \varepsilon \omega_{k},  
\end{equation*}
with $W = W (N, g, M_{\infty}, m, \delta, J) < + \infty$ (coming from Claim \ref{claim: uniform L n over 2 infinity bounf on weight}).
Thus, fixing $\varepsilon = \varepsilon (N, g, M_{\infty}, m, \delta, J) > 0$ small enough, by Lemma \ref{lem: Strict stability of the neck}, for large enough $k$
\begin{equation}\label{eqn: strict stability of C k}
        \inf \Biggl\{ \int_{C_{k}} |\nabla f|^{2} - V_{k} f^{2}  \colon f \in W_{0}^{1,2} (C_{k}), \, \int_{C_{k}} f^{2} \omega_{k} = 1 \Biggr\} \geq \gamma > 0,
\end{equation}
with $\gamma = \gamma (N, g, M_{\infty}, m, \delta, J) > 0$.
Now assuming Claim \ref{claim: we cannot converge to zero} does not hold, we will provide a contradiction to (\ref{eqn: strict stability of C k}), on at least one connected component $C_{k}$ of $M_{k} \cap (\cup_{y \in \mathcal{I}} B_{r_{0}}^{N} (y) \setminus \cup_{j = 1}^{J} \overline{\Sigma_{k}^{j, R_{0}}})$.
Note that there is a uniform bound, $m |\mathcal{I}| < + \infty$, on the number of such connected components.

\bigskip 

Consider the following smooth cutoff, 
\begin{equation*}
    \begin{cases}
        \chi \in C^{\infty} (\mathbb{R}; [0,1]), \\
        \chi(t) = 1, & t \in (-\infty,1], \\
        \chi(t) = 0, & t \in [2, \infty), \\
        -3 \leq \chi'(t) \leq 0,
    \end{cases}
\end{equation*}
and the following distance functions defined on $co (M_{k})$, 
\begin{equation*}
    d_{k} (x) = \text{dist}_{g}^{N} (\iota_{k} (x), \mathcal{I}),
\end{equation*}
and for $j = 1, \ldots, J$, 
\begin{equation*}
    d_{k}^{j} (x) = \text{dist}_{g}^{N} (\iota_{k} (x), p_{k}^{j}).
\end{equation*}
We then define the following function 
\begin{equation*}
    H_{k} (x) = \begin{cases}
        0, & x \in co (M_{k}) \setminus \cup_{y \in \mathcal{I}} \iota_{k}^{-1} (B^{N}_{r_{0}} (y)), \\
        \chi (2 r_{0}^{-1} d_{k} (x)) h_{k}, & x \in \iota_{k}^{-1} (\cup_{y \in \mathcal{I}} B^{N}_{r_{0}} (y) \setminus \cup_{j = 1}^{J} \Sigma_{k}^{j, 2 R_{0}}), \\
        h_{k} (1 - \chi ((R_{0} r_{k}^{j})^{-1} d_{k}^{j} (x))), & x \in \iota_{k}^{-1} (\Sigma_{k}^{j,2 R_{0}}), \, j = 1, \ldots, J.
    \end{cases}
\end{equation*}
We compute the gradient of $H_{k}$. 
For $x \in co (M_{k}) \setminus \cup_{y \in \mathcal{I}} \iota_{k}^{-1} (B^{N}_{r_{0}} (y))$, or $x \in \iota_{k}^{-1} (\Sigma_{k}^{j, R_{0}})$, $j = 1, \ldots, J$, we have, $\nabla H_{k} = 0$. 
For $x \in co (M_{k}) \cap\iota_{k}^{-1} (B^{N}_{r_{0}} (y) \setminus \cup_{j = 1}^{J} \Sigma_{k}^{j, 2 R_{0}})$, we have, 
\begin{equation*}
    \nabla H_{k} (x) = \nabla h_{k} (x) \chi (2 r_{0}^{-1} d_{k} (x)) + h_{k} (x) 2 r_{0}^{-1} \chi' (2 r_{0}^{-1} d_{k} (x)) \nabla d_{k} (x).
\end{equation*}
Finally, for $x \in \iota_{k}^{-1} (\Sigma_{k}^{j, 2 R_{0}} \setminus \Sigma_{k}^{j, R_{0}})$, $j = 1, \ldots, J$, we have, 
\begin{equation*}
    \nabla H_{k} (x) = \nabla h_{k} (x) (1 - \chi ((R r_{k}^{j})^{- 1} d_{k}^{j} (x))) - h_{k} (x) (R_{0} r_{k}^{j})^{-1} \chi' ((R_{0} r_{k}^{j})^{-1} d_{k}^{j} (x)) \nabla d_{k}^{j} (x). 
\end{equation*}
Noting that $H_{k}^{2} \leq h_{k}^{2}$, we have
\begin{equation}\label{eqn: difference of B k f k to B k F k}
    |B_{k} (h_{k}, h_{k}) - B_{k} (H_{k}, H_{k})| \leq \left| \int_{co (M_{k})} |\nabla h_{k}|^{2} - |\nabla H_{k}|^{2} \right| + C \int_{co (M_{k})} \omega_{k, \delta, R} (h_{k}^{2} - H_{k}^{2}),
\end{equation}
with $C = C (N, g, M_{\infty}, m, \Sigma^{1}, \ldots, \Sigma^{J}, \delta, R) < + \infty$.
We split the first term on the right hand side of (\ref{eqn: difference of B k f k to B k F k}) into separate domains,
\begin{equation}
    \left| \int_{co (M_{k})} |\nabla h_{k}|^{2} - |\nabla H_{k}|^{2} \right| \leq \rom{1} + \rom{2} + \sum_{y \in \mathcal{I}} \rom{3}^{y} + \sum_{j = 1}^{J} (\rom{4}^{j} + \rom{5}^{j}),
\end{equation}
where, 
\begin{equation*}
    \rom{1} = \int_{co (M_{k}) \setminus \cup_{y \in \mathcal{I}} \iota_{k}^{-1} (B_{r_{0}}^{N} (y))} |\nabla h_{k}|^{2}, 
\end{equation*}

\begin{eqnarray*}
    \rom{2} &\leq& \int_{\cup_{y \in \mathcal{I}} \iota_{k}^{-1} (B^{N}_{r_{0}} (y) \setminus B^{N}_{r_{0}/2} (y)) } |\nabla h_{k}|^{2} + 12 r_{0}^{-1} |h_{k}| \, |\nabla h_{k}| + 36 r_{0}^{-2} h_{k}^{2}, \\
    &\leq& C \int_{\cup_{y \in \mathcal{I}} \iota_{k}^{-1} (B^{N}_{r_{0}} (y) \setminus B^{N}_{r_{0}/2} (y)) } |\nabla h_{k}|^{2} + r_{0}^{-2} h_{k}^{2}
\end{eqnarray*}

For $y \in \mathcal{I}$, 

\begin{equation*}
    \rom{3}^{y} = \left| \int_{\iota_{k}^{-1} (B^{N}_{r_{0} / 2} (y) \setminus ( \cup_{j = 1}^{J} \Sigma_{k}^{j, 2 R_{0}}))} |\nabla h_{k}|^{2} - |\nabla H_{k}|^{2} \right| = 0,
\end{equation*}

and $j = 1, \ldots, J$, 

\begin{eqnarray*}
    \rom{4}^{j} &\leq& \int_{\iota_{k}^{-1} (\Sigma_{k}^{j, 2 R_{0}} \setminus \Sigma_{k}^{j, R_{0}})} |\nabla h_{k}|^{2} + 6 (R_{0} r_{k}^{j})^{-1} |h_{k}| \, |\nabla h_{k}| + 9 (R_{0} r_{k}^{j})^{-2} h_{k}^{2}, \\
    &\leq& C \left( \int_{\iota_{k}^{-1} (\Sigma_{k}^{j, 2 R_{0}} \setminus \Sigma_{k}^{j, R_{0}})} |\nabla h_{k}|^{2} + (R_{0} r_{k}^{j})^{-2} h_{k}^{2} \right)
\end{eqnarray*}

\begin{equation*}
    \rom{5}^{j} = \int_{\iota_{k}^{-1} (\Sigma_{k}^{j, R_{0}})} |\nabla h_{k}|^{2}.
\end{equation*}

Along our sequence we have (in a weak sense), 
\begin{equation*}
    \Delta h_{k} + (|A_{k}|^{2} + R_{k}) \, h_{k} = - \left(\sum_{i = 1}^{b} a_{i} \lambda_{i, k} f_{i, k} \right) \omega_{k, \delta, R} = - P_{k} \, \omega_{k, \delta, R}, 
\end{equation*}
and we may note that, 
\begin{equation*}
    \|h_{k}\|_{L^{2} (co (M_{k}))} + \|P_{k}\|_{L^{2} (co (M_{k}))} \leq C \sum_{i = 1}^{b} \|f_{i, k}\|_{L^{2} (co (M_{k}))} \leq C,
\end{equation*}
with $C = C (N, g, M_{\infty}, m, \Sigma^{1}, \ldots, \Sigma^{J}, \delta, R) < + \infty$.
Recalling notation from Section \ref{sec: convergence on the base}, for each $l = 1, \ldots, m$, we denote, 
\begin{equation*}
    \tilde{h}_{k}^{l} \coloneqq \sum_{i = 1}^{b} a_{i} \tilde{f}_{i, k}^{l}, \hspace{0.5cm} \text{and}, \hspace{0.5cm} \tilde{P}_{k}^{l} \coloneqq \sum_{i = 1}^{b} a_{i} \lambda_{i,k} \tilde{f}_{i,k}^{l},
\end{equation*} 
and then by standard interior estimates, 
\begin{equation*}
    \|\tilde{h}_{k}^{l}\|_{W^{2,2} (\Omega_{r_{0} / 4})} \leq \tilde{C}  (\|\tilde{h}_{k}^{l}\|_{L^{2} (\Omega_{r_{0} / 8})} + r_{0}^{-2} \|\tilde{P}_{k}^{l}\|_{L^{2} (\Omega_{r_{0} / 8})}) \leq \tilde{C}, 
\end{equation*}
for $\tilde{C} = \tilde{C} (N, g, M_{\infty}, m, \Sigma^{1}, \ldots, \Sigma^{J}, \delta, R, r_{0}) < + \infty$.
Thus, by our assumption that
\begin{equation*}
    h_{k} \rightarrow ((0, \ldots, 0), 0 \ldots, 0),
\end{equation*}
after potentially taking a subsequence and renumerating we have, 
\begin{equation*}
    \begin{cases}
        \tilde{h}_{k}^{l} \rightharpoonup 0, & \text{weakly} \, W^{2,2} (\Omega_{r_{0} / 4}), \\
        \tilde{h}_{k}^{l} \rightarrow 0, & \text{strongly} \, W^{1,2} (\Omega_{r_{0} / 4}).
    \end{cases}
\end{equation*}
By identical arguments, for each $j = 1, \ldots, J$ denoting, 
\begin{equation*}
    \tilde{h}_{k}^{\Sigma^{j}} = \sum_{i = 1}^{b} a_{i} \tilde{f}_{i, k}^{\Sigma^{j}},
\end{equation*}
we have that, after potentially taking a subsequence and renumerating, 
\begin{equation*}
    \begin{cases}
        \tilde{h}_{k}^{\Sigma^{j}} \rightharpoonup 0, & \text{weakly} \, W^{2,2} (\Sigma^{j, 4 R_{0}}), \\
        \tilde{h}_{k}^{\Sigma^{j}} \rightarrow 0, & \text{strongly} \, W^{1,2} (\Sigma^{j, 4 R_{0}}).
    \end{cases}
\end{equation*}
Therefore, we have that for all $\zeta > 0$, and then large enough $k$,
\begin{equation*}
    \begin{split}
    \left| \int_{co (M_{k})} |\nabla h_{k}|^{2} - |\nabla H_{k}|^{2} \right| \leq \hat{C} \biggl( & \sum_{l = 1}^{m} \|\tilde{h}_{k}^{l}\|_{W^{1,2} (\Omega_{r_{0} / 4})}^{2} + \sum_{j = 1}^{J} \|\tilde{h}^{\Sigma^{j}}_{k}\|^{2}_{W^{1,2} (\Sigma^{j, 4 R_{0}})} \biggr) < \zeta,
    \end{split}
\end{equation*}
with $\hat{C} = \hat{C} (N, g, M_{\infty}, m, \Sigma^{1}, \ldots, \Sigma^{J}, \delta, R, r_{0}, R_{0}) < + \infty$.
Similarly we have, 
\begin{equation*}
    \int_{co (M_{k})} \omega_{k, \delta, R} (h_{k}^{2} - H_{k}^{2}) \leq \hat{C} \left( \sum_{l = 1}^{m} \| \tilde{h}_{k} \|_{L^{2} (\Omega_{r_{0} / 4})}^{2} + \sum_{j = 1}^{J} \| \tilde{h}^{\Sigma^{j}}_{k} \|_{L^{2} (\Sigma^{j, 4 R_{0}})}^{2} \right) < \zeta.
\end{equation*}
Thus, for large $k$, there exists a connected component, 
\begin{equation*}
    C_{k} \subset M_{k} \cap (\cup_{y \in \mathcal{I}} B_{r_{0}}^{N} (y) \setminus (\cup_{j = 1}^{J} \Sigma_{k}^{j, R_{0}} ) ),
\end{equation*}
such that, denoting $\tilde{H}_{k} = (H_{k})_{|\iota_{k}^{-1} (C_{k})}$, we have that $\tilde{H}_{k} \in W^{1,2}_{0} (\iota_{k}^{-1} (C_{k}))$, and 
\begin{eqnarray*}
    \int_{\iota_{k}^{-1} (C_{k})} \tilde{H}_{k}^{2} \, \omega_{k, \delta, R} &\geq& \frac{1 - \zeta}{m |\mathcal{I}|}, \\
    B_{k} [\tilde{H}_{k}, \tilde{H}_{k}] &<& \zeta.
\end{eqnarray*}
Thus, choosing 
\begin{equation*}
    \zeta < \frac{\gamma}{2 (2 m |\mathcal{I}| + \gamma)},
\end{equation*}
we derive a contradiction to (\ref{eqn: strict stability of C k}) on $C_{k}$.
\end{proof}

\section{Equivalence of Weighted and Unweighted Eigenspaces}\label{app: equivalence of weighted and unweighted Espaces}

\begin{proposition}\label{prop: equivalence of weighted and unweighted eigenspaces}
    Let $(M, g)$ be a compact Riemannian manifold of dimension $n \geq 3$, with $V \in L^{\infty} (M)$, $\omega \in L(n/2, \infty) (M, g)$, and $\essinf \, \omega > 0$. 
    Define the elliptic operator, 
    \begin{equation*}
        L \coloneqq \Delta + V.
    \end{equation*}
    Then, 
    \begin{eqnarray}
        \label{eqn: span of eigenvectors is a direct sum} \text{span} \, \left\{ \cup_{\lambda \leq 0} \, \mathcal{E} (\lambda; L, W^{1,2} (M)) \right\} &=& \oplus_{\lambda \leq 0} \, \mathcal{E} (\lambda; L, W^{1,2} (M)), \\
        \label{eqn: span of weighted eigenvectors is a direct sum} \text{span} \, \left\{ \cup_{\lambda \leq 0} \, \mathcal{E}_{\omega} (\lambda; L, W^{1,2} (M)) \right\} &=& \oplus_{\lambda \leq 0} \, \mathcal{E}_{\omega} (\lambda; L, W^{1,2} (M)),
    \end{eqnarray}
    and
    \begin{equation*}
        \text{dim} \, \left( \text{span} \, \left\{ \cup_{\lambda \leq 0} \, \mathcal{E}_{\omega} (\lambda; L, W^{1,2} (M)) \right\} \right) = \text{dim} \, \left( \text{span} \, \left\{ \cup_{\lambda \leq 0} \, \mathcal{E} (\lambda; L, W^{1,2} (M)) \right\} \right).
    \end{equation*}
\end{proposition}

\begin{proof}
    First we note by Proposition \ref{prop: W 1 2 bounds on weighted L2 inner product}, and fact that $\essinf \omega > 0$, 
    \begin{equation*}
        \langle f_{1}, f_{2} \rangle_{\omega} \coloneqq \int_{M} f_{1} \, f_{2} \, \omega,
    \end{equation*}
    is a well-defined inner product on $W^{1,2} (M)$.
    Consider the bilinear form on $f_{1}, \, f_{2} \in W^{1,2} (M)$, corresponding to our elliptic operator $L$, 
    \begin{equation*}
        B_{L} [f_{1}, f_{2}] \coloneqq \int_{M} \nabla f_{1} \cdot \nabla f_{2} - V f_{1} f_{2}.
    \end{equation*}
    If $f_{i} \in \mathcal{E}_{\omega} (\lambda_{i}; L, W^{1,2} (M))$, for $i = 1, \, 2$, with $\lambda_{1} \not= \lambda_{2}$, we have, 
    \begin{equation*}
        \lambda_{1} \langle f_{1}, f_{2} \rangle_{\omega} = B_{L} [f_{1}, f_{2}] = \lambda_{2} \langle f_{1}, f_{2} \rangle_{\omega}.
    \end{equation*}
    Thus, $\langle f_{1}, f_{2} \rangle_{\omega} = 0$, and so (\ref{eqn: span of weighted eigenvectors is a direct sum}) follows.
    An identical argument will also conclude (\ref{eqn: span of eigenvectors is a direct sum}).

    \bigskip 

    We now show that, 
    \begin{eqnarray*}
        \text{dim} \, (\oplus_{\lambda \leq 0} \, \mathcal{E}(\lambda; L, W^{1,2} (M))) &=& \sup \{ \text{dim} \, \Pi \colon \Pi \leq W^{1,2} (M), \text{ is a linear space such that, }  (B_{L})_{| \Pi} \leq 0\} \\ 
        &=& \text{dim} \, (\oplus_{\lambda \leq 0} \, \mathcal{E}_{\omega} (\lambda; L, W^{1,2} (M))).
    \end{eqnarray*}
    Indeed, we begin by showing that, 
    \begin{equation}\label{eqn: ind plus nul of weighted is less than max dim subspace that bi-form is non-pos}
        \begin{split}
        \text{dim} \, (\oplus_{\lambda \leq 0} \, \mathcal{E}_{\omega}(\lambda; L, W^{1,2} (M))) \leq \sup \{ \text{dim} \, \Pi \colon & \Pi \leq W^{1,2} (M), \\
        & \text{ is linear space such that, } (B_{L})_{| \Pi} \leq 0\}.
        \end{split}
    \end{equation}
    If the left-hand side is equal to $0$ then the inequality is trivial.
    Thus, take $b \in \mathbb{Z}_{\geq 1}$, such that
    \begin{equation*}
        b \leq \text{dim} \, (\oplus_{\lambda \leq 0} \, \mathcal{E}_{\omega}(\lambda; L, W^{1,2} (M))),
    \end{equation*}
    then we have $b$, $\omega$-weighted eigenvectors, 
    \begin{equation*}
        \{ f_{1}, \ldots, f_{b} \} \subset W^{1,2} (M),
    \end{equation*}
    with respective non-positive eigenvalues $\{\lambda_{1}, \ldots, \lambda_{b}\}$.
    By the argument at the begining of the proof, we may take the set $\{ f_{1}, \ldots, f_{b} \}$, to be orthonormal with respect to $\langle \, \cdot \, , \, \cdot \, \rangle_{\omega}$.
    Thus, 
    \begin{equation}\label{eqn: Pi equals span of f 1 to f m}
        \Pi = \text{span} \{ f_{1}, \ldots, f_{b} \},
    \end{equation}
    is a $b$-dimensional vector space, and for $a_{1}, \ldots, a_{b} \in \mathbb{R}$, we have, 
    \begin{eqnarray*}
        B_{L} \left[ \sum_{i} a_{i} f_{i}, \sum_{j} a_{j} f_{j} \right] = \sum_{i,j} a_{i} a_{j} \lambda_{i} \langle f_{i}, f_{j} \rangle_{\omega} = \sum_{i} \lambda_{i} a_{i}^{2} \leq 0.
    \end{eqnarray*}
    Thus, (\ref{eqn: ind plus nul of weighted is less than max dim subspace that bi-form is non-pos}) holds. 
    Identical argument shows that same inequality holds for unweighted eigenspaces.

    \bigskip 

    We look to show reverse inequality of (\ref{eqn: ind plus nul of weighted is less than max dim subspace that bi-form is non-pos}). 
    As we know that (\ref{eqn: ind plus nul of weighted is less than max dim subspace that bi-form is non-pos}) holds, if the left-hand side of (\ref{eqn: ind plus nul of weighted is less than max dim subspace that bi-form is non-pos}) is unbounded, then equality holds trivially. 
    Thus, we consider, $b \in \mathbb{Z}_{\geq 0}$,
    \begin{equation*}
        b = \text{dim} \, (\oplus_{\lambda \leq 0} \, \mathcal{E}(\lambda; L, W^{1,2} (M))) < \infty,
    \end{equation*}
    and prove reverse of (\ref{eqn: ind plus nul of weighted is less than max dim subspace that bi-form is non-pos}) by contradiction. 
    Indeed, assume that we have a linear subspace $\tilde{\Pi} \leq W^{1,2} (M)$, of dimension $b + 1$, such that $B_{L}$ is non-positive on $\tilde{\Pi}$.
    Consider the $b$ dimensional linear subspace $\Pi$, identically defined as (\ref{eqn: Pi equals span of f 1 to f m}).
    Note that, 
    \begin{equation*}
        W^{1,2} (M) = \Pi \oplus \Pi^{\perp_{\omega}}.
    \end{equation*}
    The projection map $P_{\Pi} \colon \tilde{\Pi} \rightarrow \Pi$, must have a non-trivial kernel, implying that there exists a $v \in \tilde{\Pi} \cap \Pi^{\perp_{\omega}}$, with $\langle v, v \rangle_{\omega} = 1$. 
    Thus,
    \begin{equation*}
        \tilde{\lambda} = \inf \{ B_{L} [f,f] \colon f \in \Pi^{\perp_{\omega}}, \, \langle f, f \rangle_{\omega} = 1 \} \leq 0.
    \end{equation*}
    Take $\tilde{f}_{k} \in \Pi^{\perp_{\omega}}$, $\langle \tilde{f}_{k}, \tilde{f}_{k} \rangle_{\omega} = 1$, such that 
    \begin{equation*}
        \tilde{\lambda} = \lim_{k \rightarrow \infty} B_{L} [\tilde{f}_{k}, \tilde{f}_{k}].
    \end{equation*}
    Thus, noting that $\essinf \, \omega > 0$, and $|V| < + \infty$, we may deduce a lower bound on
    \begin{equation*}
        \tilde{\lambda} \geq - \frac{\| V \|_{L^{\infty} (M)}}{\essinf \, \omega},
    \end{equation*}
    and uniform $W^{1,2} (M)$ bounds for $\{\tilde{f}_{k}\}$, and conclude that there exists an $\tilde{f} \in W^{1,2} (M)$, such that after potentially taking a subsequence and renumerating,
    \begin{equation*}
        \begin{cases}
            \tilde{f}_{k} \rightharpoonup \tilde{f}, \, in \, W^{1,2} (M), \\
            \tilde{f}_{k} \rightarrow \tilde{f}, \, in \, L^{2} (M).
        \end{cases}
    \end{equation*}
    Denote, 
    \begin{eqnarray*}
        h_{1} &\coloneqq& P_{\Pi} (\tilde{f}), \\
        h_{2} &\coloneqq& \tilde{f} - h_{1} \in \Pi^{\perp_{\omega}}.
    \end{eqnarray*}
    There exists $a_{1}, \ldots, a_{b} \in \mathbb{R}$, such that, 
    \begin{equation*}
        h_{1} = \sum_{i = 1}^{b} a_{i} f_{i}.
    \end{equation*}
    Then, as $\tilde{f}_{k} \in \Pi^{\perp_{\omega}}$,
    \begin{eqnarray*}
        B_{L} [\tilde{f}_{k}, \tilde{f}] = \sum_{i = 1}^{m} a_{i} \lambda_{i} \langle \tilde{f}_{k}, f_{i} \rangle_{\omega} + B_{L} [\tilde{f}_{k}, h_{2}] = B_{L} [\tilde{f}_{k}, h_{2}].
    \end{eqnarray*}
    Thus, by weak convergence,
    \begin{eqnarray*}
        B_{L} [\tilde{f}, \tilde{f}] = \lim_{k \rightarrow \infty} B_{L} [\tilde{f}_{k}, \tilde{f}] = \lim_{k \rightarrow \infty} B_{L} [\tilde{f}_{k}, h_{2}] = B_{L} [h_{1} + h_{2}, h_{2}] &=& \sum_{i = 1}^{m} a_{i} \lambda_{i} \langle f_{i}, h_{2} \rangle_{\omega} + B_{L} [h_{2}, h_{2}] \\ 
        &=& B_{L} [h_{2}, h_{2}].
    \end{eqnarray*}
    After potentially taking a further subsequence and renumerating so that, $\tilde{f}_{k} \rightarrow \tilde{f}$ pointwise a.e., by Fatou's Lemma we have the following, 
    \begin{equation*}
        0 \leq \alpha = \langle h_{2}, h_{2} \rangle_{\omega} \leq \langle h_{2}, h_{2} \rangle_{\omega} + \langle h_{1}, h_{1} \rangle_{\omega} = \langle \tilde{f}, \tilde{f} \rangle_{\omega} \leq \liminf_{k \rightarrow \infty} \langle \tilde{f}_{k}, \tilde{f}_{k} \rangle_{\omega} = 1, 
    \end{equation*}
    and we reduce our argument to three cases. 
    
    \bigskip

    First consider, $\alpha = 0$. 
    Thus, $h_{2} = 0$, and by lower semicontinuity of the $W^{1,2} (M)$-norm under weak convergence, 
    \begin{equation*}
        0 = B_{L} [h_{2}, h_{2}] = B_{L} [\tilde{f}, \tilde{f}] \leq \lim_{k \rightarrow \infty} B_{L} [\tilde{f}_{k}, \tilde{f}_{k}] = \tilde{\lambda} \leq 0.
    \end{equation*}
    Therefore, $\tilde{\lambda} = 0$, and recalling the function $v \in \tilde{\Pi} \cap \Pi^{\perp_{\omega}}$, with $\langle v, v \rangle_{\omega} = 1$, we must have
    \begin{equation*}
        B_{L} [v, v] = \inf \{ B_{L} [f,f] \colon f \in \Pi^{\perp_{\omega}}, \, \langle f, f \rangle_{\omega} = 1 \} = 0. 
    \end{equation*}
    Standard variational arguments then show that $v \in \mathcal{E}_{\omega} (0; L, W^{1,2} (M))$, which contradicts the definition $b$. 

    \bigskip 

    If $\alpha = 1$, then, 
    \begin{equation*}
        B_{L} [h_{2}, h_{2}] = \inf \{ B_{L} [f,f] \colon f \in \Pi^{\perp_{\omega}}, \, \langle f, f \rangle_{\omega} = 1 \} = \tilde{\lambda}.
    \end{equation*}
    Again, standard variational arguments then show that $h_{2} \in \mathcal{E}_{\omega} (\tilde{\lambda}; L, W^{1,2} (M) )$, which contradicts the definition of $b$.

    \bigskip 

    The final case is $0 < \alpha < 1$. 
    Note that if $B_{L} [h_{2}, h_{2}] = 0$, then we may apply a similar argument to that in case $\alpha = 0$. 
    Therefore, we may assume that $B_{L} [h_{2}, h_{2}] < 0$.
    Define, 
    \begin{equation*}
        h \coloneqq \alpha^{-1/2} h_{2}.
    \end{equation*} 
    Thus,  
    \begin{equation*}
        B_{L} [h, h] = \alpha^{-1} B_{L} [h_{2}, h_{2}] < B_{L} [h_{2}, h_{2}] \leq \tilde{\lambda}.
    \end{equation*}
    However, as $h \in \Pi^{\perp_{\omega}}$, and $\langle h, h \rangle_{\omega} = 1$, this is a contradiction. 
\end{proof}

\begin{rem}\label{rem: equivalence of weighted and unweighted index and nullity along M k s}
    For an embedded hypersurface $M \subset N$, the method of proof in Proposition \ref{prop: equivalence of weighted and unweighted eigenspaces} may be applied to show that, 
    \begin{equation*}
        \text{span} \, \left\{ \cup_{\lambda \leq 0} \, \mathcal{E}_{\omega} (\lambda; L, W^{1,2} (co (M))^{-}) \right\} = \oplus_{\lambda \leq 0} \, \mathcal{E}_{\omega} (\lambda; L, W^{1,2} (co(M))^{-} ),
    \end{equation*}
    and,
    \begin{equation*}
        \text{dim} \, (\oplus_{\lambda \leq 0} \, \mathcal{E}_{\omega_{k, \delta, R}} (\lambda; L, W^{1,2} (co (M))^{-} )) = \text{dim} \, (\oplus_{\lambda \leq 0} \, \mathcal{E} (\lambda; L, W^{1,2} (co (M))^{-} )).
    \end{equation*}
    When reapplying this method the two major things to note is that $W^{1,2} (co (M))^{-}$ is a linear space, and that when applying Rellich--Kondrachov, we have that the limit will also lie $W^{1,2} (co (M))^{-}$.
\end{rem}

\bigskip 

Let $(M, g)$ be a complete but not necessarily compact Riemannian manifold. 
Recall the following function space, 
\begin{equation*}
    W_{\omega}^{1,2} (M) \coloneqq \{ f \in L_{\text{loc}}^{1} (M) \colon |\nabla f| \in L^{2} (M), \, \text{and} \,  f^{2} \omega \in L^{1} (M) \}.
\end{equation*} 

\begin{proposition}\label{prop: equiv of espace for non-compact manifold}
    Let $\Sigma$ be a, connected, complete $n$-dimensional manifold, $n \geq 3$, and $\iota: \Sigma \rightarrow \mathbb{R}^{n + 1}$ be a two-sided, proper, minimal immersion, with finite total curvature,
    \begin{equation*}
        \int_{\Sigma} |A_{\Sigma}|^{n} < + \infty,
    \end{equation*}
    and Euclidean volume growth at infinity. 
    Consider a function $\omega \in L^{\infty} (\Sigma)$, such that there exists $\Lambda \in [1, \infty)$, and $R \in (0, \infty)$, such that $\essinf \omega > 0$ on $\Sigma \cap \iota^{-1} (\overline{B_{R}^{n + 1} (0)})$, and 
    \begin{equation*}
        \frac{1}{\Lambda |\iota (x)|^{2}} \leq \omega (x) \leq \frac{\Lambda}{|\iota (x)|^{2}},
    \end{equation*}
    for $x \in \Sigma \setminus \iota^{-1} (B_{R}^{n + 1} (0))$.
    Then
    \begin{equation}\label{eqn: span is equal to direct sum in unbounded manifold}
        \text{span} \, \{ \cup_{\lambda \leq 0} \, \mathcal{E}_{\omega} (\lambda; L_{\Sigma}, W_{\omega}^{1,2} (\Sigma), W^{1,2} (\Sigma)) \} = \oplus_{\lambda \leq 0} \, \mathcal{E}_{\omega} (\lambda; L_{\Sigma}, W_{\omega}^{1,2} (\Sigma), W^{1,2} (\Sigma)),
    \end{equation}
    and,
    \begin{eqnarray*}
        \text{dim} \, (\oplus_{\lambda < 0} \, \mathcal{E}_{\omega} (\lambda; L_{\Sigma}, W_{\omega}^{1,2} (\Sigma), W^{1,2} (\Sigma))) &=& \text{anl-ind} \, (\Sigma) \\ &\coloneqq& \lim_{S \rightarrow \infty} \text{anl-ind}_{\iota^{-1} (B_{S}^{n + 1} (0))} \, (\Sigma) \\
        &\coloneqq& \lim_{S \rightarrow \infty} \text{dim} \, (\oplus_{\lambda < 0} \, \mathcal{E} (\lambda; L_{\Sigma}, W_{0}^{1,2} (\Sigma \cap \iota^{-1} (B_{S}^{n + 1} (0)))) 
    \end{eqnarray*}
\end{proposition}

\begin{rem}
    In the literature, for a two-sided, properly immersed minimal hypersurface $\Sigma$, the analytic index ($\text{anl-ind} \, (\Sigma)$) and analytic nullity ($\text{anl-nul} \, (\Sigma)$), are just referred to as the index and nullity of $\Sigma$. 
    We choose to maintain the terms of analytic index and analytic nullity to keep the notation and definitions consistent throughout the paper.
\end{rem}

\begin{proof}
    Denote the stability operator on $\Sigma$ by $L = L_{\Sigma}$.
    As the immersion is proper, $\essinf \omega > 0$, on compact sets of $\Sigma$, and thus $W_{\omega}^{1,2} (\Sigma) \subset W^{1,2}_{\text{loc}} (\Sigma)$.
    Moreover, as $\Sigma$ has finite total curvature, this implies that,
    \begin{equation*}
        \lim_{S \rightarrow \infty} \int_{\Sigma \cap \iota^{-1} (B_{S}^{n + 1} (0))} |A_{\Sigma}|^{n} \rightarrow 0.
    \end{equation*}
    Thus, by applying the curvature estimate of Proposition \ref{prop: curvature estimate}, we have that there exists an $S_{0} > 0$, such that for $x \in \Sigma \setminus \iota^{-1} (B_{2 S_{0}}^{n + 1} (0))$,
\begin{equation}\label{eqn: immersion omega curvature estimate on ends}
    |A_{\Sigma}|^{2} (x) \leq \frac{1}{(|\iota(x)| - S_{0})^{2}} \leq \frac{4}{|\iota(x)|^{2}}.
\end{equation}
Moreover, choosing $S_{0} \geq R$, we have that for $x \in \Sigma \setminus \iota^{-1} (B_{2 S_{0}}^{n + 1} (0))$, 
\begin{equation*}
    |A_{\Sigma}|^{2} (x) \leq 4 \Lambda \omega.
\end{equation*}
Thus, for $f, \, h \in W_{\omega}^{1,2} (\Sigma)$, 
    \begin{equation*}
        \left| \int_{\Sigma} f \, h \, \omega \right| + \left| \int_{\Sigma} |A_{\Sigma}|^{2} \, f \, h \right| < (1 + 4 \Lambda) \left( \int_{\Sigma} f^{2} \, \omega \right)^{1/2} \left( \int_{\Sigma} h^{2} \, \omega \right)^{1/2} < + \infty.
    \end{equation*}
    Therefore the quantities, $B_{L} [f, h]$, and $\langle f, h \rangle_{\omega}$, are finite and well defined for $f, \, h \in W^{1,2}_{\omega} (\Sigma)$.
    We also note that for $\lambda \in \mathbb{R}$, such that
    \begin{equation*}
        \text{dim} \, (\mathcal{E}_{\omega} (\lambda; L_{\Sigma}, W_{\omega}^{1,2} (\Sigma), W^{1,2} (\Sigma)) ) \not= 0,
    \end{equation*}
    then, similar to (\ref{eqn: lower bound on eigenvalue}) we may deduce $\lambda \geq - C$, with $C = C (\Sigma, \iota, R) < + \infty$.
    
    \bigskip 

    Recall the following function,
    \begin{equation*}
    \begin{cases}
        \chi \in C^{\infty} (\mathbb{R}; [0,1]), \\
        \chi(t) = 1, & t \in (-\infty,1], \\
        \chi(t) = 0, & t \in [2, \infty), \\
        -3 \leq \chi'(t) \leq 0.
    \end{cases}
\end{equation*}
    For large $S > 0$, we then define the following smooth function, 
    \begin{equation*}
    \chi_{S} (x) = \chi \left( \frac{|\iota (x)|}{S} \right).
    \end{equation*}
    For $f \in W^{1,2}_{\omega} (\Sigma)$, we define $f_{S} = f \chi_{S} \in W_{0}^{1,2} (\Sigma \cap \iota^{-1} (B_{2 S}^{n + 1} (0)))$.
    Again using the estimate (\ref{eqn: immersion omega curvature estimate on ends}), we may deduce that for $f, \, h \in W_{\omega}^{1,2} (\Sigma)$, 
    \begin{eqnarray*}
        \lim_{S \rightarrow \infty} B_{L} [f, h_{S}] &=& B_{L} [f, h], \\
        \lim_{S \rightarrow \infty} \langle f, h_{S} \rangle_{\omega} &=& \langle f, h \rangle_{\omega}.
    \end{eqnarray*}
    Thus if $f \in (\mathcal{E}_{\omega} (\lambda; L, W_{\omega}^{1,2} (\Sigma), W^{1,2} (\Sigma)) )$, then in fact, for all $h \in W^{1,2}_{\omega} (\Sigma)$, $B_{L} [f, h] = \lambda \langle f, h \rangle_{\omega}$.
    This allows us to conclude (\ref{eqn: span is equal to direct sum in unbounded manifold}) in an identical way to (\ref{eqn: span of weighted eigenvectors is a direct sum}).

    \bigskip 

    We now proceed similarly to Proposition \ref{prop: equivalence of weighted and unweighted eigenspaces}. 
    First recall that as $\Sigma$ has finite total curvature and Euclidean volume growth at infinity, its (analytic) index is finite \cite[Section 3]{Tysk-finite-curvature-index-minimal-surfaces} (cf. \cite{Li-Yau-SchrodingerEquationandEigenvalueProblem}).
    Thus we may pick an $S_{0} > 0$, such that
    \begin{equation}\label{eqn: Sigma is stable outside B S 0}
    \text{anl-ind}_{\iota^{-1} (B_{S_{0}}^{n + 1} (0))} \, (\Sigma) = \text{anl-ind} \, (\Sigma).
    \end{equation}
    First we show that, 
    \begin{equation}\label{eqn: b leq I non-compact}
        \text{dim} \, (\oplus_{\lambda < 0} \, \mathcal{E}_{\omega} (\lambda; L_{\Sigma}, W_{\omega}^{1,2} (\Sigma), W^{1,2} (\Sigma)) ) \leq \text{anl-ind} \, (\Sigma) = I.
    \end{equation}
    If the left hand side is equal to $0$, then the inequality is trivial, thus assume we have $b \in \mathbb{Z}_{\geq 1}$, such that, 
    \begin{equation*}
        b \leq \text{dim} \, (\oplus_{\lambda < 0} \, \mathcal{E}_{\omega} (\lambda; L_{\Sigma}, W_{\omega}^{1,2} (\Sigma), W^{1,2} (\Sigma)) ).
    \end{equation*}
    Therefore, as in Proposition \ref{prop: equivalence of weighted and unweighted eigenspaces}, we may pick a set of eigenfunctions,
    \begin{equation*}
        \{ f_{1}, \ldots, f_{b}\} \subset \oplus_{\lambda < 0} \, \mathcal{E}_{\omega} (\lambda; L, W_{\omega}^{1,2} (\Sigma), W^{1,2} (\Sigma)),
    \end{equation*}
    which are orthonormal with respect to $\langle \, \cdot, \cdot \, \rangle_{\omega}$.
    The inequality then follows noting that for large enough $S > S_{0}$, 
    \begin{equation*}
        \text{span} \{ (f_{1})_{S}, \ldots, (f_{b})_{S} \} \subset W_{0}^{1,2} (\Sigma \cap B_{2 S}^{n + 1} (0)),
    \end{equation*}
    is a $b$-dimensional subspace on which $B_{L}$ is negative definite.

    \bigskip 

    For the reverse of (\ref{eqn: b leq I non-compact}), we take $I \geq 1$ (note that if $I = 0$ then (\ref{eqn: b leq I non-compact}) implies equality), and an increasing sequence $S_{k} \rightarrow \infty$. 
    For large enough $S_{k}$, as $\Sigma \cap \iota^{-1} (B_{S_{k}}^{n + 1} (0))$ is compact with smooth boundary, by identical arguments to those contained in Proposition \ref{prop: equivalence of weighted and unweighted eigenspaces}, there exists a sequences $\lambda_{1}^{k} \leq \cdots \leq \lambda_{I}^{k} < 0$ and a set, 
    \begin{equation*}
        \{ f_{1}^{k}, \ldots, f_{I}^{k} \} \subset W_{0}^{1,2} (\Sigma \cap \iota^{-1} (B_{S_{k}}^{n + 1} (0)) ),
    \end{equation*}
    which is orthonormal with respect to $\langle \, \cdot, \cdot \, \rangle_{\omega}$, such that for all $\varphi \in W^{1,2}_{0} (\Sigma \cap \iota^{-1} (B_{S_{k}}^{n + 1} (0)))$, 
    \begin{equation*}
        B_{L} [f_{i}^{k}, \varphi] = \lambda_{i}^{k} \langle f_{i}^{k}, \varphi \rangle_{\omega}.
    \end{equation*}
    By standard theory (see \cite[Lemma 3.7]{Fritz-index-paper}), for each $i = 1, \ldots, I$, $\lambda_{i}^{k} \geq \lambda_{i}^{k + 1}$. 
    Then recalling our uniform bound $\lambda_{i}^{k} \geq - C$, for each $i = 1, \ldots, I$, there exists a $\lambda_{i}^{\infty} \in (0, - C]$, such that $\lambda_{i}^{k} \rightarrow \lambda_{i}^{\infty}$.
    By similar arguments contained in Section \ref{sec: convergence on the bubble}, we may deduce that after potentially taking a subsequence and renumerating, for each $i = 1, \ldots, I$, there exists an $f_{i}^{\infty} \in W_{\omega}^{1,2} (\Sigma)$, such that, 
    \begin{equation*}
        \begin{cases}
            f_{i}^{k} \rightharpoonup f_{i}^{\infty}, & W^{2,2}_{\text{loc}} (\Sigma), \\
            f_{i}^{k} \rightarrow f_{i}^{\infty}, & W^{1,2}_{\text{loc}} (\Sigma),
        \end{cases}
    \end{equation*}
    and $f_{i}^{\infty} \in \mathcal{E}_{\omega} (\lambda_{i}^{\infty}; L_{\Sigma}, W_{\omega}^{1,2} (\Sigma), W^{1,2} (\Sigma))$.  
    We now look to show that
    \begin{equation*}
        \text{dim} \,( \text{span} \{ f_{1}^{\infty}, \ldots, f_{I}^{\infty} \}) = I,
    \end{equation*}
    which will complete the proof. 

    \bigskip 

    For a contradiction, assume not. 
    Then there exists $a_{1}, \ldots, a_{I} \in \mathbb{R}$, with $\sum a_{i}^{2} = 1$, such that, 
    \begin{equation*}
        h \coloneqq \sum_{i = 1}^{I} a_{i} \, f_{i}^{\infty} = 0.
    \end{equation*}
    We then define, 
    \begin{equation*}
        h_{k} \coloneqq \sum_{i = 1}^{I} a_{i} \, f_{i}^{k} \in W_{0}^{1,2} (\Sigma \cap \iota^{-1} (B_{S_{k}}^{n + 1} (0)) ),
    \end{equation*}
    and note that $\langle h_{k}, h_{k} \rangle_{\omega} = 1$ for all $k$, and $h_{k} \rightarrow 0$ in $W^{1,2}_{\text{loc}} (\Sigma)$.
    By our choice of $S_{0}$ (\ref{eqn: Sigma is stable outside B S 0}), we have that $\Sigma \setminus \iota^{-1} (\overline{B_{S_{0}}^{n + 1} (0)})$ is stable, however, $\chi_{S_{0}} h_{k} \in W^{1,2}_{0} (\Sigma \setminus \iota^{-1} (\overline{B_{S_{0}}^{n + 1} (0)}))$, and for large enough $k$, 
    \begin{equation*}
        B_{L} [\chi_{S_{0}} h_{k}, \chi_{S_{0}} h_{k}] < 0,
    \end{equation*}
    which clearly contradicts the stability of $\Sigma \setminus \iota^{-1} (\overline{B_{S_{0}}^{n + 1} (0)})$.
\end{proof}

\section{Proof of the Theorem}

Again, for ease of notation we only write the proof of Theorem \ref{thm: main theorem intro} for the case of minimal hypersurfaces, however the identical argument works for the case of $H$-CMC hypersurfaces. 

\bigskip 

We follow similar arguments to those in \cite[Lemma \rom{4}.6]{DGR-morse-index-stability} and \cite[Theorem 1.2]{HL-morse-index-estimates}. 

\bigskip 

If $\limsup_{k \rightarrow \infty} (\text{ind} (M_{k}) + \text{nul} (M_{k})) = 0$, then the conclusion of Theorem \ref{thm: main theorem intro} is trivial.
Suppose for $b \in \mathbb{Z}_{\geq 1}$,
\begin{equation*}
    b \leq \limsup_{k \rightarrow \infty} (\text{ind} (M_{k}) + \text{nul} (M_{k})) = \limsup_{k \rightarrow \infty} \left( \text{dim} \, (\oplus_{\lambda \leq 0} \, \mathcal{E}_{\omega_{k, \delta, R}} (\lambda; L_{k}, W^{1,2} (co (M_{k}))^{-})) \right),
\end{equation*}
where the equality comes from equivalence of considering the weighted and unweighted eigenvalue problems along our sequence (Proposition \ref{prop: equivalence of weighted and unweighted eigenspaces} and Remark \ref{rem: equivalence of weighted and unweighted index and nullity along M k s}). 
After potentially taking a subsequence and renumerating we have that for each $k$, there exists a linear subspace,
\begin{equation*}
    W_{k} \coloneqq \text{span} \{ f_{i,k} \}_{i = 1}^{b} \subset W^{1,2} (co (M_{k}))^{-},
\end{equation*}
where, for each $i = 1, \ldots, b$, there is a $\lambda_{i,k} \leq 0$, such that, 
\begin{equation*}
    f_{i, k} \in \mathcal{E}_{\omega_{k, \delta, R}} (\lambda_{i, k}; L_{k}, W^{1,2} (co(M_{k}))^{-}),
\end{equation*}
and the set $\{f_{i,k}\}_{i = 1}^{b}$, is orthonormal with respect to the $\omega_{k, \delta, R}$-weighted $L^{2}$ inner product, i.e. for $i, \, j = 1, \ldots, b$,
\begin{equation*}
    \int_{M_{k}} f_{i, k} \, f_{j, k} \, \omega_{k, \delta, R} = \delta_{ij}.
\end{equation*}
We may assume this by the argument used to prove (\ref{eqn: span of weighted eigenvectors is a direct sum}).
Thus, as outlined in Section \ref{sec: Strict Stability of the neck}, for each $i = 1, \ldots, b$, after potentially taking a subsequence and renumerating, 
\begin{equation*}
    f_{i, k} \rightarrow ((f_{i, \infty}^{1}, \ldots, f_{i, \infty}^{m}), f_{i, \infty}^{\Sigma^{1}}, \ldots, f_{i, \infty}^{\Sigma^{J}}) \in E_{\infty}
\end{equation*}
where we are defining, 
\begin{eqnarray*}
    E_{\infty} &\coloneqq& \left( \times_{j = 1}^{m} \, \left( \oplus_{\lambda \leq 0} \mathcal{E}_{\omega_{\delta}} (\lambda; L_{\infty}, W^{1,2} (co (M_{\infty}))) \right) \right) \\
    && \times \left( \times_{j = 1}^{J} \, \left( \oplus_{\lambda \leq 0} \mathcal{E}_{\omega_{\Sigma^{j}, R}} (\lambda; L_{\Sigma^{j}}, \dot{W}_{\omega_{\Sigma^{j}, R}}^{1,2} (\Sigma^{j}), W^{1,2} (\Sigma^{j})) \right) \right).
\end{eqnarray*}

\bigskip 

We define the linear map, 
\begin{eqnarray*}
    \Pi_{k} \colon W_{k} &\rightarrow& E_{\infty}, \\
    f_{i, k} &\mapsto& ((f_{i, \infty}^{1}, \ldots, f_{i, \infty}^{m}), f_{i, \infty}^{\Sigma^{1}}, \ldots, f_{i, \infty}^{\Sigma^{J}}).
\end{eqnarray*}
Thus $W_{\infty} \coloneqq \Pi_{k} (W_{k})$ is a linear subspace of $E_{\infty}$.

\bigskip 

Define the integer 
\begin{equation*}
    co (m) = \liminf_{r \rightarrow 0} \liminf_{k \rightarrow \infty} | \{ \text{connected components of } M_{k} \setminus \cup_{y \in \mathcal{I}} B_{r}^{N} (y) \} | \leq m.
\end{equation*}
If $M_{\infty}$ is two-sided, by the graphical convergence on sets, compactly contained away from the finite collection of points $\mathcal{I}$, we have that $co (m) = m$.
If $M_{\infty}$ is one-sided, taking $r$ small enough and $k$ large enough such that, 
\begin{equation*}
    co (m) = | \{ \text{connected components of } M_{k} \setminus \cup_{y \in \mathcal{I}} B_{r}^{N} (y) \} |,
\end{equation*}
we recall notation from Section \ref{sec: convergence on the base}, and we have
\begin{equation*}
    M_{k}^{r} = \bigcup_{l = 1}^{co (m)} M_{k}^{l, r},
\end{equation*}
where each $M_{k}^{l,r}$ is a connected hypersurface. 
Then, $\cup_{j = 1}^{m} \{ (x, u_{k}^{j, r} (x)) \colon x \in \Omega_{r} \}$ is a double cover of $M_{k}^{r}$ with trivial normal bundle, implying that we identify (as in Section \ref{sec: convergence on the base})
\begin{equation*}
    \cup_{j = 1}^{m} \{ (x, u_{k}^{j, r} (x)) \colon x \in \Omega_{r} \} = \cup_{l = 1}^{co (m)} o (M_{k}^{l ,r}).
\end{equation*}
If there is an $l \in \{1, \ldots, co (m)\}$, and $j \not= J \in \{1, \ldots, m\}$, such that for all large enough $k$,
\begin{equation*}
    o (M_{k}^{l ,r}) =  \{ (x, u_{k}^{j, r} (x)) \colon x \in \Omega_{r} \} \cup \{ (x, u_{k}^{J, r} (x)) \colon x \in \Omega_{r} \},
\end{equation*} 
then (depending on the choice of unit normal in Section \ref{sec: convergence on the base}), for each $i = 1, \ldots, m$, either
\begin{equation*}
    f_{i, \infty}^{j} ((x, \nu)) = - f_{i, \infty}^{J} ((x, - \nu)), \, \text{for all} \, (x, \nu) \in co (M_{\infty}) \setminus \iota^{-1} (\mathcal{I}),
\end{equation*}
or, 
\begin{equation*}
    f_{i, \infty}^{j} ((x, \nu)) = f_{\infty}^{J} ((x, - \nu)), \, \text{for all} \, (x, \nu) \in co (M_{\infty}) \setminus \iota^{-1} (\mathcal{I}).
\end{equation*}
Thus we may define an injective map 
\begin{equation*}
    P \colon W_{\infty} \rightarrow F_{\infty},
\end{equation*}
where, 
\begin{eqnarray*}
    F_{\infty} &=& \left( \times_{l = 1}^{co (m)} \left( \oplus_{\lambda \leq 0} \mathcal{E}_{\omega_{\delta}} (\lambda; L_{\infty}, W^{1,2} (co (M_{\infty}))) \right) \right) \\
    && \times \left( \times_{j = 1}^{L} \left( \oplus_{\lambda \leq 0} \mathcal{E}_{\omega^{\Sigma^{j}, R}} (\lambda; L_{\Sigma^{j}}, W_{\omega^{\Sigma^{j}, R}}^{1,2} (\tilde{\Sigma}^{j}), W^{1,2} (\Sigma^{j})) \right) \right)
\end{eqnarray*}

\begin{claim}\label{claim: Pi k is injective}
    $\Pi_{k}$ is injective
\end{claim}

\begin{proof}
    We prove by contradiction. 
    Assume we have an $h_{k}$,
    \begin{equation*}
        h_{k} = \sum_{i = 1}^{b} a_{i} f_{i, k},
    \end{equation*}
    with $\sum_{i = 1}^{b} a_{i}^{2} = 1$, and $\Pi_{k} (h_{k}) = 0$.
    This implies that for all $k$,
    \begin{equation*}
        \int_{co (M_{k})} h_{k}^{2} \, \omega_{k, \delta, R} = 1,
    \end{equation*} 
    and,
    \begin{equation*}
        h_{k} \rightarrow ((0, \ldots, 0), 0, \ldots, 0), 
    \end{equation*}
    which contradicts Claim \ref{claim: we cannot converge to zero}.
\end{proof}

Thus, 
\begin{equation*}
    \text{dim} \, W_{\infty} = \text{dim} \, W_{k} = b,
\end{equation*}
which implies that, $b \leq \text{dim} \, F_{\infty}$.
We conclude Theorem \ref{thm: main theorem intro} by combining the results in Section \ref{app: equivalence of weighted and unweighted Espaces} (Proposition \ref{prop: equivalence of weighted and unweighted eigenspaces} and Proposition \ref{prop: equiv of espace for non-compact manifold}), and noting that
\begin{equation*}
    \mathcal{E}_{\omega_{\Sigma^{j}, R}} (0; L_{\Sigma^{j}}, W_{\omega_{\Sigma^{i}, R}}^{1,2} (\Sigma^{j}), W^{1,2} (\Sigma^{j})) = \mathcal{E}_{\omega^{\Sigma^{j}, R}} (0; L_{\Sigma^{j}}, W_{\omega^{\Sigma^{j}, R}}^{1,2} (\Sigma^{j}), W^{1,2} (\Sigma^{j})).
\end{equation*}
for $\omega_{\Sigma^{j}, R}$ as defined in the statement of Theorem \ref{thm: main theorem intro}, and $\omega^{\Sigma^{j}, R}$ as in Section \ref{sec: convergence on the bubble}. 
Lastly, by standard regularity theory for elliptic PDEs, we note that,
\begin{eqnarray*}
    \text{nul}_{\omega_{\Sigma^{j}, R}} (\Sigma^{j}) &\coloneqq& \text{dim} \, \left( \{ \psi \in C^{\infty} (\Sigma) \cap W^{1,2}_{\omega_{\Sigma^{j}, R}} (\Sigma^{j}) \colon L_{\Sigma^{j}} \psi = 0\} \right), \\
    &=& \text{dim} \, \left( \mathcal{E}_{\omega^{\Sigma^{i}, R}} (0; L_{\Sigma^{i}}, W_{\omega_{\Sigma^{i}, R}}^{1,2} (\Sigma^{j}), W^{1,2} (\Sigma^{j})) \right).
\end{eqnarray*}

\section{Finiteness of the Nullity}

\begin{proposition}\label{prop: finiteness of nullity of Sigma}
    Let $\Sigma$ be a complete, connected, $n$-dimensional manifold, $n \geq 3$, and 
    \begin{equation*}
        \iota \colon \Sigma \rightarrow \mathbb{R}^{n + 1}
    \end{equation*} 
    be a proper, two-sided, minimal immersion, of finite total curvature
    \begin{equation*}
        \int_{\Sigma} |A_{\Sigma}|^{n} < + \infty,
    \end{equation*}
    and Euclidean volume growth at infinity
    \begin{equation*}
        \limsup_{R \rightarrow \infty} \frac{\mathcal{H}^{n} (\iota (\Sigma) \cap B_{R}^{n + 1} (0))}{R^{n}} < + \infty.
    \end{equation*}
    Consider a function $\omega \in L^{\infty} (\Sigma)$, such that there exists an $R > 0$, and $\Lambda \geq 1$, such that $\essinf \omega > 0$ in $\iota^{-1} (B_{R}^{n + 1} (0))$, and for $x \in \Sigma \setminus \iota^{-1} (B_{R}^{n + 1} (0))$, 
    \begin{equation*}
        \frac{1}{\Lambda |\iota(x)|^{2}} \leq \omega \leq \frac{\Lambda}{|\iota(x)|^{2}}.
    \end{equation*}
    Then
    \begin{equation*}
        \text{anl-nul}_{\omega} (\Sigma) \coloneqq \text{dim} \, \{ \psi \in W^{1,2}_{\omega} (\Sigma) \colon L_{\Sigma} \psi = 0\} < + \infty.
    \end{equation*}
\end{proposition}

\begin{proof}
    We assume the statement does not hold and prove by contradiction. 
    There exists a set, $\{\psi_{1}, \psi_{2}, \ldots \} \subset W^{1,2}_{\omega} (\Sigma) \cap C^{\infty} (\Sigma)$, such that for all $k, l \in \mathbb{Z}_{\geq 1}$,
    \begin{equation*}
        \Delta \psi_{k} + |A_{\Sigma}|^{2} \psi_{k} = 0,
    \end{equation*}
    and,
    \begin{equation*}
        \int_{\Sigma} \psi_{k} \, \psi_{l} \, \omega = \delta_{k l}.
    \end{equation*}

    \bigskip 

    \textbf{Claim.}
        For any $\delta > 0$, there exists $k, \, l \in \mathbb{Z}_{\geq 1}$, $k \not= l$, such that 
        \begin{equation*}
            \|\psi_{k} - \psi_{l}\|^{2}_{W^{1,2} (\Sigma \cap \iota^{-1} (B_{2 S}^{n + 1} (0)))} < \delta.
        \end{equation*}

    \begin{proof}
        (of Claim)
        Fix $S > 0$, for all $k \in \mathbb{Z}_{\geq 1}$,
        \begin{equation*}
            \int_{\Sigma \cap \iota^{-1} (B_{3 S}^{n + 1} (0))} \psi_{k}^{2} \leq C (S) < + \infty.
        \end{equation*}
        By standard interior estimates for linear elliptic PDEs we have, 
        \begin{equation*}
            \|\psi_{k}\|_{W^{2,2} (\Sigma \cap \iota^{-1} (B_{2S}^{n + 1} (0))} \leq C (S) < + \infty.
        \end{equation*}
        Therefore, there exists a subsequence $\{\psi_{k'}\} \subset \{\psi_{k}\}$, and a function $\psi_{\infty} \in W^{2,2} (\Sigma \cap \iota^{-1} (B_{2 S}^{n + 1} (0)))$ such that, 
        \begin{equation*}
            \begin{cases}
                \psi_{k'} \rightharpoonup \psi_{\infty}, & \text{weakly in} \, W^{2,2} (\Sigma \cap \iota^{-1} (B_{2 S}^{n + 1} (0))), \\
                \psi_{k'} \rightarrow \psi_{\infty}, & \text{strongly in} \, W^{1,2} (\Sigma \cap \iota^{-1} (B_{2 S}^{n + 1} (0))).
            \end{cases}
        \end{equation*}
        Thus this subsequence is Cauchy, so for any $\delta > 0$, there exists $l, \, k \in \mathbb{Z}_{\geq 1}$, $l \not= k$, such that 
        \begin{equation*}
            \|\psi_{k} - \psi_{l}\|^{2}_{W^{1,2} (\Sigma \cap \iota^{-1} (B_{2 S}^{n + 1} (0)))} < \delta
        \end{equation*}
    \end{proof}

    For $\delta > 0$ fixed we denote, 
    \begin{equation*}
        \psi_{\delta} = \psi_{l} - \psi_{k} \in W^{1,2}_{\omega} (\Sigma),
    \end{equation*}
    and note that, $\Delta \psi_{\delta} + |A_{\Sigma}|^{2} \psi_{\delta} = 0$, and, 
    \begin{equation*}
        \int_{\Sigma} \psi_{\delta}^{2} \, \omega = 2.
    \end{equation*}
    Also, $B_{\Sigma} [\psi_{\delta}, \psi_{\delta}] = 0$, which may be seen by following the argument at the begining of Proposition \ref{prop: equiv of espace for non-compact manifold}.

    \bigskip 

    Recall from Remark \ref{rem: structure of ends of minimal immersions of finite total curvature}, that $\Sigma$ has finitely many ends (say $m \in \mathbb{Z}_{\geq 1}$), which, for large enough $S \geq R$, may be denoted by, 
    \begin{equation*}
        \sqcup_{i = 1}^{m} E^{i} = \Sigma \setminus \iota^{-1} (B_{S}^{n + 1} (0)).
    \end{equation*} 
    Moreover, each end $E^{i}$ is graphical over some hyperplane minus a compact set $B_{i}$ (with the graphing function having small gradient), and for each $\varepsilon > 0$, we may further choose $S = S (\Sigma, \iota, R, \Lambda, \varepsilon) < + \infty$, such that, 
    \begin{equation*}
        |A_{\Sigma}|^{2} (x) \leq \varepsilon \omega,
    \end{equation*}
    for $x \in \Sigma \setminus \iota^{-1} (\overline{B}_{S}^{n + 1} (0))$.
    We also remark that $\omega \in L (n / 2, \infty) (\Sigma)$ (this can be shown similarly to Claim \ref{claim: uniform L n over 2 infinity bounf on weight}).

    \bigskip 

    Fixing a choice of $S = S (\Sigma, \iota, \Lambda, R, \varepsilon) < + \infty$, for any $\zeta > 0$, we may pick $\delta = \delta (S, \zeta) > 0$, then $T = T (\delta) > 4 S$, such that 
    \begin{equation*}
        \| \psi_{\delta} \|_{W^{1,2} \left( \Sigma \cap \iota^{-1} (B_{2 S}^{n + 1} (0)) \right)} < \zeta, \hspace{0.5cm} \text{and,} \hspace{0.5cm} \| \nabla \psi_{\delta} \|_{L^{2} \left( \Sigma \setminus \iota^{-1} (B_{T}^{n + 1} (0)) \right)} + \left( \int_{\Sigma \setminus \iota^{-1} (B_{T}^{n + 1} (0))} \psi_{\delta}^{2} \, \omega \right)^{1/2} < \zeta.
    \end{equation*}
    Recalling definition of functions $\chi_{S}$ and $\chi_{T}$ from Proposition \ref{prop: equiv of espace for non-compact manifold}, we define 
    \begin{equation*}
        \Psi_{\delta} = \chi_{T} (1 - \chi_{S}) \psi_{\delta} \in W^{1,2}_{0} (\iota^{-1} (B_{2T}^{n + 1} (0) \setminus \overline{B_{S}^{n + 1} (0)})).
    \end{equation*}
    Performing similar computations to those in Claim \ref{claim: we cannot converge to zero} we deduce, 
    \begin{equation*}
        \int_{\Sigma} (\psi_{\delta}^{2} - \Psi_{\delta}^{2}) \omega \leq C \left( \int_{\Sigma \cap \iota^{-1} (B_{2 S}^{n + 1} (0)) } \psi_{\delta}^{2} + \int_{\Sigma \setminus \iota^{-1} (B_{T}^{n + 1} (0))} \psi_{\delta}^{2} \, \omega \right) < C \zeta, 
    \end{equation*}
    and, 
    \begin{eqnarray*}
        |B_{L} [\psi_{\delta}, \psi_{\delta}] - B_{L} [\Psi_{\delta}, \Psi_{\delta}]| &\leq& C \biggl( \| \psi_{\delta} \|_{W^{1,2} \left( \Sigma \cap \iota^{-1} (B_{2 S}^{n + 1} (0) ) \right)}^{2} \\
        && \hspace{1cm} + \| \nabla \psi_{\delta} \|_{L^{2} \left( \Sigma \setminus \iota^{-1} (B_{T}^{n + 1} (0) ) \right)}^{2} + \int_{\Sigma \setminus \iota^{-1} (B_{T}^{n + 1} (0) )} \psi_{\delta}^{2} \, \omega \biggr) \\ 
        &<& C \zeta,
    \end{eqnarray*}
    with $C = C (\Sigma, \iota, \omega, R) < + \infty$.
    Now, choosing small enough $\varepsilon = \varepsilon (\omega) > 0$, large enough $S = S (\Sigma, \iota, \Lambda, R, \varepsilon)$, then small enough $\zeta = \zeta (\Sigma, \iota, \omega, R, \varepsilon) > 0$, $\delta = \delta (S, \zeta) > 0$, and large enough $T = T (\delta) > 4 S$, $\Psi_{\delta}$ will derive a contradiction to Lemma \ref{lem: Strict stability of the neck} (in a similar fashion to Claim \ref{claim: we cannot converge to zero}) on at least one of the ends $E^{1}, \ldots, E^{m}$. 
\end{proof}

\bigskip 

Corollary \ref{cor: finiteness of index of minimal immersions} may then be concluded by noting that,
\begin{equation*}
    \text{anl-nul} \, (\Sigma) \coloneqq \text{dim} \, \{ \psi \in W^{1,2} (\Sigma) \colon L_{\Sigma} \psi = 0 \} \leq \text{dim} \, \{ \psi \in W^{1,2}_{\omega} (\Sigma) \colon L_{\Sigma} \psi = 0 \} < + \infty.
\end{equation*}

\section{Jacobi Fields on the Higher Dimensional Catenoid}\label{sec: Jacobi fields on the catenoid}

In this section we analyse Jacobi fields on the $n$-dimensional catenoid, $n \geq 3$.
First, we briefly recall the definition of the $n$-dimensional catenoid for $n \geq 3$ (as in \cite[Section 2]{TZ-StabilityHigherDimensionalCatenoid}).
For $h_{0} > 0$, consider the following integral, for $n \geq 3$, 
\begin{equation}\label{eqn: def of s(h)}
    s (h) = \int_{h_{0}}^{h} \frac{d \tau}{(a \tau^{2 (n - 1)} - 1)^{1 / 2}},
\end{equation}
with $a = h_{0}^{- 2 (n - 1)}$.
Then the function $s (h)$ is increasing and maps $[h_{0}, + \infty)$ to $[0, s_{\infty})$, with 
\begin{equation*}
    s_{\infty} = \int_{h_{0}}^{\infty} \frac{d \tau}{(a \tau^{2 (n - 1)} - 1)^{1 / 2}} < \infty.
\end{equation*}
Thus, the inverse of $s$, $h \colon [0, s_{\infty}) \rightarrow [h_{0}, + \infty)$ is well defined, with $h (0) = h_{0}$, and $h' (0) = 0$. 
We then smoothly extend $h$ as an even function across $(- s_{\infty}, s_{\infty})$.
Now, letting $S^{n - 1}$ denote the unit sphere in $\mathbb{R}^{n}$, we define the catenoid, $\mathcal{C}$, in $\mathbb{R}^{n + 1}$, $n \geq 3$, by the embedding, 
\begin{equation*}
\begin{split}
    F \colon (-s_{\infty}, s_{\infty}) \times S^{n - 1} &\rightarrow \mathbb{R}^{n + 1}, \\
    (s, w) &\mapsto (h (s) w, s).
\end{split}
\end{equation*}
For a point $y = F (s, w)$, the unit normal to $\mathcal{C}$ at $y$ is given by, 
\begin{equation*}
    \nu (y) = \frac{(w, - h' (s))}{(1 + (h')^{2})^{1 / 2}}.
\end{equation*}
It was shown by Schoen \cite[Theorem 3]{S-UniquenessSymmetryMinimalSurfaces}, up to rotations, tranlations and scalings, the catenoid $\mathcal{C}$, is the unique complete, non-flat minimal hypersurface in $\mathbb{R}^{n + 1}$, with two ends, which is regular at infinity (for a definition of regular at infinity see \cite[pp. 800]{S-UniquenessSymmetryMinimalSurfaces}).
Recall our weight, $\omega_{\mathcal{C}, R}$, which is given by 
\begin{equation*}
    \omega_{\mathcal{C}, R} (s, w) = \begin{cases}
        (h(s_{1})^{2} + s_{1}^{2})^{-1}, & s \in (-s_{1}, s_{1}), \\
        (h(s)^{2} + s^{2})^{-1}, & s \in (-s_{\infty}, s_{1}] \cup [s_{1}, s_{\infty}),
    \end{cases}
\end{equation*}
for $s_{1} \in (0, s_{\infty})$, given by $R^{2} = h (s_{1})^{2} + s_{1}^{2}$.

\bigskip 

We now look at jacobi fields on $\mathcal{C}$, 
\begin{equation*}
    JF_{\mathcal{C}} \coloneqq \{ f \in C^{\infty} (\mathcal{C}) \colon \Delta_{\mathcal{C}} f + |A_{\mathcal{C}}|^{2} f = 0 \}.
\end{equation*}
In particular we will focus on elements of $JF_{\mathcal{C}}$ which are generated through rigid motions in $\mathbb{R}^{n + 1}$ (translations, scalings and rotations of $\mathcal{C}$), and look to see which of them lie in $W^{1,2} (\mathcal{C})$ and $W^{1,2}_{\omega_{\mathcal{C}, R}} (\mathcal{C})$.
We note that it is still an open question whether these Jacobi fields generated through rigid motions account for all the Jacobi fields on $\mathcal{C}$. 

\bigskip 

We will denote the Jacobi fields on $\mathcal{C}$ defined through rigid motions by $RMJF_{\mathcal{C}}$. 
Those arising through translations are generated by the span of $f_{1}, \ldots, f_{n + 1}$, where 
\begin{equation}\label{eqn: translation Jacobi fields of catenoid}
    f_{i} (y) \coloneqq \langle \nu (y), e_{i} \rangle = \begin{cases}
        \langle (w, 0), e_{i} \rangle \, (1 + (h')^{2})^{- 1 / 2}, & i = 1, \ldots, n, \\
        - h' \, (1 + (h')^{2})^{- 1 / 2}, & i = n + 1.
    \end{cases}
\end{equation}
The $1$-dimensional subspace of $RMJF_{\mathcal{C}}$ generated through scaling has a basis element given by, 
\begin{equation}\label{eqn: scaling Jacobi field of Catenoid}
    f_{d} (y) \coloneqq \langle \nu (y), y \rangle = \frac{h - s h'}{(1 + (h')^{2})^{1 / 2}}.
\end{equation}
The last group to consider are those generated through rotations.
Rotations of $\mathbb{R}^{n + 1}$ about the origin are given by the special orthogonal group,
\begin{equation*}
    SO (n + 1) = \{ R \in G L (n + 1) \colon \text{det} \, R = 1, \, R^{-1} = R^{T} \},
\end{equation*}
where $G L (n + 1)$ denotes the general linear group of all $(n + 1) \times (n + 1)$ real matrices.
Then a smooth rotation of $\mathbb{R}^{n + 1}$, is given by a smooth curve, 
\begin{equation*}
    \gamma \colon [0, T] \rightarrow S O (n + 1),
\end{equation*}
with $\gamma (0) = Id$ (the identity).
Then the Jacobi field on $\mathcal{C}$ generated by the $1$-parameter family of catenoids, $\gamma (t) (\mathcal{C})$, is 
\begin{equation*}
    f_{\gamma} (y) = \langle \nu (y), \gamma' (0) y \rangle, 
\end{equation*}
where $\gamma' (0) \in T_{Id} S O (n + 1)$.
One may show that, 
\begin{equation*}
    T_{Id} S O (n + 1) = \{ R \in G L (n + 1) \colon R^{T} = - R \}.
\end{equation*}
Taking $R = (R_{ij})_{ij} \in T_{Id} S O (n + 1)$, we have that, 
\begin{equation*}
    R y = R (h w, s) = \sum_{i = 1}^{n + 1} \left( \sum_{j = 1}^{n} h R_{ij} w_{j} + R_{i, n + 1} s \right) e_{i},
\end{equation*}
and, 
\begin{equation*}
    \langle \nu (y), \gamma' (0) y \rangle = (1 + (h')^{2})^{- 1 / 2} \left( h \sum_{i, j = 1}^{n} R_{ij} w_{i} w_{j} + \sum_{i = 1}^{n} R_{i, n + 1} s w_{i} - h h' \sum_{j = 1}^{n} R_{n + 1, j} w_{j} - R_{n + 1, n + 1} h' s \right).
\end{equation*}
As $R^{T} = - R$, this implies, $R_{ii} = 0$, and 
\begin{equation*}
    \sum_{i, j = 1}^{n} R_{ij} w_{i} w_{j} = 0.
\end{equation*}
Thus, for $y = F (s, w)$ we have
\begin{equation*}
    f_{\gamma} (y) = (1 + (h')^{2})^{- 1 / 2} \left( (s + h h') \sum_{i = 1}^{n} R_{i, n + 1} w_{i} \right)
\end{equation*}
We now look to see which elements of $RMJF_{\mathcal{C}}$ lie in either $W^{1,2} (\mathcal{C})$, or $W^{1,2}_{\omega_{\mathcal{C}, R}} (\mathcal{C})$.
The following tells us that to check if an element of $JF_{\mathcal{C}}$ lies in $W^{1,2} (\mathcal{C})$ (resp. $W^{1,2}_{\omega_{\mathcal{C}, R}} (\mathcal{C})$), we only need to check if it lies in $L^{2} (\mathcal{C})$ (resp. $L^{2}_{\omega_{\mathcal{C}, R}} (\mathcal{C})$).
As previously discussed, there exists a $C < + \infty$, such that $|A_{\mathcal{C}}|^{2} \leq C \omega_{\mathcal{C}, R}$.
Now fix any $S > R > 0$, and recall the function $\chi_{S}$ from Proposition \ref{prop: equiv of espace for non-compact manifold}.
Then we have that, for $f \in JF_{\mathcal{C}} \cap L^{2}_{\omega_{\mathcal{C}, R}} (\mathcal{C})$,
    \begin{eqnarray*}
        \int \chi_{S}^{2} |\nabla f|^{2} &=& \int |A|^{2} \chi_{S}^{2} f^{2} - \int 2 (\chi_{S} \nabla f) \cdot ( f \nabla \chi_{S}), \\
        &\leq& C \int f^{2} \omega_{\mathcal{C}, R} + \frac{1}{2} \int \chi_{S}^{2} |\nabla f|^{2} + 2 \int f^{2} |\nabla \chi_{S}|^{2}.
    \end{eqnarray*}
Thus, 
    \begin{equation*}
        \int_{B_{S}^{n + 1} (0) \cap \mathcal{C}} |\nabla f|^{2} \leq C \int f^{2} \omega_{\mathcal{C}, R}.
    \end{equation*}
As the upper bound is finite and independent of $S$, we have that $|\nabla f| \in L^{2} (\mathcal{C})$.

\bigskip 

Consider the pull back of the Euclidean metric to $(-s_{\infty}, s_{\infty}) \times S^{n - 1}$ by $F$,
\begin{equation*}
    g = (1 + (h')^{2}) \, ds^{2} + h^{2} \, g_{S^{n - 1}},
\end{equation*}
where $g_{S^{n - 1}}$ is the standard round metric on $S^{n - 1}$. 
We have, 
\begin{equation*}
    \sqrt{|g|} = h^{n - 1} \sqrt{1 + (h')^{2}}.
\end{equation*}
We now compute the $L^{2}$-norm of our elements of $RMJF_{\mathcal{C}}$. 

\bigskip 

Starting with the translations (\ref{eqn: translation Jacobi fields of catenoid}), for $i = 1, \ldots, n$,
\begin{eqnarray*}
    \int f_{i}^{2} &=& \int_{- s_{\infty}}^{s_{\infty}} \int_{S^{n - 1}} f_{i}^{2} \, \sqrt{|g|} \, dw \, ds, \\
    &=& \left( \int_{- s_{\infty}}^{s_{\infty}} \frac{h^{n - 1}}{(1 + (h')^{2})^{1/2}} \, ds \right) \left( \int_{S^{n - 1}} w_{i} \, dw \right).
\end{eqnarray*}
Differentiating (\ref{eqn: def of s(h)}), we have that 
\begin{equation}\label{eqn: relation between h and h'}
    h' = (a h^{2(n - 1)} - 1)^{1/2},
\end{equation}
which implies that, 
\begin{equation*}
    \int f_{i}^{2} = \frac{2 s_{\infty}}{a^{1/2}} \int_{S^{n - 1}} w_{i} \, dw < + \infty.
\end{equation*}
Thus we have that $f_{i} \in W^{1,2} (\mathcal{C}) \subset W^{1,2}_{\omega_{\mathcal{C}, R}} (\mathcal{C})$.
For $f_{n + 1}$, we have that 
\begin{equation*}
    |f_{n + 1}| (s) \rightarrow 1,
\end{equation*}
as $|s| \rightarrow s_{\infty}$.
Thus as $\omega_{\mathcal{C}, R} \not\in L^{1} (\mathcal{C})$, this implies that $f_{n + 1} \not\in L^{2}_{\omega_{\mathcal{C}, R}} (\mathcal{C})$, and thus $f_{n + 1} \not\in W^{1,2}_{\omega_{\mathcal{C}, R}} (\mathcal{C}) \subset W^{1,2} (\mathcal{C})$.

\bigskip 

For the Jacobi field $f_{d}$, generated by scaling, we have, 
\begin{eqnarray*}
    |f_{d}| &=& \frac{|h - s h'|}{(1 + (h')^{2})^{1/2}} = \frac{|h (h')^{-1} - s|}{((h')^{-2} + 1)^{1/2}}.
\end{eqnarray*}
Thus as $|s| \rightarrow s_{\infty}$, we have that $|h'| \rightarrow \infty$, $h \rightarrow \infty$, and, recalling (\ref{eqn: relation between h and h'})
\begin{equation}
    \frac{h}{|h'|} = \frac{h}{(a h^{2 (n - 1)} - 1)^{1/2}} \rightarrow 0,
\end{equation}
which implies that $|f_{d}| \rightarrow s_{\infty}$.
Thus, similar to above, $f_{d}$ does not lie in $W^{1,2} (\mathcal{C})$ or $W^{1,2}_{\omega_{\mathcal{C}, R}} (\mathcal{C})$.

\bigskip 

Finally we look at elements of $RMJF_{\mathcal{C}}$ which are generated by rotations. 
Recall from above that all such Jacobi fields are of the form 
\begin{equation*}
    f (s, w) = (1 + (h')^{2})^{- 1 / 2} \left( (s + hh') \sum_{i = 1}^{n} R_{i, n + 1} w_{i} \right).
\end{equation*}
If $R_{i, n + 1} = 0$ for all $i = 1, \ldots, n$, we have that $f = 0$. 
Assuming that $R_{i, n + 1} \not= 0$ for some $i = 1, \ldots, n$, then as, 
\begin{equation*}
    \sum_{i = 1}^{n} R_{i, n + 1} w_{i} 
\end{equation*}
is a smooth function on $S^{n - 1}$, and is non-zero at, 
\begin{equation*}
    w_{f} = \left( \sum_{i = 1}^{n} R_{i, n + 1}^{2} \right)^{-1} (R_{1, n + 1}, \ldots, R_{n, n + 1}) \in S^{n - 1},
\end{equation*}
we may conclude that there exists an $\alpha > 0$, and set $U_{\alpha} \subset S^{n - 1}$ of positive $\mathcal{H}^{n - 1}$ measure, such that 
\begin{equation*}
    \left| \sum_{i = 1}^{n} R_{i, n + 1} w_{i} \right| \geq \alpha.
\end{equation*}
Thus for $w \in U_{\alpha}$, we have that, 
\begin{equation*}
    |f|^{2} \geq \alpha^{2} \frac{(s + hh')^{2}}{1 + (h')^{2}} = \alpha^{2} \frac{(s (h')^{-1} + h)^{2}}{(h')^{-2} + 1},
\end{equation*}
implying that $|f|^{2}$ becomes unbounded as $|s| \rightarrow s_{\infty}$, on a set of unbounded measure. 
Thus, $f$ does not lie in $W^{1,2} (\mathcal{C})$, or $W^{1,2}_{\omega_{\mathcal{C}, R}} (\mathcal{C})$.

\bigskip 

Therefore, we have shown that the only non-trivial Jacobi fields on $\mathcal{C}$, which are generated through rigid motions and lie in $W^{1,2} (\mathcal{C})$ or $W^{1,2}_{\omega_{\mathcal{C}, R}} ( \mathcal{C})$ are those spanned by the translations in the $\{ x_{n + 1} = 0\}$ hyperplane, i.e. Jacobi fields in $\text{span} \, (\{ f_{1}, \ldots, f_{n} \})$. 
This implies that, 
\begin{equation*}
    n \leq \text{nul} \, (\mathcal{C}) \leq \text{nul}_{\omega_{\mathcal{C}, R}} (\mathcal{C}).
\end{equation*}

\end{document}